\documentclass[a4paper, 12pt, reqno]{amsart}
\makeatletter
\def\l@subsection{\@tocline{2}{0pt}{2.5pc}{5pc}{}}
\def\l@subsubsection{\@tocline{2}{0pt}{5pc}{7.5pc}{}}

\usepackage{mathptmx}
\DeclareSymbolFont{letters}{OML}{cmm}{m}{it}
\DeclareSymbolFont{symbols}{OMS}{cmsy}{m}{n}
\DeclareSymbolFontAlphabet{\mathcal}{symbols}

\usepackage{mathrsfs}
\usepackage{amsbsy,latexsym}
\usepackage{frcursive}
\usepackage{yfonts}
\usepackage{graphicx}
\usepackage{empheq}
\usepackage{mathtools}
\usepackage{tikz}
\usepackage{amsbsy}
\usepackage{avant}
\usepackage{amsmath}
\usepackage{amsthm}
\usepackage{amssymb}
\usepackage{mathrsfs}
\usepackage{amsfonts}
\usepackage{newlfont}
\usepackage[sans]{dsfont}
\usepackage{dcolumn}
\usepackage{bm}
\usepackage{bbm}
\usepackage{stmaryrd}
\usepackage{bbm}
\usepackage{relsize}
\usepackage{euscript}
\usepackage{eufrak}
\usepackage{enumerate}
\usepackage{amsbsy}
\usepackage{fullpage}
\usepackage{mathrsfs}
\usepackage{amsbsy,latexsym}
\usepackage{frcursive}
\usepackage{yfonts}
\usepackage{graphicx}
\usepackage{empheq}
\usepackage{mathtools}
\usepackage{tikz}
\usepackage{amsbsy}
\usepackage{avant}
\usepackage{amsmath}
\usepackage{amsthm}
\usepackage{amssymb}
\usepackage{mathrsfs}
\usepackage{amsfonts}
\usepackage{newlfont}
\usepackage[sans]{dsfont}
\usepackage{dcolumn}
\usepackage{bm}
\usepackage{bbm}
\usepackage{stmaryrd}
\usepackage{bbm}
\usepackage{relsize}
\usepackage{euscript}
\usepackage{eufrak}
\usepackage{enumerate}
\usepackage{amsbsy}
\usepackage{fullpage}

\numberwithin{equation}{section}
\newtheorem{thm}{Theorem}[section]

\newtheorem{cor}[thm]{Corollary}
\newtheorem{lem}[thm]{Lemma}
\newtheorem{prop}[thm]{Proposition}
\newtheorem{defn}[thm]{Definition}
\newtheorem{rem}[thm]{Remark}

\begin{document}
\allowdisplaybreaks{
\title[]{A SPECTRAL REPRESENTATION OF A WEIGHTED RANDOM VECTORIAL FIELD: POTENTIAL APPLICATIONS TO HYDRODYNAMICAL TURBULENCE AND THE PROBLEM OF ANOMALOUS DISSIPATION IN THE INVISCID LIMIT}
\author{Steven D Miller}\email{stevendm@ed-alumni.net}
\maketitle
\begin{abstract}
Let $\mathfrak{G}\subset\mathbb{R}^{3}$ with $vol(\mathfrak{G})\sim L^{3}$. Let ${\mathscr{T}}(x)$ be a Gaussian random field $\forall~x\in\mathfrak{G}$ with expectation $\mathbf{E}[{\mathscr{T}}(x)]=0$ and correlation $\mathbf{E}[{\mathscr{T}}(x)\otimes{\mathscr{T}}(y)]=K(x,y;\lambda)$, an isotropic and regulated kernel with correlation length $\lambda$. The field has a convergent Karhunen-Loeve spectral representation
$\bm{\mathscr{T}}(x)=\sum_{I=1}^{\infty}\mathrm{Z}^{1/2}_{I}f_{I}(x)\otimes\mathscr{Z}_{I}$, with eigenvalues $\lbrace\mathrm{Z}_{I}\rbrace$, eigenfunctions $\lbrace f_{I}(x)\rbrace $ with standard Gaussian random variables $\mathscr{Z}_{I}$ where ${\mathbf{E}}[\mathscr{Z}_{I}]=0$ and $\mathbf{E}[\mathscr{Z}_{I}\otimes\mathscr{Z}_{J}]=\delta_{IJ}$. If $\mathfrak{G}$ contains incompressible fluid of viscosity $\nu$ with velocity $u_{a}(x,t)$ that evolves via the Navier-Stokes equations with a high 'Reynolds function'$\mathsf{RE}(x,t)=\tfrac{\|u_{a}(x,t)\|L}{\nu} $ then aspects of a turbulent fluid flow within $\mathfrak{G}$ with $\mathsf{RE}(x,t)\gg \mathsf{RE}_{*}$, a critical Reynolds number, might be represented by the random vectorial field
\begin{align}
\mathscr{U}_{a}(x,t)=u_{a}(x,t)+{A}u_{a}(x,t)\left(\mathsf{RE}(x,t)-\mathsf{RE}_{*}\right)^{\beta}\sum_{I=1}^{\infty}
\mathrm{Z}^{1/2}_{I}f_{I}(x)\otimes{\mathscr{Z}}_{I}\nonumber
\end{align}
whose fluctuations and amplitude scale nonlinearly with $\mathsf{RE}(x,t)$, with mean $\mathbf{E}[{\mathscr{U}}_{a}(x,t)]
=u_{a}(x,t)$. In the inviscid limit one can prove an anomalous dissipation-type law
\begin{align}
\lim_{\nu\uparrow 0}\left(\lim_{{U}_{a}(x,t)\uparrow {U}_{a}}\sup~\nu \int_{\mathfrak{G}}\int_{0}^{T}
{\mathbf{E}}\bigg[\bigg|\mathlarger{\nabla}_{a}{\mathscr{U}}_{a}(x,s)\bigg|^{2}\bigg]d\mathcal{V}(x) ds\right)>0\nonumber
\end{align}
iff $\beta=\tfrac{1}{2}$ and $\sum_{I=1}^{\infty}\mathrm{Z}_{I}\int_{{\mathfrak{G}}}\mathlarger{\nabla}_{a}f_{I}(x)\mathlarger{\nabla}_{a}f_{I}(x)d\mathcal{V}(x)$ positive
and finite.
\end{abstract}
\tableofcontents
\raggedbottom
\maketitle
\clearpage
\textit{"Something attempted, something done, has earned a night's repose." Henry Wadsworth Longfellow.}

\section{INTRODUCTION:~ANOMALOUS DISSIPATION AND THE 'ZEROTH LAW' OF TURBULENCE}
Hydrodynamic turbulence arguably remains one of the greatest technical challenges at the intersection of physics and mathematics. During the past century, understanding of this phenomenon was much enriched by the works of Taylor, Prandtl, von Karman, Richardson, Heisenberg, Onsager, Kraichnanand of course Kolmogorov. These theories have been successful in modeling many aspects of the statistics of turbulent flows and can connect successfully with experiment. There is by now a vast literature on fluid mechanics and turbulence in both the physics and mathematics literature; for example \textbf{[1-58]}. Nevertheless, to date no single mathematically rigorous connection or bridge exists between the incompressible Navier-Stokes equations at high Reynolds number and these phenomenological turbulence theories. The mathematical problem within the context of the Navier–Stokes equations is delicate and a rigorous understanding is still in it's infancy. The behavior of very turbulent fluids--that is, fluids at very high Reynolds number--can also be viewed as a very broad, and mostly open, important issue within the theory of nonlinear PDE and nonlinear physics.

The fundamental ansatz of Kolmogorov’s 1941 theory of fully developed turbulence~\textbf{[1-5]}, often called the 'zeroth law' of turbulence, postulates the anomalous dissipation of energy--the non-vanishing of the rate of dissipation of kinetic energy $\epsilon$ of turbulent fluctuations per unit mass--in the limit of zero viscosity. Both the law of finite energy dissipation and the 2/3 law are fundamental within the theory of turbulent fluid mechanics, and have been verified to a large degree experimentally, but not rigorously mathematically. On the mathematical side, the problem in the context of Navier–Stokes equations is very delicate and a rigorous understanding of these predictions is still very much in it's infancy. The law of finite energy dissipation technically states that for turbulent flows, the energy dissipation rate $\epsilon$ is finite and non zero in the inviscid limit $\nu\rightarrow 0$ whereby the viscosity vanishes. If $u(x,t)$ is the fluid velocity described by the Navier-Stokes equations within a domain $\mathfrak{G}\subset\mathbb{R}^{3}$, then anomalous dissipation can be expressed mathematically as
\begin{align}
\lim_{\nu\rightarrow 0}\epsilon(x,t)=\lim_{\nu\rightarrow 0}\nu \left\langle |\mathlarger{\nabla} u(x,t)|^{2}\right\rangle=\epsilon > 0
\end{align}
where $\langle\bullet\rangle$ is an ensemble average or time average. This can also be stated as
\begin{align}
\lim_{\nu\rightarrow 0}\sup~\nu\int_{0}^{T}\left\langle \left\|\mathlarger{\nabla} u\right\|^{2}_{L_{2}({\mathfrak{G}})}\right\rangle ds =\lim_{\nu\rightarrow 0}\sup~\nu \int_{\mathfrak{G}}\int_{0}^{T}\left\langle|\mathlarger{\nabla} u(x,t)|^{2}\right\rangle d\mathcal{V}(x) ds>0
\end{align}
Onsanger \textbf{[13,14]}interpreted this law in terms of H$\ddot{o}$lder continuity and smoothness of the velocity field such that the law holds for a H$\ddot{o}$lder exponent $\le 1/3$. The most obvious requirement for such a non-vanishing limit of dissipation is that space-gradients of velocity must diverge, that is $\nabla u(x,t)\rightarrow\infty$, as $\nu$ → 0. This is a short-distance ultraviolet (UV) divergence in the language of quantum field-theory, or what Onsager himself termed a “violet catastrophe”. The inviscid limit for turbulent fluids is then analogous to a “continuum” or “critical” limit in quantum field-theory, where a scale-invariant regime is expected and similar UV divergences are encountered. Since the NS fluid equations of motion contain diverging gradients, they will become ill-defined in the limit.

For flows on domains without boundary (such as $\mathbb{T}^{n}$ or $\mathbb{R}^{n}$), Onsager’s assertion has since been established \textbf{[59-63]}. Although there is a wealth of evidence from numerical simulations for the zeroth law or anomalous dissipation for flows on the torus \textbf{[64-66]}, all empirical evidence from laboratory experiments involve flows confined by solid boundaries. The most common experiments study turbulent flows produced downstream of wire-mesh grids or generated by flows past solid obstacles, such as plates or cylinders. Such experiments indeed indicate the presence of a dissipation anomaly; the data plotted in \textbf{[67, 68]} for example, clearly shows that $\epsilon$  is nearly independent of $\nu = 1/\mathsf{RE}$ as Reynolds number $\mathsf{RE}$ increases. The zeroth law of turbulence is therefore verified experimentally, numerically and in simulations to a very high degree. Anomalous dissipation remains fundamental to our modern understanding of turbulence.(See also\textbf{[69,70]}.)

However, this phenomenon is extremely difficult to mathematically pin down, and very few rigorous results are known. Recently, in \textbf{[71]}, Drivas et. al.  investigated anomalous scalar dissipation for fluid velocities with Holder regularity and gave sufficient conditions for anomalous dissipation in terms of the mixing rates of the advecting flow. Bru`e and De Lellis\textbf{[72]} established the anomalous dissipation for the forced Navier-Stokes equations while Jeong and Yoneda \textbf{[73,74]} considered the anomalous dissipation for the 3D NS  equations with zero external force in the framework of 2 + 1/2-dimensional flows. These results are all obtained on a torus. \emph{There are no known anomalous dissipation results for (NS) in the whole space $\mathbb{R}^{3}$ (or even in $\mathbb{R}^{2}$) where it is rigorously proved}

The transition from smooth or laminar flow to turbulence is essentially characterised by the growth and values of the Reynolds number \textbf{[27]}
\begin{align}
\mathsf{RE}=\frac{u L}{\nu}
\end{align}
where $L\sim {vol}({\mathfrak{G}}))^{1/3}$ and $x\in\mathfrak{G}$. This can be interpreted as a fundamental 'control parameter'. At very high, but not infinite, Reynolds number $\mathsf{RE}\gg 0 $ all of the small-scale statistical properties are assumed to be uniquely and universally determined by the length scale ${\ell}$ and the mean dissipation rate (per unit mass)$\epsilon$. Despite its conjectural status from the perspective of mathematical rigour, with some heuristic assumptions on statistical properties (homogeneity, isotropy), Kolmogorov \textbf{[1-5]} made key predictions about the structure of turbulent velocity fields for incompressible viscous fluids at high Reynolds number, namely that for $d=3$, the following 2/3-scaling law for the second order structure function, holds over an inertial range
\begin{align}
\mathrm{S}_{2}[\ell]=\big\langle\big|u_{a}({x}+{\ell},t)-u_{a}({x},t)\big|^{2}\big\rangle
={C}{\epsilon}^{2/3}{\ell}^{2/3}
\end{align}
where ${C}$ is some constant. These should hold in the limit of large Reynolds number and small scales $\ell\sim 0$. In particular, the 4/5-law is an exact result. In Fourier space, the 2/3-law becomes the $5/3-law$ for the energy spectrum. Here, ${\ell}$ is within the so-called \textit{inertial range} of length scales $\eta \le {\ell} \le {L}$. The length $\eta=(\nu^{3/4}\epsilon)^{-1/4}$ known as the Kolmogorov scale, represents a small scale dissipative cutoff or the size of the smallest eddies, and the integral scale L represents the size of the largest eddy in the flow--or the scale at which energy is fed into the fluid--which cannot exceed the dimensions of the domain. At this scale, viscosity dominates and the kinetic energy is dissipated into heat. There is a 'cascade' process whereby energy is transferred from the largest scales/eddies to the Kolmogorov scale. These are well known and well-established facts concerning hydrodynamic turbulence in viscous incompressible fluids.

Theories of fully developed turbulence typically attempt to make statistical predictions about fluid flow dynamics or behavior at high Reynolds numbers, away from solid boundaries, for length scales in the inertial range, and under certain assumptions–ergodicity, statistical homogeneity, isotropy, and self-similarity. Note that typically it is not possible to rigorously prove these assumptions directly from first principles (e.g. from the Navier-Stokes equations), and so certain ambiguities arise. One of these ambiguities lies in the definition of a statistical average $\langle \bullet \rangle$. One assumes statistical equilibrium, and that the probability measure 'encodes' the macroscopic statistics of the turbulent flow. In laboratory experiments a measurement of the turbulent flow is usually a long-time average at fixed viscosity, in order to reach a stationary regime. In analogy with classical statistical mechanics, turbulence theories try to deal with the possible discrepancy between ensemble averages and
statistical averages by making an ergodic hypothesis. The implication of the ergodic hypothesis is that averages against an ergodic invariant measure are the same as long-time averages. But taking averages is very difficult to formalise mathematically in a useful or rigorous way, whether long-time averages or ensemble averages and little is even known about the objects or structure over which the average is to be taken. It is not even known whether a general solution of the Navier-Stokes equations exists for a viscous incompressible fluid, and even if such a solution exists it may develop singularities or blowups. The nature of the randomness inherent within turbulent fluids is still mysterious and difficult to quantify and describe. Such a description may be beyond human ability or even mathematics itself.

It is clear that fluid mechanics and turbulence continues to be very technically challenging, both mathematically and computationally, and especially from the perspective of mathematical rigour. Many of the issues discussed by Von Neumann in his well-known review paper \textbf{[27]}still remain relevant. There remains opportunity (and an ongoing need) to try and apply new and established mathematical tools and methods to the problem of developed turbulence. While traditional methods have provided significant insights, several emerging and original mathematical approaches and new methodologies have much potential to better understand and model turbulence: machine learning, fractional calculus, multiscale models, stochastic methods, topological approaches, stochastic PDE, chaos theory, wavelet analysis, machine learning, quantum computation and advanced computational techniques, and statistical/random geometry and random fields can potentially contribute to deeper understanding and insight. These innovative approaches might pave the way for more accurate models, better predictions, and new insights into the complex behavior of turbulent flows.

However, the approach developed here will be much more modest. A central issue within fully developed turbulence is still how to define and calculate Reynolds stresses, structure functions and velocity correlations. Established methods are mostly heuristic and as discussed, it remains very difficult to rigorously define or mathematically formalise the required spatial, temporal or ensemble averages $\big\langle\bullet\big\rangle$ in a useful manner. Rigorously defining statistical averages in conventional statistical hydrodynamics is fraught with technical and mathematical difficulties and limitations, as well as having a limited scope of physical applicability. However, A key insight of Kolmogorov's (and others') work is that \emph{turbulence can essentially be interpreted as a random field.} The original papers of Kolmogorov \textbf{[1-3]} are foundational and very important within the theory of turbulence even though they are heuristic and short. As one of the main architects and pioneers of stochastic analysis and probability theory in the 20th century, it is somewhat curious that he never returned to this work to reconsider it more rigorously within  the frameworks of the formalisms and tools which he helped develop. In this paper, we tentatively construct a description of generic random stochastic vector fields or 'flows/fluxes' with an Euclidean domain, using established formalisms from stochastic functional analysis.  We first consider a purely mathematical construction of generic spatio-temporal or 'engineered' random vectorial fields or flows/fluxes ${\mathscr{I}}_{a}(x,t)$ within a (large) closed Euclidean domain ${{\mathfrak{G}}}\subset\mathbb{R}^{3}$ with $a=1,2,3$ and volume $vol({\mathfrak{G}})\sim L^{3}$. This model of random or stochastic vectorial fields is then tentatively applied to the problem of hydrodynamical turbulence for fluid flows.

The following notation will used:
\begin{itemize}
\item Coordinates:$\lbrace x^{a}\rbrace_{a=1,2,3}$ or $x=(x^{1},x^{2},x^{3})$ and $y^{b}=(y^{1},y^{2},y^{3})$ for all $(x,y)\in \mathfrak{G}\subset\mathbb{R}^{3}$. For the Kronecker delta $\delta_{ab}\delta^{ab}=3$. For space and time coordinates $(x,t)$ and $(y,t)$ for all $(x,y,t)\in\mathfrak{G}\otimes\mathbb{R}^{+}$ and $(y,t)\in{\mathfrak{G}}\otimes\mathbb{R}^{+}$.
\item Derivatives and integrals:$\mathlarger{\nabla}_{a}=\frac{\partial}{\partial x_{a}}$ and $\mathlarger{\Delta}=\mathlarger{\nabla}_{a}\mathlarger{\nabla}^{a}$ is the Laplacian operator. The volume integral over $\mathfrak{G}$ is $\int_{\mathfrak{G}}\bullet d\mathcal{V}(x)$ and $\int_{\mathfrak{G}}d\mathcal{V}(x)=vol(\mathfrak{G})$.
    \item Fields: Purely spatial random fields and variables are in mathscript such as $\mathscr{T}(x)$ and spacetime vector random fields are denoted $\mathscr{I}_{a}(x,t)$. Deterministic vector fields are $\Phi_{a}(x)$, etc.
\end{itemize}
The outline is as follows:
\begin{enumerate}
\item In Section 2, purely spatial Gaussian random scalar fields $\mathscr{T}(x)$ and their properties are generically defined with emphasis on the Karhunen-Loeve spectral decomposition of the field. Essential lemmas and results are derived which will be required for later sections. In particular, this expansion enables the spatial derivatives or gradients of the random fields $\mathscr{T}(x)$ to be defined in terms of the derivatives of a set of eigenfunctions $\lbrace f_{I}(x)\rbrace$.
\item In Section 3, stochastic spatio-temporal or random vectors fields or 'flows/fluxes' $\mathscr{I}_{a}(x,t)$ are constructed  by a 'weighted mixing' of a smooth deterministic vectorial field or flux $\Phi_{a}(x,t)$ which evolves via some underlying PDE from initial data, and the GRF $\mathscr{T}(x)$. Also, averaged Sobolov norms can be defined for the random field.
\item In Section 4, this formalism is then applied to a viscous incompressible fluid within $\mathfrak{G}\subset\mathbb{R}^{3}$ so that the stochastic or random vectorial field is now identified with a stochastic fluid flow $\mathscr{U}_{a}(x,t)$ representing turbulence, with a mean flow or expectation ${\mathbf{E}}[{\mathscr{U}}(x,t)]=u_{a}(x,t)]$ evolving via the Navier-Stokes or Burgers equation. An original and key feature of this random field is that it is nonlinearly 'weighted' so that the random fluctuations  scales with the Reynolds number; that is, there is a nonlinear coupling or 'weighting' of the underlying smooth laminar flow into the perturbed random flow. As the Reynolds number increases or decreases, the random contribution of to the overall field also grows or decreases so it becomes a control parameter.
\item In Section 5, assuming constant energy dissipation rate $\epsilon$, small constant viscosity $\nu\sim 0$, corresponding to high Reynolds number, and the standard energy balance law. Then it can be  shown that a zeroth law of anomalous dissipation holds for this random vector field within $\mathfrak{G}$ such that
\begin{align}
\lim_{\nu\rightarrow 0}\left(\lim_{u_{a}(x,t)\rightarrow u_{a}}\sup~\nu{\mathbf{E}}\left[\left|\mathlarger{\nabla}_{b}\mathscr{U}_{a}(x,s)\right|^{2}\right]\right)>0\nonumber
\end{align}
or
\begin{align}
\lim_{\nu\rightarrow 0}\left(\lim_{u_{a}(x,t)\rightarrow u_{a}}\sup~\nu \int_{\mathfrak{G}}\int_{0}^{T}
{\mathbf{E}}\left[\left|\mathlarger{\nabla}_{b}{\mathscr{U}}_{a}(x,s)\right|^{2}\right]d\mathcal{V}(x) ds\right)>0\nonumber
\end{align}
where $\int_{\mathfrak{G}}d\mathcal{V}(x)=vol(\mathfrak{G})$
\item One can also determine the expectation or average of the stochastic Navier-Stokes equation
\begin{align}
&{\mathbf{E}}\left[\frac{\partial}{\partial{t}}{{\mathscr{U}}_{a}(x,t)}+{\mathscr{U}}^{b}(x,t)\mathlarger{\nabla}_{b}{\mathscr{U}}_{a}(x,t)-\nu \mathlarger{\Delta} {{\mathscr{U}}_{a}(x,t)}+\mathlarger{\nabla}_{a}p(x,t)\right]\nonumber\\&
=\frac{\partial}{\partial{t}}u_{a}(x,t)+u^{b}(x,t)\mathlarger{\nabla}_{b}u_{a}(x,t)-\nu \mathlarger{\Delta} u_{a}(x,t)+\mathlarger{\nabla}_{a}p(x,t)+~'extra~terms'
\end{align}
where extra terms may arise due to the nonlinearity of the NS equations.
\item  In addition, one might also attempt to establish that perhaps the mathematical form of at least one of the fundamental 2/3 or 4/5 laws could emerge from the formalism (given a suitable choice of kernel for the random field ) such that the structure function to second order is interpreted as the canonical metric for the random field
\begin{align}
{\mathbf{E}}\big[\big|{\mathscr{U}}_{a}(x+\ell,t)-{\mathscr{U}}_{a}(x,t)\big|^{2}\big]\sim C\epsilon^{\frac{2}{3}}\ell^{\frac{2}{3}}\nonumber
\end{align}
or
\begin{align}
{\mathbf{E}}\big[\big|{\mathscr{U}}_{a}(x+\ell,t)-{\mathscr{U}}_{a}(x,t)\big|^{3}\big]=-\frac{4}{5}{\epsilon\ell}\nonumber
\end{align}
However, this will left for a possible future article. The definitions are made precise in the upcoming sections.
\end{enumerate}
\section{{RANDOM GAUSSIAN SCALAR FIELDS AND KARHUNEN-LOEVE SPECTRAL DECOMPOSITIONS IN~$\mathfrak{G}\subset\mathbb{R}^{d}$}}
It will be assumed that the noise or random fluctuation in fully developed hydrodynamic turbulence in some circumstances, is essentially a generic noise determined by well-established and formal general theorems in probability theory, stochastic analysis, and the theory of random fields/functions. Classical random fields or functions correspond naturally to structures, and properties of systems, that are varying randomly in time and/or space. They have found many useful applications in mathematics and applied science: in the statistical theory or turbulence, in geoscience, machine learning and data science, medical science, engineering, imaging, computer graphics, statistical mechanics and statistics, biology and cosmology $\mathbf{[75-94]}$. Gaussian random fields (GRFs) are of special significance as they are more mathematically tractable and can occur spontaneously in systems with a larger number of degrees of freedom via the central limit theorem.
\subsection {Gaussian random fields}
A GRF is defined with respect to a probability space/triplet as follows:
\begin{defn}(\textbf{Formal definition of Gaussian random fields})\newline
Let $(\bm{\Xi},\mathfrak{F},{\mathbb{P}})$ be a probability space. Then: ${\mathbb{P}}$ is a function such that $\bm{\mathbb{P}}:\mathfrak{F}\rightarrow [0,1]$, so that for all $\omega\in\mathfrak{F}$, there is an associated probability ${\mathbb{P}}(\omega)$. The measure is a probability measure when $\bm{\mathbb{P}}(\bm{\Xi})=1$. Let $x_{a}\subset{\mathfrak{G}}\subset{\mathbb{R}}^{n}$ be Euclidean coordinates and let $\mathscr{T}(x;\omega)$ be a random scalar function that depends on the coordinates $x\subset{\mathfrak{G}}\subset{\mathbb{R}}^{n}$ and also $\omega\in{\mathfrak{G}}$. Given any pair $(x,\omega)$ there ${\exists}$ map $\mathcal{M}:{\mathbb{R}}^{n}\otimes{\mathfrak{G}}\rightarrow{\mathbb{R}}$ such that $\mathcal{M}:(\omega,x)\longrightarrow\mathscr{T}(x;\omega)$, so that $\mathscr{T}(x;\omega)$ is a \textbf{random variable or field} on ${\mathfrak{G}}\subset\mathbb{R}^{n}$ with respect to the probability space $(\bm{\Xi},\mathfrak{F},\mathbb{P})$. A random field is then essentially a family of random variables $\lbrace\bm{\mathscr{T}}(x;\omega)\rbrace$ defined with respect to the space $(\bm{\Xi},\mathfrak{F},\bm{\mathbb{P}})$ and ${\mathbb{R}}^{n}$. The fields can also include a time variable $t\in{\mathbb{R}}^{+}$ so that given any triplet $(x,t,\omega)$ there is a mapping $\mathcal{M}:{\mathbb{R}}\otimes\bm{\Xi}\otimes{\mathfrak{G}}\rightarrow {\mathbb{R}}$ such that $\mathcal{M}:(x,t,\omega)\hookrightarrow {\mathscr{T}}({x},t;\omega)$ is a \textbf{spatio-temporal random field}. Normally, the field will be expressed in the form ${\mathscr{T}}({x})$ with $\omega$ dropped. From here, only spatial fields ${\mathscr{T}}(x)$ will be considered. The random field ${\mathscr{T}}(x)$ will have the following bounds and continuity properties [REFs, Adler etc]
\begin{align}
{\mathbb{P}}[\sup_{x\in\mathfrak{G}}|{\mathscr{T}}({x})|~~<~~\infty]~=+1
\end{align}
and ${\mathbb{P}}[\lim_{x\rightarrow y}\big|{\mathscr{T}}(y)-{\mathscr{T}}(x)\big|=0,~\forall~({x},{y})\in{\mathfrak{G}}]=1$.
\end{defn}
\begin{lem}
The random field is at the least, mean-square differentiable in that \textbf{[56], [62], [64]}
\begin{align}
\mathlarger{\nabla}_{b}\mathscr{T}(x)=\frac{\partial}{\partial x_{b}}\mathscr{T}(x)= \lim_{{\ell}\rightarrow 0} \big\lbrace\mathscr{T}(x+|{\ell}\bm{e}_{b})-\mathscr{T}(x)\big\rbrace{|{\ell}|^{-1}}
\end{align}
where $\bm{e}_{b}$ is a unit vector in the $b^{th}$ direction. For a Gaussian field, sufficient conditions for differentiability can be given in terms
of the covariance or correlation function, which must be regulated at ${x}={y}$ The derivatives of the field $\mathlarger{\nabla}_{a}{\mathscr{T}},
\mathlarger{\nabla}_{a}\mathlarger{\nabla}_{b}{\mathscr{T}}({x})$ exist at least up to 2nd order and do line, surface and volume integrals ${\int}_{\mathfrak{G}}{\mathscr{T}}(x,t)d\mu_{n}(x)$. The derivatives or integrals of a random field are also a random field. However, the derivative will subsequently be defined using a spectral expansion.
\end{lem}
\begin{prop}
Let $(x,y,z,w)\in{\mathfrak{G}}$ and let ${\mathscr{T}}(y),{\mathscr{T}}(x),{\mathscr{T}}(w))$ be GRFs at these points. Then products of random fields at these points are denoted
\begin{align}
&{\mathscr{T}}(x)\otimes\mathscr{T}(y)\nonumber\\&
{\mathscr{T}}(x)\otimes\mathscr{T}(y)\otimes{\mathscr{T}}(z)\nonumber\\&
{\mathscr{T}}(x)\otimes{\mathscr{T}}(y)\otimes{\mathscr{T}}(z)\otimes{\mathscr{T}}(w)
\end{align}
and so on.
\end{prop}
\begin{defn}
The stochastic expectation ${\bm{\mathbf{E}}}[\bullet] $) and binary correlation with respect to the space $(\mathfrak{G},{\mathfrak{F}},\bm{\mathscr{P}})$ is defined as follows, with $(\omega,\zeta)\in{\mathfrak{G}}$ so that
\begin{align}
&{\bm{\mathbf{E}}}\big[\bullet\big]={\int}_{\omega}\bullet~d{\mathbb{P}}[\omega]\\&
{\mathbf{E}}\big[\bullet\otimes\bullet\big]={\int}\!\!\!\!{\int}_{\mathfrak{G}}\bullet\otimes\bullet~d{\mathbb{P}}[\omega]d{\mathbb{P}}[\vartheta]
\end{align}
Then the expectations of products of random fields are
\begin{align}
&{\mathbf{E}}\left[{\mathscr{T}}(x)\otimes\mathscr{T}(y)\right]\nonumber\\&
{\mathbf{E}}\left[{\mathscr{T}}(x)\otimes\mathscr{T}(y)\otimes{\mathscr{T}}(z)\right]\nonumber\\&
{\mathbf{E}}\left[{\mathscr{T}}(x)\otimes{\mathscr{T}}(y)\otimes{\mathscr{T}}(z)\otimes{\mathscr{T}}(w)\right]
\end{align}
For Gaussian random fields where $\bullet=\mathscr{T}(x)$ only the binary correlation or covariance is required so that
$\mathbf{E}\big[{\mathscr{T}}(x)\big]={\int}_{\omega}{\mathscr{T}}(x;\omega)~d{\mathbb{P}}[\omega]=0$ and
\begin{align}
\mathbf{E}\left[{\mathscr{T}}(x)\otimes{\mathscr{T}}(y)\right]=
{\int}\!\!\!\!{\int}_{\mathfrak{G}}{\mathscr{T}}(x;\omega)
{\mathscr{T}}(y;\zeta)~d{\mathbb{P}}[\omega]d{\mathbb{P}}[\zeta]= K(x,y;\lambda)
\end{align}
with $(\omega,\zeta)\in\mathfrak{G}$. The kernel is regulated at ${x}={y}$ for all $({x},{y})\in{\mathfrak{G}}$ and $t\in[0,\infty)$ if $\mathbf{E}[\mathscr{T}(x)\otimes\mathscr{T}]<C <\infty$.
\end{defn}
\begin{defn}
Two random fields ${\mathscr{T}}(x),{\mathscr{T}}(x)({y}))$ defined for any $({x},{y})\in{\mathfrak{G}}$ are correlated or uncorrelated if
$\mathbf{E}\big[{\mathscr{T}}(x)\otimes{\mathscr{T}}(y)\big]\ne 0$ or ${\mathbf{E}}\big[{\mathscr{T}}(x)\otimes\mathscr{T}(y)\big]=0$
\end{defn}
\begin{defn}
The covariance function of a zero-centred Gaussian random field is
\begin{align}
&{Cov}\left({\mathscr{T}}(x),\mathscr{T}(y)\right)={\bm{\mathbf{E}}}[{\mathscr{T}}(x)\otimes{\mathscr{T}}(y)]
+\mathbf{E}\big[{\mathscr{T}}(x)\big[\mathbf{E}\big[{\mathscr{T}}(x)\big]\nonumber\\&
=\mathbf{E}\left[{\mathscr{T}}(x)\otimes{\mathscr{T}}(y)\right]=K(x,y;\lambda)
\end{align}
so that the binary correlation and the covariance are equivalent. Here $\lambda$ is the correlation length. The GRF is isotropic if $K(\|x-y\|;\lambda)=K(x,y;\lambda)$ depends only on the separation $\|x-y\|$ and is stationary if $K(\|(x+\delta x)-(y+\delta y)\|;\lambda)=K(x,y;\lambda)$. Hence, the 2-point function $K(x,y;\lambda)$ is translationally and rotationally invariant in $\mathbb{R}^{d}$ for all $\delta x>0$ and $\delta y>0$.
\end{defn}
Typical kernels for Gaussian random fields ${\mathscr{T}}(x)$ are the rational quadratic form
\begin{align}
K_{RQ}(x,y;\lambda)=\mathbf{E}\left[{\mathscr{T}}(x)\otimes{\mathscr{T}}(y)\right]=\mathcal{N}\left(1+{\|x-y\|^{2}}
\frac{1}{2A\lambda^{2}}\right)^{-\alpha}
\end{align}
where $\lambda$ is the correlation length and $\alpha$ is the 'scale-mixing' parameter. This kernel has found applications in machine learning for example $\mathbf{[92]}$. Another very commonly used covariance kernel is the Gaussian
\begin{align}
K_{G}(x,y;\lambda)=\mathbf{E}~\left[{\mathscr{T}}({x})\otimes{\mathscr{T}}(y)\right]={\mathcal{N}}\exp\left(-{\|x-y\|^{2}}
\frac{1}{\lambda^{2}}\right)
\end{align}
The normalisation constant $\mathcal{N}$ can be ascertained by simply integrating over the volume of the domain so that for the Gaussian kernel
$\mathcal{N}=\left(\int_{\mathfrak{G}}\exp\left(-\|x-y\|{\mathcal{N}}^{-2}\right)d\mathcal{V}(x)\right)^{-1}$. For the dimensionless random field $\mathscr{T}(x)$, one can simply set $\mathcal{N}=1$
\begin{lem}
For the Gaussian kernel on any bounded domain $\mathfrak{G}$, the following convergent limit holds so that
\begin{align}
\lim_{y\rightarrow x}K_{G}(x,y;\lambda)d\mathcal{V}=\lim_{y\rightarrow x}\exp\left(-{|x-y|^{2}}\frac{1}{\lambda^{2}}\right)=1
\end{align}
\end{lem}
In the limit that $\lambda\rightarrow 0$, the noise reduces to a white-in-space noise which is delta correlated such that
$K(x,y|\lambda)={\mathbf{E}}\big[{\mathscr{T}}(x)\otimes{\mathscr{T}}(y)\big]\longrightarrow\mathbf{E}\big[\bm{\mathscr{W}}(x)\otimes
\bm{\mathscr{W}}(s)\big]=-\delta^{3}(x-y) $ and the 2nd-order moment blows up in that $\mathbf{E}[|{\mathscr{T}}({x})|^{2}]=\infty$. The random field or noise
${\mathscr{T}}(x)$ is also differentiable because the field is regulated in that ${\mathbf{E}}[{\mathscr{T}}({x})\otimes{\mathscr{T}}(x)]=K(x,x;\lambda)<\infty$.

Gaussian random fields also have a Fourier representation.
\begin{defn}
Let $\mathcal{J}:\mathbb{R}^{3}\rightarrow\mathbb{K}^{3}$ is a Fourier transform such that
$\mathcal{J}[\mathscr{T}(\xi)=\mathscr{T}(\xi)$. A random
Gaussian scalar field ${\mathscr{T}}(x)$ is said to be \textbf{harmonisable} if it has the Fourier representation
\begin{align}
\mathcal{F}[\mathscr{T}](x)={\mathscr{T}}(\xi)={\int}_{\mathbf{K^{3}}}\exp(-i{\xi}_{a}{x}^{a})\otimes{\mathscr{T}}(x)d\mathcal{V}(x)
\end{align}
The inverse transform is then
\begin{align}
\mathcal{F}^{-1}[\mathscr{T}](\xi)={\mathscr{T}}(x)={\int}_{\mathbf{R^{3}}}\exp(+i{\xi}_{a}{x}^{a})\otimes{\mathscr{T}}(\xi)d\mathcal{V}(\xi)
\end{align}
The basic Fourier representation of the kernel is then
\begin{align}
K(x,y;\lambda)=\mathbf{E}\left[{\mathscr{T}}(x)\otimes{\mathscr{T}}(y)\right]={\int}_{\mathbf{K}^{3}}d\mathcal{V}(\xi)\mathfrak{S}(\xi)
\exp(i \xi_{a}(x-y)^{a})
\end{align}
where $\mathfrak{S}(\xi)$ is a spectral function. For ${x}={y}$ one has $\mathbf{E}\left[{\mathscr{T}}({x})\otimes{\mathscr{T}}({x})\right]={\int}_{\mathbb{K}^{3}}\mathfrak{S}(\xi)d\mathcal{V}(\xi) $.
\end{defn}
For $\mathfrak{S}(\xi)=1$ for example, one recovers an unregulated white noise with
\begin{align}
K(x,y)& ={\mathbf{E}}\big\lbrace{\mathscr{T}}(x)\otimes{\mathscr{T}}(y)\big\rbrace\nonumber={\int}_{\mathbf{K}^{3}}d\mathcal{V}(\xi)
\exp(i\xi_{a}(x-y)^{a})\nonumber\\&={\mathbf{E}}[{\mathscr{W}}(x)\otimes{\mathscr{W}}(y)]=\mathcal{N}\delta^{3}(x-y)
\end{align}
 Again, this is not regulated at $x=y$ and the field is not differentiable. For $\mathfrak{S}(\xi)=\tfrac{\beta}{\xi^{2}}\exp\left(-\tfrac{1}{4}\lambda^{2}\xi^{2}\right)$, one recovers the Gaussian kernel (2.10).
\subsection{The Karhunen-Loeve spectral decomposition}
The Karhunen-Loeve theorem is a fundamental and powerful result from stochastic analysis and has played an important role in the the theory of stochastic processes, numerical analysis, statistical inference and imaging and data analysis \textbf{[94-98]}. A stochastic process or GRF can be expanded as an infinite sum of orthonormal random variables or eigenfunctions weighted by the square root of the eigenvalues.
\begin{thm}(\textbf{The Kuhunen-Loeve spectral expansion of a random field})\newline
Let $\bm{\mathscr{T}}(x)$ be a zero-centred Gaussian random field defined over a domain $\mathfrak{G}$ so that $\mathbf{E}[\bm{\mathscr{T}}(x)\big]=0$ and $\mathbf{E}\big[\bm{\mathscr{T}}(x)\otimes\bm{\mathscr{T}}(y)\big]=K(x,y;\lambda)$ for a stationary and isotropic regulated kernel $K(x,y;\lambda)$ with $K(x,x;\lambda)<\infty$. The Karhunen-Loeve theorem states or spectral representation theorem states that a GRF $\bm{\mathscr{T}}(x)$ on a domain $\mathfrak{G}\subset\mathbb{R}^{3}$ or the whole space, can be represented as a series expansion in terms of orthogonal functions or eigenfunctions. There is a series of real-valued eigenfunctions $f_{I}(x)$, eigenvalues $\mathrm{Z}_{I}$ and uncorrelated (standard) normal Gaussian random variables $\mathscr{Z}_{I}$ such that
\begin{align}
\bm{\mathscr{T}}(x)=\sum_{I=1}^{\infty}\mathrm{Z}_{I}^{1/2}{f}_{I}(x)\otimes \mathscr{Z}_{I}
\end{align}
For some practical applications, the number of terms can also be truncated to any order of accuracy. A finite-dimensional approximation is
${\mathscr{T}}(x)=\sum_{I=1}^{N}{\mathrm{Z}^{1/2}_{I}}f_{I}(x)\otimes\mathscr{Z}_{I}$ with $\mathbf{E}[\mathscr{Z}_{I}]=0$ and $\mathbf{E}[\mathscr{Z}_{I}\otimes\mathscr{Z}_{J}]=\delta_{IJ}$ and eigenvalues $\lbrace\mathrm{Z}_{I}\rbrace$. The eigenvalues are positive and real so that $\mathrm{Z}_{I}\in\mathbb{R}^{+}, \forall I\in\mathbb{Z}$, and the eigenfunctions $f_{I}(x)$ are real functions $f_{I}:\mathfrak{G}\rightarrow\mathbb{R}^{+}$. Then
\begin{align}
{\mathbf{E}}[\big[{\mathscr{T}}(x)\big]=\sum_{I=1}^{\infty}{\mathrm{Z}_{I}^{1/2}}f_{I}(x){\mathbf{E}}\left[\mathscr{Z}_{I}\right]=0
\end{align}
\end{thm}
Given the kernel $K(x,y;\lambda)$, there is a Fredholm integral of the form
\begin{align}
\int_{{\mathfrak{G}}}K(x,y;\lambda)f_{I}(y)d\mathcal{V}(y)=\mathrm{Z}_{I}f_{I}(x)
\end{align}
The set $\lbrace f_{I}(x)\rbrace $ forms a complete set in the space of square-integrable functions on ${\mathfrak{G}}$. The eigenfunctions also obey the
orthonormality condition for all $x\in\mathfrak{G}$.
\begin{align}
\int_{\mathfrak{G}}f_{I}(x)f_{J}(x)d\mathcal{V}(x)=\delta_{IJ}
\end{align}
and
\begin{align}
\sum_{I,J}\int_{\mathfrak{G}}f_{I}(x)f_{J}(x)d\mathcal{V}(x)=\sum_{IJ}\delta_{IJ}
\end{align}
The volume measure $d\mathcal{V}(x)$ is such that $\int_{\mathfrak{G}}d\mathcal{V}(x)=vol(\mathfrak{G})$.
\begin{rem}
The convergence of the series and the properties of the eigenvalues depends on some specific properties of the GRF, its kernel and perhaps the space $\mathfrak{G}$.
\begin{enumerate}
\item \textbf{Convergence}. Convergence of the series is necessary and this depends on the properties of $\mathrm{Z}_{I},f_{I}(x)$ and ${\mathscr{Z}}_{I}$.
where $\|\bullet\|$ is an Euclidean norm. The partial sums can also converge in the $L_{2}$ norms sense
\begin{align}
\lim_{N\rightarrow\infty}\mathbf{E}\left[\left\|{\mathscr{T}}(x)-\sum_{I=1}^{N}{\mathrm{Z}_{I}^{1/2}}f_{I}(x)\otimes{\mathscr{Z}}_{I}
\right\|^{2}_{L_{2}(\mathfrak{G})}\right]=0
\end{align}
Thr proof is given in Appendix A.
\item \textbf{Eigenvalue sum}. The sum $\sum_{I=1}^{\infty}\mathrm{Z}_{I}$ is a measure of the total variance of the random field $\mathscr{T}(x)$ within $\mathfrak{G}$. This sum should therefore converge if the kernel $K(x,y;\lambda)$ is regulated in that $K(x,x;\lambda)<\infty$ and $K(x,y;\lambda)={\bm{\mathbf{E}}}[\mathscr{T}(x)\otimes\mathscr{T}(x)]={\mathbf{E}}[|\mathscr{T}(x)|^{2}]\equiv var(x)$, which is the variance at $x\in\mathfrak{G}$. If the sum $\sum_{I=1}^{\infty}\mathrm{Z}_{I}$ converges rapidly then this implies that most of the variability of the field is captured by the first few terms of the KL expansion, while the remainder decay rapidly, and so a partial sum approximation is often justified.
\item \textbf{$L_{2}$ norms of} $\mathscr{T}$. The $L_{2}$ norms $\|\mathscr{T}\|_{L_{2}(\mathfrak{G})}$ are a measure of the square root of the total variance and should converge.
\end{enumerate}
\end{rem}
If the mean is non-zero then $\mathbf{E}[{\mathscr{T}}(x)]=F(x)$. The KL spectral expansion is then
\begin{align}
\mathscr{T}(x)=F(x)+\sum_{I=1}^{\infty}{\mathrm{Z}_{I}^{1/2}}f_{I}(x)\otimes\mathscr{Z}_{I}
\end{align}
The square root of the eigenvalues $\mathrm{Z}_{I}$ scales the eigenfunctions $f_{I}(x)$ accordingly and represents the magnitude or 'strength' of each mode in the expansion.

One can also have a K-L representation of purely temporal random fields and spatio-temporal random fields. For a purely spatial GRF $\mathscr{T}(t)$ then one has
\begin{align}
\mathscr{T}(t)=\sum_{I=1}^{\infty}\mathrm{Z}_{I}^{1/2}f_{I}(t)\otimes\mathscr{Z}_{I}
\end{align}
with orthogonal eigenfunctions $f_{I}(t)$ such that $\int_{0}^{\infty}f_{I}(s)f_{I}(s)ds=\delta_{IJ}$. The integral can also be defined over a finite interval $[0,T]$. For a spatio-temporal GRF field $\mathscr{T}(x,t)\equiv\mathscr{T}(x,t)$ the expansion has the form
\begin{align}
\mathscr{T}(x,t)=\mathscr{T}(x,t)=\sum_{I=1}^{\infty}\mathrm{Z}_{I}^{1/2}f_{I}(x)f_{I}(t)\otimes{\mathscr{Z}}_{I}
\end{align}
In this paper however, only the spatial GRF $\mathscr{T}(x)$ will be used.

A useful result that will required is Mercer's Theorem for the expectation of a product of two or more fields within the KL representation, which is the (regulated) kernel \textbf{[98]}.
\begin{lem}(\textbf{Mercer's Theorem})
Given the Gaussian random field $\mathscr{T}$ then the expectation of the product of two fields at any $x\in\mathfrak{G}$ and $(x,y)\in\mathfrak{G}$ are
\begin{align}
&K(x,x;\lambda)={{\bm{\mathbf{E}}}}\left[{\mathscr{T}}(x)\otimes{\mathscr{T}}(x)\right]={{\bm{\mathbf{E}}}}[|{\mathscr{T}}(x)|^{2}]
=\sum_{I=1}^{\infty}\mathrm{Z}_{I}f_{I}^{2}(x)=\mathcal{N}=1\nonumber\\&
K(x,y;\lambda)={{\bm{\mathbf{E}}}}\left[{\mathscr{T}}(x)\otimes{\mathscr{T}}(y)\right]
=\sum_{I=1}^{\infty}\mathrm{Z}_{I}f_{I}(x)f_{I}(y)
\end{align}
and all the cross terms vanish upon taking the expectation since $\mathbf{E}[\mathscr{Z}_{I}\otimes\mathscr{Z}_{J}]=\delta_{IJ}$.
\end{lem}
\begin{proof}
\begin{align}
{\mathscr{T}}(x)\otimes{\mathscr{T}}(x)&=\left(\sum_{I=1}^{\infty}{\mathrm{Z}_{I}^{1/2}}f_{I}(x)\otimes{\mathscr{Z}}_{I}\right)
\otimes\left(\sum_{J=1}^{\infty}{\mathrm{Z}_{I}^{1/2}}f_{I}(x)\otimes\mathscr{Z}_{I}\right)\nonumber\\&
=\sum_{I=1}^{\infty}\sum_{J=1}^{\infty}{\mathrm{Z}_{I}^{1/2}}{\mathrm{Z}_{I}^{1/2}}f_{I}(x)f_{I}(x)\bigg(\mathscr{Z}_{I}
\otimes\mathscr{Z}_{J}\bigg)
\end{align}
Taking the expectation
\begin{align}
K(x,x;\lambda)&=\mathbf{E}\left[\mathscr{T}(x)\otimes\mathscr{T}(x)\right]=\mathbf{E}\left[\left(\sum_{I=1}^{\infty}
{\mathrm{Z}_{I}^{1/2}}f_{I}(x)\otimes\mathscr{Z}_{I}\right)\otimes\left(\sum_{J=1}^{\infty}{\mathrm{Z}_{I}^{1/2}}f_{I}(x)\otimes\mathscr{Z}_{I}\right)\right]\nonumber\\&
=\sum_{I=1}^{\infty}\sum_{J=1}^{\infty}{\mathrm{Z}_{I}^{1/2}}{\mathrm{Z}_{I}^{1/2}}f_{I}(x)f_{I}(x){{\bm{\mathbf{E}}}}
\left[\mathscr{Z}_{I}\otimes\mathscr{Z}_{J}\right]\nonumber\\&=\sum_{I=1}^{\infty}\sum_{J=1}^{\infty}{\mathrm{Z}_{I}^{1/2}}{\mathrm{Z}_{I}^{1/2}}
f_{I}(x)f_{I}(x)\delta_{IJ}\equiv \sum_{I=1}^{\infty}{\mathrm{Z}_{I}}f_{I}(x)f_{I}(x)
\end{align}
and where $f_{I}(x)=f_{I}(x)$ and $\mathrm{Z}_{I}=\mathrm{Z}_{I}$ for all $I=J$.
\end{proof}
The result generalises to a product of p fields
\begin{lem}
\begin{align}
{\mathbf{E}}\left[\otimes|\mathscr{T}(x)|^{p}\right]=\sum_{I=1}^{\infty}\mathrm{Z}_{I}^{p/2}f_{I}^{p}(x)\left(\frac{p}{2}-1\right)!!
\end{align}
\end{lem}
The proof is similar to that of Lemma (2.11) and using the fact that $\mathbf{E}[\otimes\mathscr{Z}_{I}]^{p}=(\tfrac{p}{2}-1)!!$ for standard Gaussian random variables.
\begin{cor}
A regulated kernel then has the following representations in terms of Fourier integrals or the KL expansion so that
\begin{align}
K(x,y;\lambda)&={{\bm{\mathbf{E}}}}[\mathscr{T}(x)\otimes\mathscr{T}(y)]\nonumber\\&=\frac{1}{(2\pi)^{3}}\int_{\mathbf{K}^{3}}d\mathcal{V}(\xi)\mathfrak{S}(\xi)\exp(i\xi_{a}(x-y)^{a})
\equiv\sum_{I=1}^{\infty}\mathrm{Z}_{I}f_{I}(x)f_{I}(y)
\end{align}
with
\begin{align}
K(x,x;\lambda)&={{\bm{\mathbf{E}}}}[\mathscr{T}(x)\otimes\mathscr{T}(y)]=\int_{\mathbf{K}^{3}}d\mathcal{V}(\xi)\mathfrak{S}(\xi)
\equiv\sum_{I=1}^{\infty}\mathrm{Z}_{I}f_{I}(x)f_{I}(x)<\infty
\end{align}
Specific choices of the spectral function then lead to specific kernels; for example, the choice $\mathfrak{S}(\xi)=(C/\lambda^{2})\exp(-\tfrac{1}{4}\xi^{2}\lambda^{2})$
again gives the Gaussian kernel.
\end{cor}
Mercer's Theorem is also necessarily consistent with the $L_{2}$-norm convergence of the partial sums of the KL spectral expansion.
\begin{lem}
Let $\mathscr{T}(x)$ be a GRF for all $x\in\mathfrak{G}$. If the KL spectral expansion of the field has $L_{2}$-norm convergence such that
\begin{align}
\lim_{N\rightarrow\infty}{{\bm{\mathbf{E}}}}\left[\left\|\mathscr{T}(x)-\sum_{I=1}^{N}{\mathrm{Z}_{I}^{1/2}}f_{I}(x)\otimes\mathscr{Z}_{I}
\right\|^{2}_{L_{2}(\mathfrak{G})}\right]=0
\end{align}
then one must have
\begin{align}
&{\mathbf{E}}[\big\|\mathscr{T}\big\|^{2}_{L_{2}(\mathfrak{G})}]=\sum_{I=1}^{\infty}\mathrm{Z}_{I}vol(\mathfrak{G})\\&
{\mathbf{E}}[\big|\mathscr{T}(x)\big|^{2}]=\sum_{I=1}^{\infty}\mathrm{Z}_{I}
\end{align}
where $vol(\mathfrak{G})=\int_{\mathfrak{G}}d\mathcal{V}(x)$ and $\|\bullet\|=\int_{\mathfrak{G}}|\bullet|^{2} d\mathcal{V}(x)$
\end{lem}
The proof is in Appendix A.
\begin{lem}
Let $(x,y)\in\mathfrak{G}$ and let ${\mathscr{Z}}(x),{\mathscr{Z}}(y)$ be the fields at $(x,y)$ having a Karhunen-Loeve spectral expansion with
$\mathbf{E}[{\mathscr{Z}}(x)]=0$ and $\mathbf{E}[\mathscr{Z}(x)\otimes{\mathscr{Z}}(y)]=K(x,y;\lambda)$. Then
\begin{enumerate}
\item The Gaussian random normal variable can be expressed as
\begin{align}
\mathscr{Z}_{I}=\frac{1}{\mathrm{Z}_{I}^{1/2}}\int_{\mathfrak{G}}{\mathscr{T}}(x)f_{I}(x)d\mathcal{V}(x)
\end{align}
\item The expectations of the standard Gaussian random variables are again $\mathbf{E}[{\mathscr{Z}}_{I}]=0$ as $\mathbf{E}[{\mathscr{Z}}_{I}\otimes{\mathscr{Z}}_{J}]=\delta_{IJ}$.
\end{enumerate}
\end{lem}
\begin{proof}
To prove (1) from the KL expansion of $\mathscr{T}(x)$ multiply both sides by $f_{I}(x)$ and integrate over $\mathfrak{G}$ so that
\begin{align}
&\int_{\mathfrak{G}}{\mathscr{T}}(x)f_{J}(x)d\mathcal{V}(x)=\int_{\mathfrak{G}}\sum_{I=1}^{n}\mathrm{Z}_{I}^{1/2}f_{I}(x)f_{J}(x)
{{\mathscr{Z}}}_{I}d\mathcal{V}(x)\nonumber\\&=\sum_{I=1}^{\infty}\int_{\mathfrak{G}}f_{I}(x)f_{J}(x)d\mathcal{V}(x)\mathrm{Z}_{I}^{1/2}{{\mathscr{Z}}}_{I}
=\sum_{I=1}^{\infty}\delta_{IJ}\mathrm{Z}_{I}^{1/2}\otimes{{\mathscr{Z}}}_{I}=\mathrm{Z}_{I}^{1/2}\otimes{{\mathscr{Z}}}_{J}
\end{align}
Hence
\begin{align}
\mathscr{Z}_{J}=\frac{1}{\mathrm{Z}_{I}^{1/2}}\int_{\mathfrak{G}}{\mathscr{T}}(x)f_{I}(x)d\mathcal{V}(x)
\end{align}
and $I=J$.
To prove or recover (2) for the Gaussian normal variables ${\mathscr{Z}}$
\begin{align}
\mathscr{Z}_{I}\otimes\mathscr{Z}_{J}&=\left(\frac{1}{{\mathrm{Z}_{I}^{1/2}}}\int_{\mathfrak{G}}{\mathscr{T}}(x)f_{I}(x)d\mathcal{V}(x)\right){\otimes}
\left(\frac{1}{{\mathrm{Z}_{I}^{1/2}}}\int_{\mathfrak{G}}{\mathscr{T}}(y)f_{I}(y)d\mathcal{V}(y)\right)\nonumber\\&
=\frac{1}{{\mathrm{Z}_{I}^{1/2}}{\mathrm{Z}_{I}^{1/2}}}\int_{\mathfrak{G}}\int_{\mathfrak{G}}
\left({\mathscr{T}}(x)\otimes{\mathscr{T}}(y)\right)f_{I}(x)f_{J}(x)d\mathcal{V}(x)d\mathcal{V}(y)
\end{align}
Taking the expectation
\begin{align}
\mathbf{E}\left[{{\mathscr{Z}}}_{I}\otimes{{\mathscr{Z}}}_{J}\right]&=\mathbf{E}\left[\left(\frac{1}{{\mathrm{Z}_{I}^{1/2}}}
\int_{\mathfrak{G}}{\mathscr{T}}(x)f_{I}(x)d\mathcal{V}(x)\right)\otimes
\left(\frac{1}{{\mathrm{Z}_{I}^{1/2}}}\int_{\mathfrak{G}}{\mathscr{T}}(y)f_{I}(y)d\mathcal{V}(y)\right)\right]\nonumber\\&
=\frac{1}{{\mathrm{Z}_{I}^{1/2}}{\mathrm{Z}_{I}^{1/2}}}\int_{\mathfrak{G}}\int_{\mathfrak{G}}\mathbf{E}\left[{\mathscr{T}}(x)\otimes{\mathscr{T}}(y)\right]
f_{I}(x){\mathrm{Z}_{I}}^{1/2}(y)d\mathcal{V}(x)d\mathcal{V}(y)\nonumber\\&
=\frac{1}{{\mathrm{Z}_{I}^{1/2}}{\mathrm{Z}_{I}^{1/2}}}\int_{\mathfrak{G}}\int_{\mathfrak{G}}K(x,y;\lambda)f_{I}(x)f_{I}(y)d\mathcal{V}(x)\nonumber\\&
=\frac{1}{{\mathrm{Z}_{I}^{1/2}}{\mathrm{Z}_{I}^{1/2}}}\int_{\mathfrak{G}}f_{I}(x)\left(\int_{\mathfrak{G}}K(x,y;\lambda)f_{I}(y)d\mathcal{V}(y)
\right)d\mathcal{V}(x)\nonumber\\&=\frac{1}{{\mathrm{Z}_{I}^{1/2}}{\mathrm{Z}_{I}^{1/2}}}\int_{\mathfrak{G}}\mathrm{Z}_{I}f_{I}(x)f_{I}(x)d\mathcal{V}(x)
\nonumber\\&=\frac{{\mathrm{Z}_{I}^{1/2}}}{{\mathrm{Z}_{I}^{1/2}}}\int_{\mathfrak{G}}f_{I}(x)f_{J}(x)d\mathcal{V}(x)
=\frac{{\mathrm{Z}_{I}^{1/2}}}{{\mathrm{Z}_{I}^{1/2}}}\delta_{IJ}=\delta_{IJ}
\end{align}
since $\delta_{IJ}=1$ iff $I=J$. Note the Fredholm integral equation (2.18) was used.
\end{proof}
\begin{prop}
Suppose a GRF has a KL expansion $\bm{\mathscr{T}}(x)=\sum_{I=1}^{\infty}{\mathrm{Z}_{I}}f_{I}(x)\otimes{{\mathscr{Z}}}_{I}$
with kernel ${\bm{{\bm{\mathbf{E}}}}}[\bm{\mathscr{T}}(x)\otimes\bm{\mathscr{T}}(y)]=K(x,y;\lambda)$ then if
\begin{align}
-\mathlarger{\Delta}_{a}f_{I}(x)=\int_{\mathfrak{G}}K(x,y;\lambda){\mathrm{Z}_{I}}(y)d\mathcal{V}(y)
\end{align}
then one can choose the eigenvalues and eigenfunctions to be those of the Laplace operator with Dirichlet boundary conditions such that
\begin{align}
\mathlarger{\Delta} f_{I}(x)=-\mathrm{Z}_{I}f_{I}(x),~~x\in\mathfrak{G},I\in\mathbb{Z},~~~f_{I}(x),~~x\in\partial{\mathfrak{G}},I\in\mathbb{Z}
\end{align}
where the eigenvalues are then all positive and monotonically increasing $\mathrm{Z}_{I}{1}<\mathrm{Z}_{I}{2}<\mathrm{Z}_{I}{3}+<...<\mathrm{Z}_{I}{n-1}
<\mathrm{Z}_{I}{n}<...$.
\end{prop}
\begin{proof}
Using the Fredholm integral (2.18) and (2.39).
\begin{align}
-\mathlarger{\Delta}_{a}f_{I}(x)=\int_{\mathfrak{G}}K(x,y;\lambda)f_{I}(y)d\mathcal{V}(y)=-\mathrm{Z}_{I}f_{I}(x)
\end{align}
\end{proof}
\begin{lem}
Let ${\mathfrak{G}}\subset\mathbb{R}^{3}$ and let the conditions of the Karhunen-Loeve theorem hold such that the Gaussian random field ${\mathscr{T}}(x)$ has the spectral representation $\mathscr{T}(x)=\sum_{I=1}^{\infty}{{\mathrm{Z}_{I}}^{1/2}_{I}}f_{I}(x)\otimes{\mathscr{Z}}_{I}$ with ${{{\bm{\mathbf{E}}}}}[\bm{\mathscr{T}}(x)\otimes\bm{\mathscr{T}}(y)]=K(x,y;\lambda)$, $\mathbf{E}[\bm{\mathscr{T}}(x)\otimes\bm{\mathscr{T}}(x)]=C$, ${{{\bm{\mathbf{E}}}}}[\bm{\mathscr{T}}(x)]=0$ a homogenous and isotropic kernel which depends on $\|x-y\|$ and decays to zero for large $\|x-y\|$. The eigenfunctions satisfy the orthonormality condition $ \int_{\mathfrak{G}}\mathrm{Z}_{I}(x)f_{I}(x)d\mathcal{V}(x)=\delta_{IJ}$. If the random field is square integrable and regulated then
\begin{align}
&\int_{\mathfrak{G}}var(x)d\mathcal{V}=\int_{\mathfrak{G}}K(x,x)={\bm{\mathbf{E}}}\left[\int_{\mathfrak{G}}\bm{\mathscr{T}}(x)\otimes
\bm{\mathscr{T}}(x)\right]d\mathcal{V}(x)\nonumber\\&\equiv
\int_{\mathfrak{G}}{\bm{\mathbf{E}}}\left[{\mathscr{T}}(x)\otimes{\mathscr{T}}(x)\right]d\mathcal{V}(x)]=\int_{\mathfrak{G}}d\mathcal{V}(x)
=vol({\mathfrak{G}})>0\\&\int_{\mathfrak{G}}var(x)d\mathcal{V}=\int_{\mathfrak{G}}K(x,x)
={\bm{\mathbf{E}}}\int_{\mathfrak{G}}\left[\bm{\mathscr{T}}(x)\otimes\bm{\mathscr{T}}(x)\right]d\mathcal{V}(x)\nonumber\\&\equiv \int_{\mathfrak{G}}{\bm{\mathbf{E}}}\left[\bm{\mathscr{T}}(x)\otimes\bm{\mathscr{T}}(x)\right]d\mathcal{V}(x)=\int_{\mathfrak{G}}d\mathcal{V}(x)\nonumber=\Phi vol(\mathfrak{G})<\infty
\end{align}
which requires that $\sum_{I=1}^{\infty}\mathrm{Z}_{I}>0$ and $\sum_{I=1}^{\infty}{\mathrm{Z}_{I}}<\infty $. Then
\begin{align}
\int_{\mathfrak{G}}var(x)d\mathcal{V}=\sum_{I=1}^{\infty}\mathrm{Z}_{I}\equiv\sum_{I=1}^{\infty}\mathrm{Z}_{I}\int_{\mathfrak{G}}
f_{I}(x)f_{I}(x)d\mathcal{V}(x)
\end{align}
since $\int_{\mathfrak{G}}f_{I}(x)f_{I}(x)d\mathcal{V}(x)=1$.
\end{lem}
\begin{proof}
\begin{align}
&\int_{\mathfrak{G}}{\bm{{\bm{\mathbf{E}}}}}\left[\bm{\mathscr{T}}(x)\otimes\bm{\mathscr{T}}(x)\right]d\mathcal{V}(x)\nonumber\\&
=\int_{\mathfrak{G}}\left(\sum_{I=1}^{\infty}{\mathrm{Z}_{I}}f_{I}(x)\otimes{\mathscr{Z}}_{I}\right)\otimes\left(\sum_{I=1}^{\infty}
{\mathrm{Z}_{I}^{1/2}}f_{I}(x)\otimes{\mathscr{Z}}_{I}\right)d\mathcal{V}(x)
\nonumber\\&
=\int_{\mathfrak{G}}\sum_{I=1}^{\infty}\mathrm{Z}_{I}f_{I}(x)f_{I}(x){\bm{{\bm{\mathbf{E}}}}}
\left({\mathscr{Z}}_{I}\otimes{\mathscr{Z}}_{I}\right)d\mathcal{V}(x)\nonumber\\&=\sum_{I=1}^{\infty}\mathrm{Z}_{I}\int_{\mathfrak{G}}f_{I}(x)
f_{I}(x)d\mathcal{V}(x)=\sum_{I=1}^{\infty}\mathrm{Z}_{I}
\end{align}
\end{proof}
\subsection{{Differentiability of random fields and their Sobolov norms}}
Since the KL expansion separates the field $\mathscr{T}(x)$ into it's spatially independent random contribution ${\mathscr{Z}}_{I}$ and its spatially dependant non-random eigenfunction contribution $f_{I}(x)$, the derivatives of the random field at any $x\in\mathfrak{G}$ are
\begin{align}
&\mathlarger{\nabla}_{a}\mathscr{T}(x)=\mathlarger{\nabla}_{a}\mathcal{M}(x)+\sum_{I=1}^{\infty}{\mathrm{Z}_{I}^{1/2}}\mathlarger{\nabla}_{a}f_{I}(x)\otimes\mathscr{Z}_{I}\\&
\mathlarger{\nabla}_{a}\mathlarger{\nabla}_{b}\mathscr{T}(x)=\mathlarger{\nabla}_{a}\mathlarger{\nabla}_{b}\mathcal{M}(x)+\sum_{I=1}^{\infty}{\mathrm{Z}_{I}^{1/2}}\mathlarger{\nabla}_{a}\mathlarger{\nabla}_{b}f_{I}(x)\otimes
\mathscr{Z}_{I}\\&\mathlarger{\Delta}\mathscr{T}(x)\equiv\mathlarger{\nabla}_{a}\mathlarger{\nabla}^{a}\mathscr{T}(x)=\mathlarger{\Delta} \mathcal{M}(x)+\sum_{I=1}^{\infty}{\mathrm{Z}_{I}^{1/2}}\mathlarger{\Delta} f_{I}(x)\otimes
\mathscr{Z}_{I}
\end{align}
For a zero-centred field $\mathcal{M}(x)=0$. The following observations can be made:
\begin{enumerate}
\item $\textbf{Differentiability}$:
The fields $\mathscr{T}(x)$ are essentially differentiable if the eigenfunctions $f_{I}(x)$ are differentiable. Differentiability then depends on the properties and regularity of the $f_{I}$ and $f_{I}(x)$. The $f_{I}(x)$ are square integrable and orthogonal and can form a basis in a  function space like $L_{2}(\mathfrak{G})$. If the $f_{I}(x)$ are smooth enough at least to second order in derivatives then the field $\mathscr{T}(x)$ should inherit that smoothness.
\item$\textbf{Sobolov norms}$: Sobolov norms provide a way to measure smoothness of functions. If $f_{I}(x)\in {W}^{s,p}(\mathfrak{G})$, where $W^{s,p}(\mathfrak{G})$ is the Sobolov space of order p=2 and differentiability order $s$, then the field $\mathscr{T}(x)$ will have similar differentiability properties and inherit the same level of smoothness. A higher-order Sobolov space implies a higher order of differentiability for $\mathscr{T}(x)$. The dimensionality of the domain $\mathfrak{G}$ or whole space $\mathbb{R}^{3}$ can also be crucial.
\end{enumerate}
The criteria for the existence of the derivative of the random field $\mathlarger{\nabla}_{a}\mathscr{T}_{I}(x)$ and $\mathlarger{\nabla}_{a}\mathlarger{\nabla}^{a}\mathscr{T}_{I}(x)$ is now equivalent to the criteria for the existence of the derivatives of the orthogonal eigenfunctions $\mathlarger{\nabla}_{a}f_{I}(x)$ and $\mathlarger{\nabla}_{a}\mathlarger{\nabla}^{a}f_{I}(x)$.
\textbf{Existence of $\mathlarger{\nabla}_{a}f_{I}(x)$}:~ $\mathlarger{\nabla}_{a}f_{I}(x)$ is defined as the partial derivative of $f_{I}(x)$ with respect to the coordinate $x^{a}$. For this derivative to exist one requires:
\begin{itemize}
\item $\underline{\textbf{Differentiability}}$: The set of  eigenfunctions $f_{I}(x)$ must be continiously differentiable with respect to all $x\in\mathfrak{G}$ and $f_{I}(x)$ must belong to the space of continuously differentiable functions, for all $I\ge 1$.
\item $\underline{\textbf{Smoothness}}$:The set of  eigenfunctions $f_{I}(x)$ must be sufficiently smooth so that existence of $\mathlarger{\nabla}_{a}f_{I}(x)$ must also be sufficiently smooth. Hence, $f_{I}(x)$ must not have any discontinuities or singularities within $\mathfrak{G}$ and $f_{I}(x)\in\mathrm{C}^{1}$ for all $I\ge 1$.
\end{itemize}
\textbf{Existence of $\mathlarger{\nabla}_{a}\mathlarger{\nabla}^{a}f_{I}(x)$}:~ $\mathlarger{\nabla}_{a}\mathlarger{\nabla}^{a}f_{I}(x)$ is defined as the 2nd partial derivative of $f_{I}(x)$ with respect to the coordinate $x^{a}$. For this derivative to exist one requires:
\begin{itemize}
\item $\underline{\textbf{Differentiability}}$: The set of  eigenfunctions $f_{I}(x)$ must be twice continuously differentiable with respect to all $x\in\mathfrak{G}$ and $f_{I}(x)$ must belong to the space of continuously differentiable functions, for all $I\ge 1$.
\item $\underline{\textbf{Smoothness}}$:The set of  eigenfunctions $f_{I}(x)$ must be sufficiently smooth so that existence of $\mathlarger{\nabla}_{a}\mathlarger{\nabla}^{a}f_{I}(x)$ must also be sufficiently smooth. Hence, $f_{I}(x)$ must not have any discontinuities or singularities within $\mathfrak{G}$ and $f_{I}(x)\in\mathrm{C}^{2}$ for all $I\ge 1$.
\end{itemize}

The Sobolov norms of a random field $\mathscr{T}(x)$ can be considered. First of all, the norms fort the eigenfunctions are defined:
\begin{prop}
Let $\mathscr{T}(x)$ be a GRV as described and having a KL expansion. If the eigenfunctions $f_{I}(x)$ are differentiable to order $s$ with respect to their pth-order norms then the field $\mathscr{T}(x)$ for all $x\in\mathfrak{G}$ will have similar differentiability properties. The $(s,p)$-Sobolov norms of $f_{I}(x)$ are defined as
\begin{align}
&\big\|f_{I}\big\|_{W^{s,p}(\mathfrak{G})}=\left(\sum_{|\alpha|\le s}\big\|\mathlarger{\nabla}_{a}^{(s)}f_{I}(x)\big\|_{L_{p}(\mathfrak{G})}\right)^{1/p}=\left(\sum_{|\alpha|\le s}\int_{\mathfrak{G}}\big|\mathlarger{\nabla}_{a}^{(s)}f_{I}(x)\big|^{p}d\mathcal{V}(x)
\right)^{1/p}\\&
\big\|f_{I}\big\|_{W^{s,p}(\mathfrak{G})}^{p}=\sum_{|\alpha|\le s}\big\|\mathlarger{\nabla}_{a}^{(s)}f_{I}(x)\big\|_{L_{p}(\mathfrak{G})}=\sum_{|\alpha|\le s}\int_{\mathfrak{G}}\big|\mathlarger{\nabla}_{a}^{(s)}f_{I}(x)\big|^{p}d\mathcal{V}(x)
\end{align}
where $\|\bullet\|_{L_{p}(\mathfrak{G})}=(\int_{\mathfrak{G}}\|\otimes\|^{p}d\mathcal{V}(x))^{1/p}$. If $p=2$ then $W^{s,2}(\mathfrak{G})\equiv {H}^{s}(\mathfrak{G})$ and so
\begin{align}
&\big\|f_{I}\big\|_{H^{s}(\mathfrak{G})}=\left(\sum_{|\alpha|\le s}\big\|\mathlarger{\nabla}_{a}^{(s)}f_{I}(x)\big\|_{L_{p}(\mathfrak{G})}\right)^{1/p}=\left(\sum_{|\alpha|\le s}\int_{\mathfrak{G}}\big|\mathlarger{\nabla}_{a}^{(s)}f_{I}(x)\big|^{p}d\mathcal{V}(x)
\right)^{1/2}\\&
\big\|f_{I}\big\|_{H^{s}(\mathfrak{G})}^{p}=\sum_{|\alpha|\le s}\big\|\mathlarger{\nabla}_{a}^{(s)}f_{I}(x)\big\|_{L_{p}(\mathfrak{G})}=\sum_{|\alpha|\le s}\int_{\mathfrak{G}}\big|\mathlarger{\nabla}_{a}^{(s)}f_{I}(x)\big|^{p}d\mathcal{V}(x)
\end{align}
Then to order $s=2$
\begin{align}
&\big\|f_{I}\big\|^{2}_{\mathrm{H}^{2}(\mathfrak{G})}=\sum_{|\alpha|\le s}\big\|\mathlarger{\nabla}_{a}^{(s)}f_{I}(x)\big\|^{2}_{L_{2}(\mathfrak{G})}\nonumber\\&
=\big\|f_{I}\big\|^{2}_{L_{2}(\mathfrak{G})}+\big\|\mathlarger{\nabla}_{a}f_{I}\big\|^{2}_{L_{2}(\mathfrak{G})}+\big\|\mathlarger{\nabla}_{a}\mathlarger{\nabla}_{b}
f_{I}\big\|^{2}_{L_{2}(\mathfrak{G})}\nonumber\\&
=\int_{\mathfrak{G}}\big|f_{I}(x)\big|^{2}d\mathcal{V}(x)+\int_{\mathfrak{G}}\big|\mathlarger{\nabla}_{a}f_{I}(x)\big|^{2}d\mathcal{V}(x)
+\int_{\mathfrak{G}}\big|\mathlarger{\nabla}_{a}\mathlarger{\nabla}_{b}f_{I}(x)\big|^{2}d\mathcal{V}(x)\nonumber\\&
=1+\int_{\mathfrak{G}}\big|\mathlarger{\nabla}_{a}f_{I}(x)\big|^{2}d\mathcal{V}(x)
+\int_{\mathfrak{G}}\big|\mathlarger{\nabla}_{a}\mathlarger{\nabla}_{b}f_{I}(x)\big|^{2}d\mathcal{V}(x)
\end{align}
and using the orthogonality condition. For $s=1$
\begin{align}
&\big\|f_{I}\big\|^{2}_{{H}^{1}(\mathfrak{G})}=\sum_{|\alpha|\le s}\big\|\mathlarger{\nabla}_{a}^{(s)}f_{I}(x)\big\|^{2}_{L_{2}(\mathfrak{G})}
=\big\|f_{I}\big\|^{2}_{L_{2}(\mathfrak{G})}+\big\|\mathlarger{\nabla}_{a}f_{I}\big\|^{2}_{L_{2}(\mathbf{W})}\nonumber\\&
=\int_{\mathfrak{G}}\big|f_{I}(x)\big|^{2}d\mathcal{V}(x)+\int_{\mathfrak{G}}\big|\mathlarger{\nabla}_{a}f_{I}(x)\big|^{2}d\mathcal{V}(x)=1+\int_{\mathfrak{G}}\big|\mathlarger{\nabla}_{a}
f_{I}(x)\big|^{2}d\mathcal{V}(x)
\end{align}
\end{prop}
Next, Sobolov norms for the field $\mathscr{T}(x)$ can be defined.
\begin{prop}
\begin{align}
&\big\|\mathscr{T}\big\|_{W^{(s,p)}(\mathfrak{G})}=\left(\sum_{|\alpha|\le s}\big\|\mathlarger{\nabla}_{a}^{(s)}\mathscr{T}\big\|_{L_{p}(\mathfrak{G})}\right)^{1/p}\\&\big\|\mathscr{T}\big\|^{p}_{W^{(s,p)}(\mathfrak{G})}=\sum_{|\alpha|\le s}\big\|\mathlarger{\nabla}_{a}^{(s)}\mathscr{T}\big\|^{p}_{L_{p}(\mathfrak{G})}
\end{align}
and for $p=2$
\begin{align}
&\big\|\mathscr{T}\big\|_{H^{s}(\mathfrak{G})}=\left(\sum_{|\alpha|\le s}\big\|\mathlarger{\nabla}_{a}^{(s)}\mathscr{T}\big\|_{L_{2}(\mathfrak{G})}\right)^{1/2}\\&\big\|\mathscr{T}\big\|^{2}_{H^{s}(\mathfrak{G})}=\sum_{|\alpha|\le s}\big\|\mathlarger{\nabla}_{a}^{(s)}\mathscr{T}\big\|^{2}_{L_{2}(\mathfrak{G})}
\end{align}
The expectations are then
\begin{align}
&\mathbf{E}\big[\big\|\mathscr{T}\big\|_{H^{s}(\mathfrak{G})}\big]=\mathbf{E}\left(\sum_{|\alpha|\le s}\big[\big\|\mathlarger{\nabla}_{a}^{(s)}\mathscr{T}\big\|_{L_{2}(\mathfrak{G})]}\right)^{1/2}\\&
\mathbf{E}\big[\big\|\mathscr{T}\big\|^{2}_{H^{s}(\mathfrak{G})}\big]=\sum_{|\alpha|\le s}{\bm{{\bm{\mathbf{E}}}}}\big[\big\|\mathlarger{\nabla}_{a}^{(s)}\mathscr{T}\big\|^{2}_{L_{2}(\mathfrak{G})]}
\end{align}
Using the KL expansion for $s=2$
\begin{align}
\big\|\mathscr{T}\big\|^{2}_{H^{2}({\mathfrak{G}})}&=\sum_{|\alpha|\le s}\big\|\mathlarger{\nabla}_{a}^{(s)}\mathscr{T}\big\|^{2}_{L_{2}(\mathfrak{G})}\nonumber\\& =\big\|\mathscr{T}\big\|^{2}_{L_{2}(\mathfrak{G})}+\|\mathlarger{\nabla}_{a}\mathscr{T}\big\|^{2}_{L_{2}(\mathfrak{G})}+
\|\mathlarger{\nabla}_{a}\mathlarger{\nabla}_{b}\mathscr{T}\big\|^{2}_{L_{2}(\mathfrak{G})}\nonumber\\&
=\sum_{I=1}^{\infty}\big\|\mathrm{Z}_{I}^{1/2}{f}_{I}(x)\otimes\mathscr{Z}_{I}\big\|^{2}_{L_{2}(\mathfrak{G})}+
\sum_{I=1}^{\infty}\|\mathrm{Z}_{I}^{1/2}\mathlarger{\nabla}_{a}{f}_{I}(x)\otimes\mathscr{Z}_{I}\big\|^{2}_{L_{2}(\mathfrak{G})}\nonumber\\&+\sum_{I=1}^{\infty}\big\|
\mathrm{Z}_{I}\mathlarger{\nabla}_{a}\mathlarger{\nabla}_{b}{f}_{I}(x){f}_{I}(x)\big\|^{2}_{L_{2}(\mathfrak{G})}
\end{align}
The expectation is then
\begin{align}
\mathbf{E}\big[\big\|\mathscr{T}\big\|^{2}_{H^{2}(\mathfrak{G})}\big]&=\sum_{I=1}^{\infty}\mathbf{E}\big[\big\|\mathrm{Z}_{I}^{1/2}
{f}_{I}(x)\otimes{\mathscr{Z}}_{I}\big\|^{2}_{L_{2}
(\mathfrak{G})}\big]\nonumber\\&+\sum_{I=1}^{\infty}\mathbf{E}\big[\big\|\mathrm{Z}_{I}^{1/2}\mathlarger{\nabla}_{a}{f}_{I}(x)\otimes
\mathscr{Z}\big\|^{2}_{L_{2}(\mathfrak{G})}\big]+\sum_{I=1}^{\infty}\mathbf{E}\big[\big\|\mathrm{Z}_{I}\mathlarger{\nabla}_{a}\mathlarger{\nabla}_{b}
{f}_{I}(x)\otimes\mathscr{Z}_{I}\big\|^{2}_{L_{2}(\mathfrak{G})}\big]
\end{align}
Similarly, for $s=1$
\begin{align}
&\mathbf{E}\big[\big\|\mathrm{Z}_{I}\big\|^{2}_{H^{2}(\mathfrak{G})}\big]=\sum_{I=1}^{\infty}\mathbf{E}\big[\big\|\mathrm{Z}_{I}^{1/2}
f_{I}(x)\otimes\mathscr{Z}\big\|^{2}_{L_{2}(\mathfrak{G})}\big]+\sum_{I=1}^{\infty}\mathbf{E}\big[\big\|\mathrm{Z}_{I}^{1/2}\mathlarger{\nabla}_{a}f_{I}(x)
\mathscr{Z}\big\|^{2}_{L_{2}(\mathfrak{G})}\big]
\end{align}
\end{prop}
\begin{lem}
For the GRF $\mathscr{T}(x)$
\begin{align}
s=2:~~\mathbf{E}\big[\big\|\mathscr{T}\big\|^{2}_{H^{2}(\mathfrak{G})}\big]&=\sum_{I=1}^{\infty}\mathrm{Z}_{I}+\sum_{I=1}^{\infty}\mathrm{Z}_{I}
\int_{\mathfrak{G}}\mathlarger{\nabla}_{a}f_{I}(x)\mathlarger{\nabla}^{a}f_{I}(x)d\mathcal{V}(x)\nonumber\\&+\sum_{I=1}^{\infty}\mathrm{Z}_{I}\int_{\mathfrak{G}}\mathlarger{\nabla}_{a}\mathlarger{\nabla}_{b}
f_{I}(x)\mathlarger{\nabla}^{a}\mathlarger{\nabla}^{b}f_{I}(x)d\mathcal{V}(x)
\end{align}
\begin{align}
&s=1:~~\mathbf{E}\left[\big\|\mathscr{T}\big\|^{1}_{H^{1}(\mathfrak{G})}\right]=\sum_{I=1}^{\infty}\mathrm{Z}_{I}+\sum_{I=1}^{\infty}\mathrm{Z}_{I}
\int_{\mathfrak{G}}\mathlarger{\nabla}_{a}f_{I}(x)\mathlarger{\nabla}^{a}
f_{I}(x)d\mathcal{V}(x)
\end{align}
where $\sum_{I=1}^{\infty}\mathrm{Z}_{I}<\infty$ and greater than zero.
\end{lem}
\begin{proof}
For the case $s=2$, the norm is
\begin{align}
\big\|\mathscr{T}\big\|^{2}_{H^{2}(\mathfrak{G})}&=\sum_{|\alpha|\le s}\big\|\mathlarger{\nabla}_{a}^{(s)}\mathscr{T}\big\|^{2}_{L_{2}(\mathfrak{G})} =\big\|\mathscr{T}\big\|^{2}_{L_{2}(\mathfrak{G})}+\|\mathlarger{\nabla}_{a}\mathscr{T}\big\|^{2}_{L_{2}(\mathfrak{G})}+\|\mathlarger{\nabla}_{a}\mathlarger{\nabla}_{b}\mathscr{T}\big\|
^{2}_{L_{2}(\mathscr{Z})}\nonumber\\&
=\sum_{I=1}^{\infty}\big\|{\mathrm{Z}_{I}^{1/2}}f_{I}(x)\otimes{\mathscr{Z}}_{I}\big\|^{2}_{L_{2}(\mathfrak{G})}\nonumber\\&+
\sum_{I=1}^{\infty}\|\mathrm{Z}_{I}^{1/2}\mathlarger{\nabla}_{a}f_{I}(x)\otimes{\mathscr{Z}}_{I}\big\|^{2}_{L_{2}(\mathfrak{G})}+\sum_{I=1}^{\infty}\big\|\mathrm{Z}_{I}^{1/2}\mathlarger{\nabla}_{a}\mathlarger{\nabla}_{b}
f_{I}(x)\otimes{\mathscr{Z}}_{I}\big\|^{2}_{L_{2}(\mathfrak{G})}\nonumber\\&
=\sum_{I=1}^{\infty}\int_{\mathfrak{G}}\big|{\mathrm{Z}_{I}^{1/2}}f_{I}(x)\otimes{\mathscr{Z}}_{I}\big|^{2}_{L_{2}\mathfrak{G})}d\mathcal{V}(x)\nonumber\\&+
\sum_{I=1}^{\infty}\int_{\mathfrak{G}}\big|\mathrm{Z}_{I}^{1/2}\mathlarger{\nabla}_{a}f_{I}(x)\otimes{\mathscr{Z}}_{I}|^{2}d\mathcal{V}(x)+\sum_{I=1}^{\infty}\int_{\mathfrak{G}}
\big|\mathrm{Z}_{I}^{1/2}\mathlarger{\nabla}_{a}\mathlarger{\nabla}_{b}f_{I}(x)\big|^{2}\bigg(\mathscr{Z}_{I}\otimes\mathscr{Z}_{I}\bigg)d\mathcal{V}(x)\nonumber\\&
=\sum_{I=1}^{\infty}\mathrm{Z}_{I}\bigg({\mathscr{Z}}_{I}\otimes{\mathscr{Z}}_{I}\bigg)\underbrace{\int_{\mathfrak{G}}\big|f_{I}(x)\big|^{2}d\mathcal{V}(x)}_{=1}
\nonumber\\&+\sum_{I=1}^{\infty}\mathrm{Z}_{I}\bigg({\mathscr{Z}}_{I}\otimes{\mathscr{Z}}_{I}\bigg)\int_{\mathfrak{G}}\big|\mathlarger{\nabla}_{a}f_{I}(x)\big|^{2}d\mathcal{V}(x)
+\sum_{I=1}^{\infty}\mathrm{Z}_{I}\bigg({\mathscr{Z}}_{I}\otimes{\mathscr{Z}}_{I}\bigg)\int_{\mathfrak{G}}\big|\mathlarger{\nabla}_{a}\mathlarger{\nabla}_{b}f_{I}(x)\big|^{2}d\mathcal{V}(x)
\nonumber\\&=\sum_{I=1}^{\infty}\mathrm{Z}_{I}\bigg({\mathscr{Z}}_{I}\otimes{\mathscr{Z}}_{I}\bigg)+\sum_{I=1}^{\infty}\mathrm{Z}_{I}
\bigg({\mathscr{Z}}_{I}\otimes{\mathscr{Z}}_{I}\bigg)\int_{\mathfrak{G}}\big|\mathlarger{\nabla}_{a}f_{I}(x)\big|^{2}d\mathcal{V}(x)\nonumber\\&+\sum_{I=1}^{\infty}\mathrm{Z}_{I}
\bigg({\mathscr{Z}}_{I}\otimes{\mathscr{Z}}_{I}\bigg)\int_{\mathfrak{G}}\big|\mathlarger{\nabla}_{a}\mathlarger{\nabla}_{b}f_{I}(x)\big|^{2}d\mathcal{V}(x)\nonumber\\&
=\sum_{I=1}^{\infty}\mathrm{Z}_{I}\bigg({\mathscr{Z}}_{I}\otimes{\mathscr{Z}}_{I}\bigg)+\sum_{I=1}^{\infty}\mathrm{Z}_{I}\bigg({\mathscr{Z}}_{I}\otimes{\mathscr{Z}}_{I}
\bigg)\int_{\mathfrak{G}}\big|\mathlarger{\nabla}_{a}f_{I}(x)\big|^{2}d\mathcal{V}(x)\nonumber\\&
+\sum_{I=1}^{\infty}\mathrm{Z}_{I}\bigg({\mathscr{Z}}_{I}\otimes\mathscr{Z}_{I}\bigg)\int_{\mathfrak{G}}\big|\mathlarger{\nabla}_{a}\mathlarger{\nabla}_{b}f_{I}(x)\big|^{2}d\mathcal{V}(x)
\end{align}
Now taking the expectation
\begin{align}
&\mathbf{E}[\big\|\mathscr{T}\big\|^{2}_{H^{2}({\mathfrak{G}})}]=\sum_{I=1}^{\infty}\mathrm{Z}_{I}
\mathbf{E}\left[\mathscr{Z}_{I}\otimes\mathscr{Z}_{I}\right]
\nonumber\\&
+\sum_{I=1}^{\infty}\mathrm{Z}_{I}\mathbf{E}\big[\mathscr{Z}_{I}\otimes\mathscr{Z}_{I}\big]\int_{\mathfrak{G}}\big|\mathlarger{\nabla}_{a}f_{I}(x)\big|^{2}d\mathcal{V}(x)
+\sum_{I=1}^{\infty}\mathrm{Z}_{I}\mathbf{E}\big[\mathscr{Z}_{I}\otimes\mathscr{Z}_{I}\big]\int_{\mathfrak{G}}\big|\mathlarger{\nabla}_{a}\mathlarger{\nabla}_{b}f_{I}(x)
\big|^{2}d\mathcal{V}(x)
\nonumber\\&
=\sum_{I=1}^{\infty}\mathrm{Z}_{I}+\sum_{I=1}^{\infty}\mathrm{Z}_{I}\int_{\mathfrak{G}}\big|\mathlarger{\nabla}_{a}f_{I}(x)\big|^{2}d\mathcal{V}(x)
+\sum_{I=1}^{\infty}\mathrm{Z}_{I}\int_{\mathfrak{G}}\big|\mathlarger{\nabla}_{a}\mathlarger{\nabla}_{b}f_{I}(x)\big|^{2}d\mathcal{V}(x)\nonumber\\&
=\sum_{I=1}^{\infty}\mathrm{Z}_{I}\left(1+\int_{\mathfrak{G}}\big|\mathlarger{\nabla}_{a}f_{I}(x)\big|^{2}d\mathcal{V}(x)
+\int_{\mathfrak{G}}\big|\mathlarger{\nabla}_{a}\mathlarger{\nabla}_{b}f_{I}(x)\big|^{2}d\mathcal{V}(x)\right)
\end{align}
which is (2.62).
\end{proof}
The significance of the expectations of the ${H}^{1}(\mathfrak{G})$ and ${H}^{2}(\mathfrak{G})$ Sobolev norms in this context is related to the regularity of the random field. The $H^{1}(\mathfrak{G})$ Sobolev norm essentially measures the expectation of the energy norm of the first-order derivatives of the field. In the context of a random field, the expectation of ${H}^{1}(\mathfrak{G})$ gives a measure of the average energy associated with the variations and gradients of the random field. It is an indicator of the smoothness or regularity of the field in a weak sense. Similarly, the ${H}^{2}(\mathfrak{G})$ Sobolev norm measures the expectation of the square root of the energy norm of the second-order derivatives of the field. The expectation of $H^{2}(\mathfrak{G})$ provides information about the average energy associated with the random field. It gives insights into the regularity of the field in a stronger sense compared to the $H^{1}$ norm.

In summary, analyzing the expectations of Sobolev norms helps to characterize the regularity properties of the random field. Higher Sobolev norms correspond to higher degrees of smoothness and regularity in the field. Understanding these norms is crucial in applications where the regularity of the random field is important. Within the context of Sobolev norms, a smaller $H^{2}$ norm indicates that the field is smoother. The $H^{2}$ norm measures the energy of a function along with its first and second-order spatial derivatives. If the $H^{2}$ norm is small, it means that these derivatives are well-behaved and the function has a higher degree of smoothness. To elaborate further.
\begin{cor}
The Sobolov norms of the GRF $\mathscr{T}(x)$ for $s=1$ an $s=2$ can also expressed in the following form via Mercer's Theorem. For $s=1$
\begin{align}
&\mathbf{E}\left[\big\|\mathscr{T}(x)\big\|^{2}_{H^{1}(\mathfrak{G})}\right]=\sum_{I=1}^{\infty}\mathrm{Z}_{I}+\sum_{I=1}^{\infty}\mathrm{Z}_{I}
\int_{\mathfrak{G}}\mathlarger{\nabla}_{a}f_{I}(x)\mathlarger{\nabla}_{a}f_{I}(x)d\mathcal{V}(x)\nonumber\\&
=\sum_{I=1}^{\infty}\mathrm{Z}_{I}\int_{\mathfrak{G}}f_{I}(x)f_{I}(x)d\mathcal{V}(x)+\sum_{I=1}^{\infty}\mathrm{Z}_{I}
\int_{{\mathfrak{G}}}\mathlarger{\nabla}_{a}f_{I}(x)\mathlarger{\nabla}_{a}f_{I}(x)d\mathcal{V}(x)\nonumber\\&
=\int_{{\mathfrak{G}}}\sum_{I=1}^{\infty}\mathrm{Z}_{I}f_{I}(x)f_{I}(x)d\mathcal{V}(x)+\sum_{I=1}^{\infty}\mathrm{Z}_{I}
\int_{{\mathfrak{G}}}\mathlarger{\nabla}_{a}f_{I}(x)\mathlarger{\nabla}_{a}f_{I}(x)d\mathcal{V}(x)\nonumber\\&
=K(x,x;\lambda)+\sum_{I=1}^{\infty}\mathrm{Z}_{I}
\int_{{\mathfrak{G}}}\mathlarger{\nabla}_{a}f_{I}(x)\mathlarger{\nabla}_{a}f_{I}(x)d\mathcal{V}(x)\nonumber\\&
=var(x)+\sum_{I=1}^{\infty}\mathrm{Z}_{I}
\int_{{\mathfrak{G}}}\mathlarger{\nabla}_{a}f_{I}(x)\mathlarger{\nabla}_{a}f_{I}(x)d\mathcal{V}(x)
\end{align}
and similarly for $s=2$
\begin{align}
&{\bm{{\bm{\mathbf{E}}}}}\left[\big\|\mathscr{T}(x)\big\|^{2}_{H^{2}(\mathfrak{G})}\right]=\int_{\mathfrak{G}}var(x)d\mathcal{V}(x)+\sum_{I=1}^{\infty}\mathrm{Z}_{I}
\int_{\mathfrak{G}}\mathlarger{\nabla}_{a}f_{I}(x)\mathlarger{\nabla}_{a}f_{I}(x)d\mathcal{V}(x)\nonumber\\&+\sum_{I=1}^{\infty}\mathrm{Z}_{I}
\int_{\mathfrak{G}}\mathlarger{\nabla}_{a}\mathlarger{\nabla}_{b}f_{I}(x)\mathlarger{\nabla}_{a}\mathlarger{\nabla}_{b}f_{I}(x)d\mathcal{V}
\end{align}
\end{cor}
The volume integral of the variance
\begin{align}
\int_{\mathfrak{G}}var(x)d\mathcal{V}(x)=\sum_{I=1}^{\infty}\mathrm{Z}_{I}
\end{align}
represents the total variance of the Gaussian random field over the domain and the
total contribution of each eigenfunction to the overall variance within the domain. The larger eigenvalues contribute more to the total variance.
\begin{enumerate}
\item \textbf{Smaller $H^{2}$ Norm}: If $\mathbf{E}[|\mathscr{T}|_{H^{2}(\mathfrak{G})}]$ is small, it implies that on average the field $\mathscr{T}(x)$, along with its first and second-order spatial derivatives, varies smoothly across the domain $\mathfrak{G}$. The smaller norm indicates that the field is regular and doesn't exhibit rapid or abrupt changes in its behavior. Consequently its variance is smaller.
\item \textbf{Larger $H^{2}$ Norm}: Conversely, if ${\mathbf{E}}[|\mathscr{T}|_{H^{2}(\mathfrak{G})}]$ is large, it suggests on average that the field has more complex and oscillatory variations, and the spatial derivatives are not as well-behaved. A larger norm indicates less smoothness and more irregularity in the field.
\end{enumerate}
In summary, a smaller $H^{2}$ norm is associated with a smoother field, while a larger $H^{2}$ norm suggests a less smooth or more oscillatory behavior in the field. Therefore if we have two uncorrolated random fields $\mathscr{T}(x)$ and $\mathscr{Z}(x)$ existing on $\mathfrak{G}$ such that ${\mathbf{E}}[\mathscr{T}(x)\otimes\mathscr{Z}(x)]=0$ with ${{\bm{\mathbf{E}}}}[|\mathscr{T}|_{H^{2}(\mathfrak{G})}]\le {\mathbf{E}}[|\mathscr{Z}|_{H^{2}(\mathfrak{G})}]$ then $\mathscr{T}(x)$ is smoother across the domain than the 'rougher' field $\mathscr{Z}(x)$. Consequently, its variance is larger.
\subsection{Integrals over derivatives of the eigenfunctions}
In this subsection, the following integrals over eigenfunctions and their derivatives are considered, knowing that the eigenfunctions are orthogonal over a volume integral of the domain.
\begin{align}
&\mathbf{H}(\mathfrak{G},I)=\big\|f_{I}\big\|^{2}_{L^{2}(\mathfrak{G})}=\int_{\mathfrak{G}}
f_{I}(x)f_{I}(x)d\mathcal{V}(x)=1\\&
\mathbf{H}_{a}(\mathfrak{G},I)=\int_{\mathfrak{G}}\mathlarger{\nabla}_{a}f_{I}(x)f_{I}(x)d\mathcal{V}(x)\\&
\mathbf{H}_{a}^{a}(\mathfrak{G},I)=\big\|\mathlarger{\nabla}_{a}f_{I}\big\|^{2}_{L^{2}(\mathfrak{G})}
=\int_{\mathfrak{G}}\mathlarger{\nabla}_{a}f_{I}(x)\mathlarger{\nabla}^{a}f_{I}(x)d\mathcal{V}(x)
\end{align}
It will be important for the main proof in Section 5 to establish whether the integrals are strictly zero or greater than zero. The following lemma establishes that the integral over the product of $f_{I}(x)$ and its 1st derivative vanishes.
\begin{lem}
\begin{align}
\mathbf{H}_{a}(\mathfrak{G},I)=\int_{\mathfrak{G}}f_{I}(x)\mathlarger{\nabla}_{a}f_{I}(x)d\mathcal{V}(x)=0
\end{align}
\end{lem}
\begin{proof}
From the orthonormality condition
\begin{align}
\mathlarger{\nabla}_{a}\mathbf{H}(\mathfrak{G},I)=\mathlarger{\nabla}_{a}\int_{\mathfrak{G}}f_{I}(x)f_{I}(x)d\mathcal{V}(x)d\mathcal{V}(x)=\mathlarger{\nabla}_{a}~1=0
\end{align}
so that
\begin{align}
&\mathlarger{\nabla}_{a}\int_{\mathfrak{G}}f_{I}(x)f_{I}(x)d\mathcal{V}(x)\nonumber\\&
=\int_{\mathfrak{G}}f_{I}(x)\mathlarger{\nabla}_{a}f_{I}(x)+\mathlarger{\nabla}_{a}f_{I}(x)
f_{I}(x) d\mathcal{V}(x)=2\int_{\mathfrak{G}}f_{I}(x)\mathlarger{\nabla}_{a}f_{I}(x)=0
\end{align}
and hence $\mathlarger{\nabla}_{a}\mathbf{H}(\mathfrak{G},I)=2{\mathbf{H}}_{a}(\mathfrak{G},I)$ and the result follows.
\end{proof}
\subsubsection{Proofs of positivity of $\int_{\mathfrak{G}}\mathlarger{\nabla}_{a}f_{I}(x)\mathlarger{\nabla}^{a}f_{I}(x)d\mathcal{V}(x)$}
We will later will require the following result for the derivatives of the eigenfunctions
\begin{lem}
Given the following criteria:
\begin{enumerate}
\item $f_{I}(x)>0$ for all $x\in{\mathfrak{G}}\subset\mathbb{R}$ and $I\in\mathbb{Z}$.
\item ${\mathfrak{G}}$ is a bounded domain.
\item $\int_{\mathfrak{G}}f_{I}(x)f_{I}(x)d\mathcal{V}(x)=\delta_{IJ}$, for all $x\in{\mathfrak{G}},a\in\mathbb{Z}$.
\item $\mathlarger{\nabla}_{a}f_{I}(x)>0$ or $\mathlarger{\nabla}_{a}f_{I}(x)<0$ for all $x\in{\mathfrak{G}}$, $a\in \mathbb{Z}$.
\end{enumerate}
then
\begin{align}
\mathbf{H}(\mathfrak{G},I)_{a}^{a}=\|\mathlarger{\nabla}_{a}f_{I}\|^{2}_{L^{2}(\mathfrak{G}}=\int_{\mathfrak{G}}
\mathlarger{\nabla}_{a}f_{I}(x)\mathlarger{\nabla}^{a}f_{I}(x)d\mathcal{V}(x)>0
\end{align}
\end{lem}
\begin{proof}
Since $\mathlarger{\nabla}_{a}f_{I}(x)>0$ or $\mathlarger{\nabla}_{a}f_{I}(x)<0$ it follows that $|\mathlarger{\nabla}_{a}f_{I}(x)|^{2}>0$ for all $x\in{\mathfrak{G}}$ and $I\in\mathbb{Z}$. The integral over a positive integrand within a bounded domain is always positive so the result follows.
\end{proof}
Other more detailed criteria also establish the positivity of this integral.
\begin{lem}
If the eigenfunctions are positive then the integral or $L_{2}$ norm over the derivatives i s always positive and non-zero.
\begin{align}
\|\mathlarger{\nabla}_{b}f_{I}(x)\|_{L_{2}(\mathfrak{G})}^{2}=\int_{\mathfrak{G}}\mathlarger{\nabla}_{b}f_{I}(x)\mathlarger{\nabla}^{b}f_{I}(x)d\mathcal{V}(x)>0
\end{align}
This will be guaranteed for all stationary, isotropic and homogenous kernels $K(x,y;\lambda)$ such that
\begin{align}
\lim_{y\rightarrow x}\int_{\mathfrak{G}}\mathlarger{\nabla}_{a}^{(x)}\mathlarger{\nabla}_{b}^{(y)}K(x,y|;\lambda)d\mathcal{V}(x)>0
\end{align}
where $\mathlarger{\nabla}_{b}^{(y)}=\tfrac{\partial}{\partial y}_{b}$ and $\mathlarger{\nabla}_{b}^{(x)}=\tfrac{\partial}{\partial x}_{I}$
\end{lem}
\begin{proof}
\begin{align}
&\mathbf{H}(\mathfrak{G},I)=\mathlarger{\nabla}_{b}\int_{\mathfrak{G}}f_{I}(x)\mathlarger{\nabla}_{b}f_{I}(x)d\mathcal{V}(x)\nonumber\\&
=\int_{\bm{\mathbf{O}}}\mathlarger{\nabla}_{b}f_{I}(x)\mathlarger{\nabla}^{b}f_{I}(x)d\mathcal{V}(x)+\int_{\mathfrak{G}}
f_{I}(x)\mathlarger{\Delta} f_{I}(x)d\mathcal{V}(x)=0
\end{align}
Hence
\begin{align}
\int_{\bm{\mathbf{O}}}\mathlarger{\nabla}_{b}f_{I}(x)\mathlarger{\nabla}^{b}f_{I}(x)d\mathcal{V}(x)=-\int_{\mathfrak{G}}
f_{I}(x)\mathlarger{\Delta} f_{I}(x)d\mathcal{V}(x)>0
\end{align}
Since the eigenfunctions are positive then so is their gradient. However, even for a negative gradient the square is positive so that
$|\mathlarger{\nabla}_{b}f_{I}(x)|^{2}>0$. However, the positivity is guaranteed if\newline $\int_{\mathfrak{G}}\mathlarger{\nabla}_{a}^{(x)}\mathlarger{\nabla}_{b}^{(y)}K(x,y;\lambda)d\mathcal{V}(x)>0$.
One has
\begin{align}
&\int_{\mathfrak{G}}\mathlarger{\nabla}_{a}^{(x)}\mathlarger{\nabla}_{b}^{(y)}K(x,y;\lambda)d\mathcal{V}(x)=\int_{\mathfrak{G}}\mathlarger{\nabla}_{a}^{(x)}\mathlarger{\nabla}_{b}^{(y)}
{\mathbf{E}}\bigg[{{\mathscr{Z}}}_{I}\otimes{\mathscr{Z}}_{J}\bigg]d\mathcal{V}(x)\nonumber\\&
=\int_{\mathfrak{G}}\mathlarger{\nabla}_{a}^{(x)}\mathlarger{\nabla}_{b}^{(y)}{\mathbf{E}}\left[\sum_{I=1}^{\infty}\sum_{I=1}^{\infty}{\mathrm{Z}_{I}}^{1/2}
{\mathrm{Z}_{I}^{1/2}}f_{I}(x)f_{J}(y)\bigg({\mathscr{Z}}_{I}\otimes{\mathscr{Z}}_{J}\bigg)\right]d\mathcal{V}(x)\nonumber\\&
=\int_{\mathfrak{G}}\mathlarger{\nabla}_{a}^{(x)}\mathlarger{\nabla}_{b}^{(y)}\sum_{I=1}^{\infty}\sum_{I=1}^{\infty}{\mathrm{Z}_{I}}^{1/2}\mathrm{Z}^{1/2}_{I}
f_{I}(x)f_{I}(y)\bm{\mathbf{E}}\left[{\mathscr{Z}}_{I}\otimes{\mathscr{Z}}_{J}\right]d\mathcal{V}(x)\nonumber\\&
=\int_{\mathfrak{G}}\mathlarger{\nabla}_{a}^{(x)}\mathlarger{\nabla}_{b}^{(y)}\sum_{I=1}^{\infty}\sum_{I=1}^{\infty}{\mathrm{Z}_{I}}^{1/2}
{\mathrm{Z}_{I}}^{1/2}f_{I}(x)f_{J}(y)\delta_{IJ}d\mathcal{V}(x)\nonumber\\&=\int_{\mathfrak{G}}\sum_{I=1}^{\infty}\sum_{J=1}^{\infty}{\mathrm{Z}_{I}}^{1/2}
{\mathrm{Z}_{I}^{1/2}}\mathlarger{\nabla}_{a}^{(x)}f_{I}(x)\mathlarger{\nabla}_{b}^{(y)}f_{J}(y)\delta_{IJ}d\mathcal{V}(x)\nonumber\\&
=\sum_{I=1}^{\infty}\sum_{J=1}^{\infty}{\mathrm{Z}_{I}}^{1/2}{\mathrm{Z}_{I}^{1/2}}\int_{\mathfrak{G}}\mathlarger{\nabla}_{a}^{(x)}f_{I}(x)\mathlarger{\nabla}_{b}^{(y)}f_{J}(y)
\delta_{IJ}d\mathcal{V}(x)\nonumber\\&=\sum_{I=1}^{\infty}\mathrm{Z}_{I}\int_{\mathfrak{G}}\mathlarger{\nabla}_{a}^{(x)}f_{I}(x)\mathlarger{\nabla}_{b}^{(y)}f_{I}(y)
d\mathcal{V}(x)
\end{align}
Then
\begin{align}
&\lim_{y\rightarrow x}\int_{\mathfrak{G}}\mathlarger{\nabla}_{a}^{(x)}\mathlarger{\nabla}_{b}^{(y)}K(x,y;\lambda)d\mathcal{V}(x)=\lim_{y\rightarrow x}\sum_{I=1}^{\infty}{\mathrm{Z}_{I}}\int_{\mathfrak{G}}\mathlarger{\nabla}_{a}^{(x)}f_{I}(x)\mathlarger{\nabla}_{b}^{(y)}f_{I}(y)d\mathcal{V}(x)\nonumber\\&
\equiv \sum_{I=1}^{\infty}{\mathrm{Z}_{I}}\int_{\mathfrak{G}}\mathlarger{\nabla}_{b}^{(x)}f_{I}(x)\mathlarger{\nabla}^{b}{(x)}f_{I}(y)d\mathcal{V}(x)>0
\end{align}
and the $\mathrm{Z}_{I}$ are positive eigenvalues so (2.76) holds.
\end{proof}
The following basic preliminary theorem from functional analysis is given
\begin{thm}(\textbf{Dominated Convergence Theorem}
Let $(x,y)\in{\mathfrak{G}}$and let $g:{\mathfrak{G}}\otimes{\mathfrak{G}}\rightarrow \mathbb{R}$, that is $g(|x-y|)$ be a function of the
separation $|x-y|$. The integral $\int_{\mathfrak{G}}g(|x-y|)d\mathcal{V}(x)$ exists and is finite. Suppose the following hold for $g(|x-y|)$:
\begin{enumerate}
\item $g(|x-y|)$ is bounded by an integrable function $\phi(x)$ for all $(x,y)\in{\mathfrak{G}}$ such that $g(|x-y|)\le \phi(|x|)$
\item The function $g(|x-y|)$ converges uniformly to a constant $\mathrm{B}$ such that
\begin{align}
\lim_{y\rightarrow x}g(|x-y|)=\mathrm{B}>0
\end{align}
\end{enumerate}
Then one can take $\phi(|x|)=B$ and the dominated convergence theorem ensures that the integral and limit are interchangable
\begin{align}
&\lim_{y\rightarrow x}\int_{\mathfrak{G}}g(|y-x|)d\mathcal{V}(x)\equiv \int_{\mathfrak{G}}\lim_{y\rightarrow x}g(|y-x|)d\mathcal{V}(x)\nonumber\\&=\int_{\mathfrak{G}}\mathrm{B} d\mathcal{V}(x)
=\mathrm{B} vol(\mathfrak{G})
\end{align}
\end{thm}
\begin{thm}
From Lemma (2.13), if $K(x,y;\lambda)$ is a homogenous, isotropic and regulated kernel existing for all $(x,y)\in\mathfrak{G}$ such that
\begin{align}
\lim_{y\rightarrow x}\int_{\mathfrak{G}}\mathlarger{\nabla}_{a}^{(x)}\mathlarger{\nabla}^{a}_{(y)}K(x,y;\lambda)d\mathcal{V}(x)>0
\end{align}
then from Lemma 2.24 it follows that one always has
\begin{align}
\int_{\mathfrak{G}}\mathlarger{\nabla}_{a}f_{I}(x)\mathlarger{\nabla}^{a}f_{I}(x)d\mathcal{V}(x)>0
\end{align}
Then:
\begin{enumerate}
\item This holds for the isotropic Gaussian kernel so that
\begin{align}
&\lim_{y\rightarrow x}\int_{\mathfrak{G}}\mathlarger{\nabla}_{a}^{(x)}\mathlarger{\nabla}^{a}_{(y)}K_{G}(x,y;\lambda)d\mathcal{V}(x)=\lim_{y\rightarrow x}\int_{\mathfrak{G}}\mathlarger{\nabla}_{a}^{(x)}\mathlarger{\nabla}^{a}_{(y)}\exp\left(-\frac{|x-y|^{2}}{\lambda^{2}}\right)
d\mathcal{V}(x)>0
\end{align}
for all $(x,y)\in{\mathfrak{G}}$.
\item This also holds for the rational quadratic kernel so that
\begin{align}
&\lim_{y\rightarrow x}\int_{\mathfrak{G}}\mathlarger{\nabla}_{a}^{(x)}\mathlarger{\nabla}^{a}_{(y)}K_{G}(x,y;\lambda)d\mathcal{V}(x)=\lim_{y\rightarrow x}\int_{\mathfrak{G}}\mathlarger{\nabla}_{a}^{(x)}\mathlarger{\nabla}^{a}_{(y)}\frac{1}{\left(1-\frac{|x-y|^{2}}{2\alpha\lambda^{2}}\right)^{\alpha}}d\mathcal{V}(x)>0
\end{align}
for all $(x,y)\in{\mathfrak{G}}$
\end{enumerate}
\end{thm}
The proof is given in Appendix B.
\begin{lem}
If $f_{I}(x)$ and $\mathrm{Z}_{I}$ are eigenfunctions and eigenvalues of the Laplace operator under Dirichlet boundary conditions on a
bounded domain ${\mathfrak{G}}$ then $\mathlarger{\Delta} f_{I}(x)=-\mathrm{Z}_{I}f_{I}(x)$, for all $x\in {\mathfrak{G}},a\in\mathbb{Z}$ with
$f_{I}(x)=0$ for all $x\in\partial{\mathfrak{G}}$, and $\mathlarger{\Delta} f_{I}(x)=-\int_{\mathfrak{G}}K(x,y;\lambda)f_{I}(y)d\mathcal{V}(y)$ then it is always true that
\begin{align}
\int_{\mathfrak{G}}\mathlarger{\nabla}_{a}f_{I}(x)\mathlarger{\nabla}^{a}f_{I}(x)d\mathcal{V}(x)>0
\end{align}
\end{lem}
\begin{proof}
The eigenfunctions are orthonormal so $ \int_{\mathfrak{G}}f_{I}(x)f_{I}(x)d\mathcal{V}(x)=1$. Taking the first derivative
\begin{align}
&\mathlarger{\nabla}_{a}\int_{\mathfrak{G}}f_{I}(x)f_{I}(x)d\mathcal{V}(x)\nonumber=
\int_{\mathfrak{G}}\mathlarger{\nabla}_{a}f_{I}(x)f_{I}(x)d\mathcal{V}(x)\nonumber\\&+\int_{\mathfrak{G}}
f_{I}(x)\mathlarger{\nabla}_{a}f_{I}(x)d\mathcal{V}(x)=2\int_{\mathfrak{G}}\mathlarger{\nabla}_{a}f_{I}(x)f_{I}(x)d\mathcal{V}(x)=0
\end{align}
Taking the derivative again
\begin{align}
&\mathlarger{\nabla}^{a}\mathlarger{\nabla}_{a}\int_{\mathfrak{G}}f_{I}(x)f_{I}(x)d\mathcal{V}(x)
=\int_{\mathfrak{G}}\mathlarger{\nabla}^{a}\mathlarger{\nabla}_{a}f_{I}(x)f_{I}(x)d\mathcal{V}(x)+\int_{\mathfrak{G}}\mathlarger{\nabla}_{a}f_{I}(x)\mathlarger{\nabla}^{a}
f_{I}(x)d\mathcal{V}(x)\nonumber\\&=\int_{\mathfrak{G}}\mathlarger{\Delta} f_{I}(x)f_{I}(x)d\mathcal{V}(x)+\int_{{\mathfrak{G}}}\mathlarger{\nabla}_{a}
f_{I}(x)\mathlarger{\nabla}^{a}f_{I}(x)d\mathcal{V}(x)=0
\end{align}
so that
\begin{align}
\int_{\mathfrak{G}}\mathlarger{\nabla}_{a}f_{I}(x)\mathlarger{\nabla}^{a}f_{I}(x)d\mathcal{V}(x)=
-\int_{\mathfrak{G}}\mathlarger{\Delta} f_{I}(x)f_{I}(x)d\mathcal{V}(x)
\end{align}
Now since $\mathlarger{\Delta} f_{I}(x)=-\mathrm{Z}_{I}f_{I}(x)$ then
\begin{align}
\int_{\mathfrak{G}}\mathlarger{\nabla}_{a}f_{I}(x)\mathlarger{\nabla}^{a}f_{I}(x)d\mathcal{V}(x)=
\mathrm{Z}_{I}\int_{\mathfrak{G}}f_{I}(x)f_{I}(x)d\mathcal{V}(x)=\mathrm{Z}_{I}>0
\end{align}
since all eigenvalues ${\mathrm{Z}_{I}}_{a}$ are positive for the Dirichlet boundary conditions.
\end{proof}
\section{STOCHASTIC VECTORIAL FLOWS IN A DOMAIN VIA A 'WEIGHTED MIXING' OF A DETERMINISTIC VECTORIAL FIELD WITH A GAUSSIAN RANDOM FIELD}
Given the GRF ${\mathscr{T}}(x)$ as previously described, and a smooth deterministic vector field or flux/current $\Phi_{a}(x,t)$, that evolves via some PDE from initial data, then a stochastic vector field or flux $\mathscr{I}_{a}(x,t)$ can be defined by a 'weighted mixing' of the deterministic and random fields to create a new random or randomly perturbed vector field.
\begin{prop}
Let ${\mathscr{T}}(x)$ be a GRF and let $\Phi_{a}:{\mathfrak{G}}\otimes\mathbb{R}^{+}\rightarrow\mathbb{R}^{3}$ be a smooth deterministic vector field or flux existing for all $(x,t)\in\mathfrak{G}\otimes\mathbb{R}^{+}$ such that following hold:
\begin{enumerate}
\item By smooth, the first and second derivatives $\mathlarger{\nabla}_{b}\Phi_{a}(x,t)$ and $\mathlarger{\nabla}_{a}\mathlarger{\nabla}_{b}\Phi_{a}(x,t)$ exist, and deterministic is taken to mean that the field
$\Phi_{a}(x,t)$ evolves from some initial data $\Phi_{a}(x,0)$ and boundary conditions on ${{\mathfrak{G}}}$, via some linear or nonlinear PDE of the generic form
\begin{align}
&\frac{\partial}{\partial{t}}\Phi_{a}(x,t)+\mathbb{I\!L}\left[\mathlarger{\nabla}_{b},\mathlarger{\Delta}\right]\Phi_{a}(x,t)=0
\end{align}
or
\begin{align}
\frac{\partial}{\partial t}\Phi_{a}(x,t)+\mathbb{I\!N}\left[\mathlarger{\nabla}_{b},\mathlarger{\Delta},\Phi_{a}(x,t)\right]\Phi_{a}(x,t)=0
\end{align}
where $\mathbb{I\!L}$ is a linear differential operator and $\mathbb{I\!N}[\mathlarger{\nabla}_{b},\mathlarger{\Delta},\Phi_{a}]$ is a nonlinear differential operator. For example, for $\mathbb{I\!L}[\mathlarger{\nabla}_{b},\mathlarger{\Delta}]=-\mathlarger{\Delta}$ we obtain the heat equation, and for a nonlinear Burgers-type equation $\mathbb{I\!N}[\mathlarger{\nabla}_{b},\mathlarger{\Delta}]=f^{b}(x,t)\mathlarger{\nabla}_{b}\Phi_{a}(x,t)-\mathlarger{\Delta} \Phi_{a}(x,t)$.
\item The Euclidean norm is $\|\Phi_{a}(x,t)\|$ and $\mathbf{W}(\mu\|\Phi_{a}(x,t)\|)$ is an arbitrary dimensionless 'weighting' or 'modulating' function of the norm. For example, one could have an exponential form $\mathbf{W}(\mu\|\Phi_{a}(x,t)\|)=\exp(\mu\|\Phi_{a}(x,t)\|)$ or a polynomial form $\mathbf{W}(\mu\|\Phi_{a}(x,t)\|)=|\mu\|\Phi_{a}(x,t)\||^{\beta}$). This function then evolves in space and time nonlinearly as the vector field or flux $\Phi(x,t)$ 'flows' or evolves in space and time.
\item The GRF ${\mathscr{T}}(x)$ has a spectral representation via the Karhunen-Loeve theorem and has a well-behaved homogenous and regulated isotropic kernel
$K(x,y;\lambda)$.
\end{enumerate}
Then the following stochastic vector or flux  $\mathscr{I}_{a}(x,t)$ can be defined or 'engineered' such that
\begin{align}
&{\mathscr{I}}_{a}(x,t)=\Phi_{a}(x,t)+{A}\Phi_{a}(x,t)\mathbf{W}(\mu\|\Phi_{a}(x,t)\|)\sum_{I=1}^{\infty}
{\mathrm{Z}_{I}^{1/2}}f_{I}(x)\otimes\mathscr{Z}_{I}\nonumber\\&
=\Phi_{a}(x,t)\left(1+{A}\mathbf{W}(\mu\|\Phi_{a}(x,t)\|)\sum_{I=1}^{\infty}{\mathrm{Z}_{I}^{1/2}}
f_{I}(x)\otimes\mathscr{Z}_{I}\right)\nonumber\\&\equiv \Phi_{a}(x,t)+{\mathscr{F}}_{a}(x,t)
\end{align}
where ${A}$ is an arbitrary amplitude or constant. The expectation then gives the mean field as
\begin{align}
&\mathbf{E}\big[{\mathscr{I}}_{a}(x,t)\big]=\Phi_{a}(x,t)+{A}\Phi_{a}(x,t)\mathbf{W}(\mu\|\Phi_{a}(x,t)\|)
\sum_{I=1}^{\infty}{\mathrm{Z}_{I}^{1/2}}f_{I}(x)\mathbf{E}\bigg[\mathscr{Z}^{1/2}_{I}\bigg]\nonumber\\&=
\Phi_{a}(x,t)\left(1+{A}\mathbf{W}(\mu\|\Phi_{a}(x,t)\|)\sum_{I=1}^{\infty}{\mathrm{Z}^{1/2}_{I}}
f_{I}(x)\mathbf{E}\bigg[\mathscr{Z}_{I}\bigg]\right)=\Phi_{a}(x,t)
\end{align}
since $\mathbf{E}[\mathscr{Z}_{I}]=0$.
\end{prop}
This can be interpreted as a random perturbation of the smooth deterministic field by a GRF represented by a KL expansion so that
\begin{align}
&\mathscr{I}_{a}(x,t)=\underbrace{\Phi_{a}(x,t)}_{deterministic}+\underbrace{{A}\Phi_{a}(x,t)\mathbf{W}(\mu\|\Phi_{a}(x,t)\|)
\sum_{I=1}^{\infty}{\mathrm{Z}_{I}^{1/2}}f_{I}(x)\otimes{\mathscr{Z}}_{I}}_{stochastic}\nonumber\\&
=\Phi_{a}(x,t)+{\mathscr{F}}_{I}(x,t)
\end{align}
Then $\Phi_{a}(x,t)$ describes a stochastic vector flow or current within the domain ${\mathfrak{G}}$. Here ${A}$ is an amplitude and again it is emphasised that $
\mathbf{W}(\mu\|\Phi_{a}(x,t)\|)$ is a 'weighting' or modulating term which modulates the spatial variation of the perturbation across the domain at any time t. This then controls the strength or magnitude of the stochastic or random term. Hence, the random fluctuations/amplitude increases or decreases as both $\|\Phi_{a}[(x,t)\|$ and $\mathbf{W}(\mu\|\Phi_{a}(x,t)\|)$ increase or decrease. The underlying PDE controlling the flow or evolution of $\Phi_{a}(x,t)$ therefore introduces dynamics and feedback between the random and deterministic elements or contributions. This stochastic or random vector field could then potentially be applied to complex systems in physics or biology which have both a deterministic and stochastic contributions or components. In  particular, we wish to apply this to the problem of turbulent flows of incompressible viscous fluids.
\begin{rem}
Given that the field $\mathscr{T}(x)$ is Gaussian, then the field $\Phi_{a}(x,t)$ may or may not be Gaussian. This depends on the properties of the 'weighting term'
$\mathbf{W}(\mu\|\Phi_{a}(x,t)\|)$. Gaussianality is generally preserved under linear operations, but this may not be the case depending on the form of
$\mathbf{W}(\mu\|\Phi_{a}(x,t)\|)$.
which can be nonlinear.
\end{rem}
Given the stochastic vector field or current $\Phi_{a}(x,t)$ within $\mathfrak{G}$, one can evaluate the basic covariance $Cov(x,y)$, which measures how random fields at any pair $(x,y)\in\mathfrak{G}$ at time $t$ vary or co-vary within $\mathfrak{G}$ relative to their expected values.
\begin{thm}
Let $\mathscr{I}_{a}(x)$ be a stochastic vector field or current within ${\mathfrak{G}}$ as previously defined. Then the covariance of this field is given by
\begin{align}
&Cov(x,y)=\mathbf{E}\left[\mathscr{I}_{a}(x,t)\otimes\mathscr{I}_{a}(x,t)\right]
-\mathbf{E}[{\mathscr{I}}_{a}(x,t)]\mathbf{E}[{\mathscr{I}}_{a}(y,t)]\nonumber\\&={A}^{2}\Phi_{a}(x,t)\Phi_{b}(y,t)\mathbf{W}
(\mu\|\Phi_{a}(x,t)\|)\mathbf{W}(\mu\|\Phi_{a}(y,t)\|)\sum_{I=1}^{\infty}\mathrm{Z}_{I}f_{I}(x)f_{I}(y))\nonumber\\&
\equiv {A}^{2}\Phi_{a}(x,t)\Phi_{b}(y,t)\mathbf{W}(\mu\|\Phi_{a}(x,t)\|)\mathbf{W}(\mu\|\Phi_{a}(y,t)\|)K(x,y;\lambda)
\end{align}
since the kernel is given by the mercer Theorem as $K(x,y;\lambda)=\sum_{I=1}^{\infty}\mathrm{Z}_{I}f_{I}(x)f_{I}(y)$
Then for $\|x-y\|\gg \lambda$ one has $Cov(x,y)=0$ and the random fields are uncorrelated.
\end{thm}
\begin{proof}
The random vector fields at $(x,y)\in\mathfrak{G}$ are
\begin{align}
&\mathscr{I}_{a}(x,t)=\Phi_{a}(x,t)+A \Phi_{a}(x,t)\mathbf{W}(\mu\|\Phi_{a}(x,t)\|)\sum_{I=1}^{\infty}{\mathrm{Z}_{I}^{1/2}}
f_{I}(x)\otimes\mathscr{Z}_{I}\nonumber\\&\mathscr{I}_{b}(y,t)=\mathrm{\Phi}_{b}(y,t)+A \Phi_{b}(y,t)\mathbf{W}(\mu\| \Phi_{b}(y,t)\|)
\sum_{J=1}^{\infty}{\mathrm{Z}_{I}^{1/2}}f_{I}(y)\otimes\mathscr{Z}_{J}
\end{align}
Then the mean or expected values are
\begin{align}
&\mathbf{E}[\mathscr{I}_{a}(x,t)]=\Phi_{a}(x,t)+A \Phi_{a}(x,t)\mathbf{W}(\mu\|\Phi_{a}(x,t)\|)\sum_{I=1}^{\infty}{\mathrm{Z}_{I}^{1/2}}
f_{I}(x)\mathbf{E}[\mathscr{Z}_{I}]=\Phi_{a}(x,t)\nonumber\\&
\mathbf{E}[\mathscr{I}_{b}(y,t)]=\Phi_{b}(y,t)+A \Phi_{b}(y,t)\mathbf{W}(\mu\|\Phi_{b}(y,t)\|)\sum_{J=1}^{\infty}
\mathrm{Z}_{I}^{1/2}f_{I}(y)\mathbf{E}[\mathscr{Z}_{J}]=\Phi_{b}(y,t)
\end{align}
so that
$\mathbf{E}[\mathscr{I}_{a}(x,t)]\mathbf{E}[\mathscr{I}_{b}(y,t)]=\Phi_{a}(x,t)\Phi_{b}(y,t)$. Next
\begin{align}
\mathscr{I}_{a}(x,t){\otimes}&~\mathscr{I}_{b}(y,t)=\Phi_{a}(x,t)\Phi_{b}(y,t)\nonumber\\&
+{A}\Phi_{a}(x,t)\Phi_{b}(y,t)\mathbf{W}(\mu\|\Phi_{b}(y,t)\|)\sum_{J=1}^{\infty}{\mathrm{Z}_{I}^{1/2}}f_{I}(y)\otimes{\mathscr{Z}}_{J}\nonumber\\&
+{A}\Phi_{b}(y,t)\Phi_{b}(x,t)\mathbf{W}(\mu\|\Phi_{a}(x,t)\|)\sum_{I=1}^{\infty}{\mathrm{Z}_{I}^{1/2}}f_{I}(y)\otimes{\mathscr{Z}}_{I}\nonumber\\&
+{A}^{2}\Phi_{a}(x,t)\Phi_{b}(y,t)\mathbf{W}(\mu\|\Phi_{a}(x,t)\|)\mathbf{W}(\mu\|\Phi_{b}(x,t)\|)\nonumber\\&
\otimes\sum_{I=1}^{\infty}\sum_{J=1}^{\infty}{\mathrm{Z}_{I}^{1/2}}{\mathrm{Z}_{I}^{1/2}}f_{I}(x)f_{I}(y)
\bigg(\mathscr{Z}_{I}\otimes\mathscr{Z}_{J}\bigg)
\end{align}
Taking the expectation of this product
\begin{align}
\mathbf{E}\big[\mathscr{I}_{a}(x,t)&{\otimes}~\mathscr{I}_{b}(y,t)\big]=\Phi_{a}(x,t)\Phi_{b}(y,t)
+\nonumber\\&{A}\Phi_{a}(x,t)\Phi_{b}(y,t)\mathbf{W}(\mu\|\Phi_{b}(y,t)\|)\sum_{J=1}^{\infty}{\mathrm{Z}_{I}^{1/2}}f_{J}(y)\mathbf{E}
\left[\mathscr{Z}_{b}\right]\nonumber\\&+{A}\Phi_{b}(y,t)\Phi_{b}(x,t)\mathbf{W}(\mu\|\Phi_{a}(x,t)\|)\sum_{I=1}^{\infty}{\mathrm{Z}^{1/2}_{I}}f_{I}(x)
\mathbf{E}\left[\mathscr{Z}_{J}\right]\nonumber\\&+{A}^{2}\Phi_{a}(x,t)\Phi_{b}(y,t)\mathbf{W}(\mu\|\Phi_{a}(x,t)\|)\mathbf{W}(\mu\|\Phi_{b}(y,t)\|)\nonumber\\&\otimes
\sum_{I=1}^{\infty}\sum_{J=1}^{\infty}{\mathrm{Z}_{I}^{1/2}}{\mathrm{Z}_{I}^{1/2}}f_{I}(x)f_{J}(y)
\mathbf{E}\left[\bigg(\mathscr{Z}_{I}\otimes\mathscr{Z}_{J}\bigg)\right]\nonumber\\&=\Phi_{a}(x,t)\Phi_{b}(y,t)+
{A}^{2}\Phi_{a}(x,t)\Phi_{b}(y,t)\mathbf{W}(\mu\|\Phi_{a}(x,t)\|)\mathbf{W}(\mu\|\Phi_{b}(y,t))\nonumber\\&\otimes\sum_{I=1}^{\infty}\sum_{J=1}^{\infty}{\mathrm{Z}_{I}^{1/2}}
{\mathrm{Z}_{I}^{1/2}}f_{I}(x)f_{J}(y)\delta_{IJ}
\nonumber\\&=\Phi_{a}(x,t)\Phi_{b}(y,t)+{A}^{2}\Phi_{a}(x,t)\Phi_{b}(y,t)\mathbf{W}(\mu\|\Phi_{a}(x,t)\|)\mathbf{W}(\mu\|\Phi_{b}(y,t)\sum_{I=1}^{\infty}{\mathrm{Z}_{I}}
f_{I}(x)f_{I}(y)\nonumber\\&=\Phi_{a}(x,t)\Phi_{b}(y,t)+{A}^{2}\Phi_{a}(x,t)\Phi_{b}(y,t)\mathbf{W}(\mu\|\Phi_{a}(x,t)\|)
\mathbf{W}(\mu\|\Phi_{b}(y,t)\|)K(x,y;\lambda)
\end{align}
Hence
\begin{align}
&Cov(x,y;t)=\mathbf{E}[{\mathscr{I}}_{a}(x,t)\otimes{\mathscr{I}}_{b}(y,t)]
-\mathbf{E}[\mathscr{I}_{a}(x,t)]\mathbf{E}[\mathscr{I}_{b}(y,t)]\nonumber\\&={A}^{2}\Phi_{a}(x,t)\Phi_{b}(y,t)
\mathbf{W}(\mu\|\Phi_{a}(x,t)\|)\mathbf{W}(\mu\|\Phi_{b}(y,t)\|)K(x,y;\lambda)
\end{align}
so that $Cov(x,y)=0$ if $\|x-y\|\gg \lambda$.
\end{proof}
\begin{cor}
The 2nd-order moments or variance at any time $t$ then follows in the limit as $y\rightarrow x$ or equivalently $(y,t)\rightarrow (x,t)$ so that
\begin{align}
&{\mathbb{M}}_{2}(x,t)=\lim_{y\rightarrow x}\mathbf{E}\bigg[{{\mathscr{T}}_{a}(x,t)}\otimes{{\mathscr{T}}_{b}(y,t)}\bigg]\nonumber\\&
=\lim_{y\rightarrow x}{A}^{2}\Phi_{a}(x,t)\Phi_{b}(y,t)\mathbf{W}(\mu\|\Phi_{a}(x,t)\|)\mathbf{W}(\mu\|\Phi_{b}(y,t)\|)
K(x,y;\lambda)\nonumber\\&
={A}^{2}\Phi_{a}(x,t)\Phi_{b}(x,t)\mathbf{W}(\mu\|\Phi_{b}(x,t)\|)\mathbf{W}(\mu\|\Phi_{b}(x,t)\|)
\Phi\nonumber\\&
={A}^{2}\Phi_{a}(x,t)\Phi_{b}(x,t)\mathbf{W}(\mu\|\Phi_{b}(x,t)\|)\mathbf{W}(\mu\|\Phi_{b}(x,t)\|)
var(x)
\end{align}
so that the second order moments or variance of the field $\Phi_{a}(x,t)$ depends on the variance ${var}(x)$ of the Gaussian random field $\mathscr{T}(x)$.
\end{cor}
The random field $\mathscr{I}_{a}(x,t)$ is not homogenous and isotropic unless the underlying deterministic vector field is constant.
\begin{lem}
The field $\mathscr{I}_{a}(x,t)$ is not homogenous and isotropic for a generic current or flow $\Phi_{a}(x,t)$ since
\begin{align}
Cov(x+\lambda,y+\lambda;t)={\bm{{\bm{\mathbf{E}}}}}\left[\mathscr{I}_{a}(x+\lambda)\otimes\mathscr{I}_{b}(y+\lambda)]\right]\ne {\bm{{\bm{\mathbf{E}}}}}
\left[\mathscr{I}_{a}(x,t)\otimes \mathscr{I}_{b}(y,t)]\right]
\end{align}
where $(x+\ell)\equiv(x+\ell,t)$. However, the kernel for $\mathscr{T}(x)$ is homogenous and isotropic and stationary so that $K(x+\ell,y+\ell;\lambda)=K(x,y;\lambda)$ for all $\lambda>0$. This is true for the Gaussian kernel. Therefore, if the underlying deterministic vector field $\Phi_{a}(x,t)=\Phi_{a}$ is a constant flow for all $(x,t)=(x,t)\in\mathfrak{G}\otimes\mathbb{R}^{+}$ then the flow is homogenous and isotropic since
\begin{align}
&\mathbf{E}\left[\mathscr{I}_{a}(x+\ell)\otimes \mathscr{I}_{b}(y+\ell)\right]\nonumber\\&
= {A}^{2}\Phi_{a}\Phi_{b}+|{A}|^{2}\Phi_{a}\Phi_{b}\mathbf{W}(\mu\|\Phi_{a}\|)\mathbf{W}(\mu\|\Phi_{b}\|)
\sum_{I=1}^{\infty}\mathrm{Z}_{I}f_{I}(x+\ell)f_{I}(y+\ell)\nonumber\\&
\equiv {A}^{2}\Phi_{a}\Phi_{b}+|{A}|^{2}\Phi_{a}\Phi_{b}\mathbf{W}(\mu\|\Phi_{a}\|)\mathbf{W}(\mu\|\Phi_{b}\|)
K(x+\ell,y+\ell;\xi)\nonumber\\&\equiv {A}^{2}\Phi_{a}\Phi_{b}+|{A}|^{2}\Phi_{a}\Phi_{b}\mathbf{W}(\mu\|\Phi_{a}\|)\mathbf{W}(\mu\|\Phi_{b}\|)
K(x,y;\lambda)=\mathbf{E}\left[\mathscr{I}_{a}(x,t)\otimes\mathscr{I}_{b}(y,t)\right]
\end{align}
\end{lem}
Holder and Lipschitz continuity can also be defined for this random field
\begin{lem}
Let $\Phi_{a}(x,t)$ be a smooth flow or current within a domain ${\mathfrak{G}}\subset\mathbb{R}^{3}$, evolving by some linear or nonlinear PDE equation from some initial data $\Phi_{a}(X_{o})\equiv \Phi_{a}(x,o)=\Phi_{a}(x)$. Then $\Phi_{a}(x,t)$ is Lipschitz continuous if for all $(x,y)\in{\mathfrak{G}}$, $\exists C$ such that $|\Phi_{a}(x,t)-\Phi_{a}(y,t)|\le C|x-y| $. And $\Phi_{a}(x,t)$ is Holder continuous if for all $(x,y)\in{\mathfrak{G}}$, $\exists D>0$ and $\exists D>0$ such that
$ |\Phi_{a}(x,t)-\Phi_{a}(y,t)|\le D|x-y|^{\alpha}$ with $\alpha\in(0,1]$. Consider now the stochastic vector field $\mathscr{I}_{a}(x,t)$ such that
\begin{align}
\mathscr{I}_{a}(x,t)=\Phi_{a}(x,t)+{A}\Phi_{a}(x,t)\mathbf{W}(\mu\|\Phi_{a}(x,t)\|)\sum_{I=1}^{\infty}
{\mathrm{Z}^{1/2}_{I}}\Phi_{I}(x)\otimes{\mathscr{Z}}_{I}
\end{align}
Then:
\begin{enumerate}
\item The stochastic vector flow $\Phi_{a}(x,t)$ is Lipschitz continuous if for all $(x,y)\in{\mathfrak{G}}$, $\exists \mathrm{Z}$ such that
\begin{align}
\mathbf{E}[|\mathscr{I}_{a}(x,t)-\mathscr{I}_{a}(y,t)|]\le\mathrm{B} |x-y|
\end{align}
and $|\Phi_{a}(x,t)-\Phi_{a}(y,t)|\le |x-y|$.
\item The random flow or current $\bm{\mathscr{I}}_{a}(x,t)$ is Holder continuous if for all $(x,y)\in{\mathfrak{G}}$, $\exists B>0$ and $\exists\alpha>0$ such that
\begin{align}
\mathbf{E}\left[|{\mathscr{I}}_{a}(x,t)-{\mathscr{I}}_{a}(y,t)|\right]\le \mathrm{B}|x-y|^{\alpha}
\end{align}
and $|\Phi_{a}(x,t)-\Phi_{a}(y,t)|\le \mathrm{B}|x-y|^{\alpha}$.
\end{enumerate}
\end{lem}
\begin{proof}
To prove (1)
\begin{align}
{\mathbf{E}}\big[\big|{\mathscr{I}}_{a}(x,t)&-{\mathscr{I}}_{a}(y,t)\big|\big]=\bm{\mathbf{E}}\bigg[\Phi_{a}(x,t)+{A} \Phi_{a}(x,t){\mathbf{W}}(\mu\|\Phi_{a}(x,t)\|)\sum_{I=1}^{\infty}{\mathrm{Z}_{I}^{1/2}}f_{I}(x)\otimes\mathscr{Z}_{I}\nonumber\\&-\Phi_{a}(y,t)-{A} \Phi_{a}(y,t)\mathbf{W}(\mu\|\Phi_{a}(x,t)\|)\sum_{I=1}^{\infty}{\mathrm{Z}_{I}^{1/2}}f_{I}(y)\otimes\mathscr{Z}_{I}\bigg]\nonumber\\&
=|\Phi_{a}(x,t)+{A} \Phi_{a}(x,t)\mathbf{W}(\mu\|\Phi_{a}(x,t)\|)\sum_{I=1}^{\infty}{\mathrm{Z}_{I}^{1/2}}f_{I}(x)\bm{\mathbf{E}}\big[\mathscr{Z}_{I}\big]\nonumber\\&
-|\Phi_{a}(y,t)-{A}\Phi_{a}(y,t)\mathbf{W}(\mu\|\Phi_{a}(x,t)\|)\sum_{I=1}^{\infty}{\mathrm{Z}_{I}^{1/2}}f_{I}(y)\bm{\mathbf{E}}
\big[\mathscr{Z}_{I}\big]\nonumber\\&=|\Phi_{a}(x,t)-\Phi_{a}(y,t)|\le \mathrm{Z}|x-y|
\end{align}
since $\mathbf{E}[\mathscr{Z}_{I}]=0$ for the independent Gaussian random variables ${\mathscr{Z}}_{I}$. Similarly for the Holder continuity condition.
\end{proof}
\begin{lem}
Given the stochastic vector field  ${\mathscr{I}}_{a}(x,t)$ then the moments are $\mathbf{E}[|{\mathscr{I}}_{a}(x,t)|^{p}]$. Then for all $p\ge2$ one has
\begin{align}
&{\mathbb{M}}_{p}(x,t)=\mathbf{E}\left[|{\mathscr{I}}_{a}(x,t)|^{p}\right]\nonumber\\&\le 2^{p}|\Phi_{a}(x,t)|^{P}(1+{A}^{p}\mathbf{W}(\mu\|\Phi_{a}(x,t)\|))^{p}\sum_{I=1}^{\infty}{\mathrm{Z}_{I}}^{p/2}|f_{I}(x)|^{p}
\left(\frac{p}{2}-1\right)!!
\end{align}
It follows that if $|\Phi_{a}(x,t)|^{p}<\infty$ for all $(x,t)\in{\mathfrak{G}}\otimes\mathbb{R}^{+}$ then $\mathbf{E}[|\Phi_{a}(x,t)|^{p}]<\infty$, and assuming that the series converges.
\end{lem}
\begin{proof}
\begin{align}
|{\mathscr{I}}_{a}(x,t)|^{p}=\left|\Phi_{a}(x,t)+{A}\Phi_{a}(x,t)\mathbf{W}(\mu\|\Phi_{a}(x,t)\|)\sum_{I=1}^{\infty}
{\mathrm{Z}_{I}^{1/2}}f_{I}(x)\otimes\mathscr{Z}_{I}\right|^{p}
\end{align}
Applying the basic inequality $|a+b|^{p}\le 2^{p-1}|a|^{p}+2^{p-1}|b|^{P}$
\begin{align}
&|{\mathscr{I}}_{a}(x,t)|^{p}\le 2^{p-1}\left|\Phi_{a}(x,t)\right|^{p}+2^{p-1} \left| {A}\Phi_{a}(x,t)\mathbf{W}(\mu\|\Phi_{a}(x,t)\|)\sum_{I=1}^{\infty}\mathrm{Z}_{I}^{1/2}f_{I}(x)\otimes{\mathscr{Z}}_{I}\right|^{p}\nonumber\\&
\le 2^{p-1}\left|\Phi_{a}(x,t)\right|^{p}+2^{p-1}|{A}^{p}|\Phi_{a}(x,t)|^{p}|\mathbf{W}(\mu\|\Phi_{a}(x,t)\|)|^{p}\left(\sum_{I=1}^{\infty}
\mathrm{Z}_{I}^{1/2}f_{I}(x)\otimes\mathscr{Z}_{I}\right)^{p}\nonumber\\&
=2^{p-1}\left|\Phi_{a}(x,t)\right|^{p}+2^{P-1}|{A}^{p}|\Phi_{a}(x,t)|^{p}|\mathbf{W}(\mu\|\Phi_{a}(x,t)\|)|^{p}\nonumber\\&\otimes
\underbrace{\left(\sum_{I=1}^{\infty}{\mathrm{Z}_{I}^{1/2}}f_{I}(x)\otimes{\mathscr{Z}}_{I}\right)\otimes...\otimes\left(\sum_{I=1}^{\infty}
{\mathrm{Z}_{I}^{1/2}}f_{I}(x)\otimes\mathscr{Z}_{I}\right)}_{p~times}
\nonumber\\&=2^{p-1}\left|\Phi_{a}(x,t)\right|^{p}+2^{p-1}|{A}^{p}|\Phi_{a}(x,t)|^{p}|\mathbf{W}(\mu\|\Phi_{a}(x,t)\|)|^{p}\nonumber\\&\otimes
\sum_{a_{1}=1}^{\infty}\sum_{a_{2}=1}^{\infty}...
\sum_{a_{p}=1}^{\infty}\left(\prod_{\eta=1}^{\infty}\mathrm{Z}_{I}^{1/2}f_{I}(x)\otimes\mathscr{Z}_{I}\right)\nonumber\\&
=2^{p-1}\left|\Phi_{a}(x,t)\right|^{p}+2^{p-1}|{A}^{p}|\Phi_{a}(x,t)|^{p}|\mathbf{W}(\mu\|\Phi_{a}(x,t)\|)|^{p}\sum_{I=1}^{\infty}
\left(\prod_{\eta=1}^{p}{\mathrm{Z}_{I}^{1/2}}f_{I}(x)\otimes\mathscr{Z}_{I}\right)\nonumber\\&
=2^{p-1}\left|\Phi_{a}(x,t)\right|^{p}+2^{p-1}|{A}^{p}|\Phi_{a}(x,t)|^{p}|\mathbf{W}(\mu\|\Phi_{a}(x,t)\|)|^{p}
\sum_{I=1}^{\infty}\bigg({\mathrm{Z}_{I}^{1/2}}f_{I}(x)\otimes\mathscr{Z}_{I}\bigg)^{p}\nonumber\\&
=2^{p-1}\left|\Phi_{a}(x,t)\right|^{p}+2^{p-1}|{A}^{p}|\Phi_{a}(x,t)|^{p}|\mathbf{W}(\mu\|\Phi_{a}(x,t)\|)|^{p}
\sum_{I=1}^{\infty}\mathrm{Z}_{I}^{p/2}f_{I}^{p}(x)\otimes\mathscr{Z}_{I}^{p}
\end{align}
where it has been assumed that $a_{\eta}=a$ for all $\eta$. Then $a_{1}=a_{2}=...=a_{p}=a$. Taking the expectation
\begin{align}
&\mathbb{M}_{p}(x,t)=\mathbf{E}\left[|{\mathscr{I}}_{a}(x,t)|^{p}\right]\nonumber\\&=2^{p-1}\left|\Phi_{a}(x,t)\right|^{p}+2^{p-1}|{A}^{p}|\Phi_{a}(x,t)|^{p}|
\mathbf{W}(\mu\|\Phi_{a}(x,t)\|)|^{p}\sum_{I=1}^{\infty}\mathrm{Z}_{I}^{p/2}f_{I}^{p}(x)\mathbf{E}[\mathscr{Z}_{I}^{p}]
\end{align}
Now the ${\mathscr{Z}}_{I}$ are normal standard Gaussian random variables with zero mean
$\mathbf{E}[{\mathscr{Z}}_{I}]=0$ and ${\bm{{\bm{\mathbf{E}}}}}[\mathscr{Z}_{I}\otimes\mathscr{Z}_{I}]=1$. The moments for a standard normal distribution are given in terms of the double factorial as $ {\bm{{\bm{\mathbf{E}}}}}[{\mathscr{Z}}_{I}^{p}]=\left(\frac{p}{2}-1\right)!!,~~p~even $ and zero for all odd values of $p$. Here $6!!=6.4.2$ and $5!!=5,3,1$ for example. Hence
\begin{align}
&\mathbf{E}\left[|{\mathscr{I}}_{a}(x,t)|^{p}\right]\le 2^{p-1}\left|\Phi_{a}(x,t)\right|^{p}\nonumber\\&+2^{p-1}|{A}^{p}|\Phi_{a}(x,t)|^{p}|\mathbf{W}(\mu\|\Phi_{a}(x,t)\|)|^{p}
\sum_{I=1}^{\infty}\mathrm{Z}_{I}^{p/2}f_{I}^{p}(x)\left(\frac{p}{2}-1\right)!!
\end{align}
\end{proof}
\begin{cor}
In the limit that $\Phi_{a}(x,t)\rightarrow \Phi_{a}=const$
\begin{align}
&\mathbf{W}((x,t;p)={\bm{{\bm{\mathbf{E}}}}}[|\bm{\mathscr{U}}_{a}(x,t)|^{P}]\nonumber\\&\le 2^{p-1}\left|{A}_{a}\right|^{p}+2^{P-1}|{A}^{p}|\Phi_{a}|^{p}|\mathbf{W}(\mu\|\Phi_{a}\|)|^{p}
\sum_{I=1}^{\infty}\mathrm{Z}_{I}^{p/2}f_{I}^{p}(x)\left(\frac{p}{2}-1\right)!!
\end{align}
\end{cor}
Similarly, one could do a brute-force calculation to estimate $\mathfrak{G}(x,t;p)={\mathbf{E}}[|\mathlarger{\nabla}_{b}{\Phi_{a}(x,t)}|^{p}]$, the $p^{th}$-order moments of the gradient of the random field. However, one can consider the case when the underlying vector field or current is constant in space and time so that $\Phi_{a}(x,t)=\Phi_{a}=const.$
\begin{lem}
Given the stochastic current or random vector field
\begin{align}
{\mathscr{I}}_{a}(x,t)=\Phi_{a}+A\Phi_{a}\mathbf{W}(\mu\|\Phi_{a}\|)
\sum_{I=1}^{\infty}\mathrm{Z}_{I}^{1/2}f_{I}(x)\otimes\mathscr{Z}_{I}
\end{align}
then the moments of the gradients $\mathbb{G}(x,t;p)$ are bounded and finite for all $(x,t)\in\mathfrak{G}\otimes\mathbb{R}^{+}$ such that
\begin{align}
\mathbb{G}_{p}(x,t)=\mathbf{E}[|\mathlarger{\nabla}_{b}{\mathscr{I}}_{a}(x,t)|^{p}]\le 2^{p-1}\Phi_{a}^{p}(1+(\mathbf{W}(\mu\|\Phi_{a}\|)^{p})\sum_{I=1}^{\infty}\mathrm{Z}_{I}^{p/2}|\mathlarger{\nabla}_{b}
f_{I}(x)|^{p}\left(\frac{p}{2}-1\right)!!
\end{align}
\end{lem}
\begin{proof}
The gradient of the field is
\begin{align}
\mathlarger{\nabla}_{b}{\mathscr{I}}_{a}(x,t)={A}\Phi_{a}\mathbf{W}(\mu\|\Phi_{a}\|)
\sum_{I=1}^{\infty}{\mathrm{Z}^{1/2}_{I}}\mathlarger{\nabla}_{b}f_{I}(x)\otimes\mathscr{Z}_{I}
\end{align}
since $\mathlarger{\nabla}_{b}\Phi_{a}=0$. Then
\begin{align}
&|\mathlarger{\nabla}_{b}{\mathscr{I}}_{a}(x,t)|^{p}=\left|\Phi_{a}+{A}\Phi_{a}\mathbf{W}(\mu\|\Phi_{a}\|)\sum_{I=1}^{\infty}
{\mathrm{Z}_{I}^{1/2}}\mathlarger{\nabla}_{b}f_{I}(x)\otimes\mathscr{Z}_{I}\right|^{p}\nonumber\\&\le 2^{p-1}|C_{a}|^{p}+2^{p-1}\left|{A}\Phi_{a}\mathbf{W}(\mu\|\Phi_{a}\|)\sum_{I=1}^{\infty}{\mathrm{Z}_{I}^{1/2}}\mathlarger{\nabla}_{b}f_{I}(x)\otimes
\mathscr{Z}_{I}\right|^{p}\nonumber\\&=2^{p-1}|\Phi_{a}|^{p}+2^{p-1}{A}^{p}\Phi_{a}|^{p}\mathbf{W}(\mu\|\Phi_{a}\|))^{p}
\sum_{I=1}^{\infty}{\mathrm{Z}_{I}}^{p/2}|\mathlarger{\nabla}_{b}f_{I}(x)|^{p}\otimes|\mathscr{Z}_{I}|^{p}
\end{align}
Taking the expectation then gives
\begin{align}
&\mathbb{G}_{p}(x,t)=\mathbf{E}[|\mathlarger{\nabla}_{b}{\mathscr{I}}_{a}(x,t)|^{p}]\nonumber\\&\le 2^{p-1}\Phi_{a}^{p}(1+{A}^{p}(\mathbf{W}(\mu \|\Phi_{a}\|))^{p}\sum_{I=1}^{\infty}{\mathrm{Z}_{I}}^{p/2}|\mathlarger{\nabla}_{b}f_{I}(x)|^{p}{\bm{{\bm{\mathbf{E}}}}}[|\mathscr{Z}_{I}|^{p}]\nonumber\\&
=2^{p-1}\Phi_{a}^{p}(1+{A}^{P}\mathbf{W}(\mu\|\Phi_{a}\|))^{p}
\sum_{I=1}^{\infty}{\mathrm{Z}_{I}}^{p/2}|\mathlarger{\nabla}_{b}f_{I}(x)|^{p}\left(\frac{p}{2}-1\right)!!
\end{align}
which is finite and bounded for all $x\in\mathfrak{G}$.
\end{proof}
\begin{cor}
When $p=2$
\begin{align}
&\mathbb{G}_{2}(x,t)=\mathbf{E}[|\mathlarger{\nabla}_{b}{\mathscr{I}}_{a}(x,t)|^{2}]\le
2 \Phi_{a}^{p}(1+A^{2}\mathrm{Z}_{I}(\mu\|\Phi_{a}\|))^{2}\sum_{I=1}^{\infty}{\mathrm{Z}_{I}}|\mathlarger{\nabla}_{b}f_{I}(x)|^{2}
\end{align}
since $0!!=1$.
\end{cor}
\subsection{Canonical metric and the structure function}
The canonical metric--or structure function--for the vector field ${\mathscr{I}}_{a}(x,t)$ are now defined. In the Kolmogorov turbulence theory, the third-order structure function in 3 dimensions leads to the famous 4/5 scaling law and the second order structure function to the 2/3-law \textbf{[1-5]}. However, within the theory of random fields, the second-order structure function is also equivalent of the square of the canonical metric for the random vector fields (Ref Adler.)
\begin{defn}
Given the GRVF ${\mathscr{I}}_{a}(x,t)$ for all $(x,t)\in{\mathfrak{G}}\otimes \mathbb{R}^{+} $, the \textbf{canonical metric} is defined as \textbf{[58]}
\begin{align}
d_{2}(x,y;t)=\left(\mathbf{E}\left[\bigg|{\mathscr{I}}_{a}(y,t)
-{\mathscr{I}}_{a}(x,t)\bigg|^{2}\right]\right)^{1/2}\equiv
\left(\mathbf{E}\left[\bigg|{\mathscr{I}}_{a}(y,t)-{\mathscr{I}}_{a}(x,t)\bigg|^{2}\right]\right)^{1/2}
\end{align}
The \textbf{structure functions} $\mathbb{S}[{\mathscr{I}}]$ of ${\mathscr{I}}_{a}(x,t)$ are then equivalent to the square of the canonical metric
\begin{align}
{\mathbb{S}}_{2}[x,y]\equiv {\mathrm{d}}_{2}^{2}(x,y)=\mathbf{E}\left[\big|{\mathscr{I}}_{a}(y,t)-{\mathscr{I}}_{a}(x,t)\big|^{2}\right]
\end{align}
If $\mathbb{S}_{2}$ obeys a scaling law over some range of length scales $\ell=\|y-x\|\le L $ then one expects
\begin{align}
{\mathbb{S}}_{2}[x,y]=\mathbf{E}\left[\big|{\mathscr{I}}_{a}(y,t)-{\mathscr{I}}_{a}(x,t)\big|^{2}\right]
\sim C|y-x|^{\zeta}=C\ell^{\zeta}
\end{align}
where $\zeta>0$ is some (fractional) power. If $y=x+{\ell}$ with ${\ell}\ll {L}$ and $vol(\mathfrak{G})\sim {L}^{3}$, then
\begin{align}
{\mathbb{S}}_{2}[\ell]\equiv d_{2}^{2}({x},{x}+{\ell})=\mathbf{E}\left[\big|
{\mathscr{I}}_{a}(x+{\ell},t)-{\mathscr{I}}_{a}(x,t)\big|^{2}\right]\nonumber\\
=\mathbf{E}\left[\big|{\mathscr{I}}_{a}(x+{\ell},t)\big|^{2}\right]
-2\mathbf{E}\left[\big|{\mathscr{I}}_{a}(x+{\ell},t){\mathscr{I}}_{a}(x,t)\big|\right]
+\mathbf{E}\left[\big|{\mathscr{I}}_{a}({x},t)\big|^{2}\right]
\end{align}
If $\mathbb{S}_{2}[\ell]$ obeys a scaling law then one expects
\begin{align}
{\mathbb{S}}_{2}[x+\ell,x]]=\mathbf{E}\left[\big|{\mathscr{I}}_{a}(x+\ell,t)-{\mathscr{I}}_{a}(x,t)\big|^{2}\right]\sim C\ell^{\zeta}
\end{align}
for some constant $C$ and power $\zeta$.
\end{defn}
\begin{cor}
If $\bm{\ell}=0$ or ${x}=y$ then ${\mathbb{S}}_{2}[\ell=0]={{d}}_{2}(x,x)=0$
\end{cor}
It will be convenient always to assume that $\mathbb{S}_{2}$ is continuous so that (keeping t fixed)
\begin{align}
&\lim_{y\rightarrow x}d_{2}^{2}(x,y)=\lim_{y\rightarrow x}\mathbf{E}\left[\bigg|{\mathscr{I}}_{a}(y,t)-{\mathscr{I}}_{a}(x,t)\bigg|^{2}\right]
=\mathbf{E}\left[\lim_{y\rightarrow x}\big|{\mathscr{I}}_{a}(y,t)-{\mathscr{I}}_{a}(x,t)\big|^{2}\right]
\end{align}
which is equivalent to the condition $ \bm{\mathbb{P}}\big[\lim_{y\rightarrow x}\big|\mathscr{I}_{a}(y,t)-\mathscr{I}_{a}({x},t)\big|=0,~\forall~{(x,y)}\in\mathfrak{G}\big]=1 $
which is also consistent with a scaling law. However, attempting to recover the 2/3-scaling law may be considered in a future article.
\subsection{Sobolov norms for stochastic vector fields $\mathscr{I}_{a}(x,t)$}
The Sobolov norms of the SVF are now considered up to order $s=2$. The size of these norms will be related to smoothness or roughness of the random vectorial field
$\mathscr{I}_{a}(x,t)$, as well as differentiability. The norms will also be shown to be dependent on the variance $var(x)$ of the underlying GRF $\mathscr{T}(x)$ To simplify matters, the norms will be computed for the case when the underlying or deterministic vector field is constant so that $\Phi_{a}(x,t)=\Phi_{a}$ and therefore $\mathlarger{\nabla}_{a}\Phi_{a}=0$.
\begin{prop}
Let $\Phi_{a}(x,t)=\Phi_{a}(x,t)=\Phi_{a}$ so that now
\begin{align}
\mathscr{I}_{a}(x,t)=\Phi_{a}+A\Phi_{a}\mathbf{W}\big(\mu\|\Phi_{a}\|_{\mathrm{E}(\mathfrak{G}])}\big)\sum_{I=1}^{\infty}{\mathrm{Z}_{I}}^{1/2}
f_{I}(x)\otimes\mathscr{Z}_{I}
\end{align}
The derivatives to first and second order now depend only on the derivatives of the eigenfunctions $f_{I}$ so that
\begin{align}
&\mathlarger{\nabla}_{a}\mathscr{I}_{a}(x,t)=A\Phi_{a}\mathbf{W}\big(\mu\|\Phi_{a}\|_{E(\mathfrak{G})}\big)\sum_{I=1}^{\infty}{\mathrm{Z}_{I}}^{1/2}
\mathlarger{\nabla}_{a}f_{I}(x)\otimes\mathscr{Z}_{I}\\&
\mathlarger{\nabla}_{a}\mathlarger{\nabla}_{b}\mathscr{I}_{a}(x,t)={A}{\Phi}_{a}\mathbf{W}\big(\mu\|{\Phi}_{a}\|_{\mathrm{E}(\mathfrak{G})}\big)\sum_{I=1}^{\infty}
{\mathrm{Z}_{I}}^{1/2}\mathlarger{\nabla}_{a}\mathlarger{\nabla}_{b}f_{I}(x)\otimes\mathscr{Z}_{I}\\&
\mathlarger{\Delta}\mathscr{I}_{a}(x,t)={A}{\Phi}_{a}\mathbf{W}\big(\mu\|{\Phi}_{I}\|_{\mathrm{E}(\mathfrak{G})}\big)\sum_{I=1}^{\infty}
{\mathrm{Z}_{I}}^{1/2}\mathlarger{\Delta} f_{I}(x)\otimes\mathscr{Z}_{I}
\end{align}
The generic Sobolov norm to order s is then
\begin{align}
\left\|\mathscr{I}_{a}\right\|^{2}_{H^{s}(\mathfrak{G})}=\sum_{|\alpha|\le s}\left\|\mathlarger{\nabla}_{a}^{\alpha}
\mathscr{I}_{a}(x,t)\right\|^{2}_{L_{2}(\mathfrak{G})}
\end{align}
with expectation
\begin{align}
&\left\|\!\left\|\mathscr{I}_{a}\right\|\!\right\|^{2}_{H^{s}(\mathfrak{G})}\equiv
\mathbf{E}\left[\left\|\mathscr{I}_{a}\right\|^{2}_{H^{s}(\mathfrak{G})}\right]\nonumber\\&=\mathbf{E}\left[\sum_{|\alpha|\le s}\left\|\mathlarger{\nabla}_{a}^{\alpha}
\mathscr{I}_{a}(x,t)\right\|^{2}_{L_{2}(\mathfrak{G})}\right]=\sum_{|\alpha|\le s}\mathbf{E}\left[\left\|\mathlarger{\nabla}_{a}^{\alpha}
\mathscr{I}_{a}(x,t)\right\|^{2}_{L_{2}(\mathfrak{G})}\right]
\end{align}
For $s=2$
\begin{align}
&\left\|\!\left\|\mathscr{I}_{a}\right\|\!\right\|^{2}_{H^{2}(\mathfrak{G})}\equiv
\mathbf{E}\left[\left\|{\mathscr{I}}_{a}\right\|^{2}_{H^{2}(\mathfrak{G})}\right]\nonumber\\&=\mathbf{E}\left[\sum_{|\alpha|\le s}\left\|\mathlarger{\nabla}_{a}^{\alpha}
\mathscr{I}_{a}(x,t)\right\|^{2}_{L_{2}(\mathfrak{G})}\right]=\sum_{|\alpha|\le s}{\bm{{\bm{\mathbf{E}}}}}\left[\left\|\mathlarger{\nabla}_{a}^{(2)}
\mathscr{I}_{a}(x,t)\right\|^{2}_{L_{2}(\mathfrak{G})}\right]\nonumber\\&
=\mathbf{E}\left[\left\|\mathscr{I}_{a}(x,t)\right\|^{2}\right]+\mathbf{E}\left[\left\|\mathlarger{\nabla}_{a}\mathscr{I}_{a}(x,t)\right\|^{2}\right]
+\mathbf{E}\left[\left\|\mathlarger{\nabla}_{a}\mathlarger{\nabla}_{b}\mathscr{I}_{a}(x,t)\right\|^{2}\right]\nonumber\\&
=\int_{\mathfrak{G}}\mathbf{E}\left[\left|\mathscr{I}_{a}(x,t)\right|^{2}\right]d\mathcal{V}(x)
+\int_{\mathfrak{G}}\mathbf{E}\left[\left|\mathlarger{\nabla}_{a}\mathscr{I}_{a}(x,t)\right|^{2}\right]d\mathcal{V}(x)
+\int_{\mathfrak{G}}\mathbf{E}\left[\left|\mathlarger{\nabla}_{a}\mathlarger{\nabla}_{b}\mathscr{I}_{a}(x,t)\right|^{2}\right]d\mathcal{V}(x)+
\end{align}
\end{prop}
\begin{lem}
The 1st and 2nd-order Sobolov norms are given by
\begin{align}
\left\|\!\left\|\mathscr{I}_{a}\right\|\!\right\|^{2}_{H^{1}(\mathfrak{G})}&\equiv
\mathbf{E}\left[\left\|\Phi_{a}\right\|^{2}_{H^{1}(\mathfrak{G})}\right]&\nonumber\\&
= \big\|\Phi_{a}\big\|^{2}_{{E}(\mathfrak{G})}
{vol}(\mathfrak{G})+{A}^{2}\big\|\Phi_{a}\big\|^{2}_{\mathrm{E}(\mathfrak{G})}\left(\mathbf{W}\big(\mu\|\Phi_{a}\|_{{E}(\mathfrak{G})}\right)^{2}
\int_{\mathfrak{G}}var(x)d\mathcal{V}(x)\nonumber\\&
+{A}^{2}\big\|\Phi_{a}\big\|^{2}_{\mathrm{E}(\mathfrak{G})}\left(\mathbf{W}\big(\mu\|\Phi_{a}\|_{\mathrm{E}(\mathfrak{G})}\right)^{2}
\sum_{I=1}^{\infty}\int_{\mathfrak{G}}\mathlarger{\nabla}_{a}f_{I}(x)\mathlarger{\nabla}_{a}f_{I}(x)d\mathcal{V}(x)
\end{align}
and
\begin{align}
\left\|\!\left\|\mathscr{I}_{a}\right\|\!\right\|^{2}_{H^{2}(\mathfrak{G})}\equiv&
\mathbf{E}\left[\left\|\mathscr{I}_{a}\right\|^{2}_{H^{2}(\mathfrak{G})}\right]
 \big\|\Phi_{a}\big\|^{2}_{E(\mathfrak{G})}
vol(\mathfrak{G})+{A}^{2}\big\|\Phi_{a}\big\|^{2}_{E(\mathfrak{G})}\left(\mathbf{W}\big(\mu\|\Phi_{a}\|_{{E}(\mathfrak{G})}\right)^{2}
\int_{\mathfrak{G}}var(x)d\mathcal{V}(x)\nonumber\\&
+{A}^{2}\big\|\Phi_{a}\big\|^{2}_{E(\mathfrak{G})}\left(\mathbf{W}\big(\mu\|\Phi_{a}\|_{\mathrm{E}(\mathfrak{G})}\right)^{2}
\sum_{I=1}^{\infty}{\mathrm{Z}_{I}}_{a}\int_{\mathfrak{G}}\mathlarger{\nabla}_{a}f_{I}(x)\mathlarger{\nabla}_{a}f_{I}(x)d\mathcal{V}(x)\nonumber\\&
+{A}^{2}\big\|\Phi_{a}\big\|^{2}_{E(\mathfrak{G})}\left(\mathbf{W}\big(\mu\|\Phi_{a}\|_{\mathrm{E}(\mathfrak{G})}\right)^{2}
\sum_{I=1}^{\infty}{\mathrm{Z}_{I}}_{a}\int_{\mathfrak{G}}\mathlarger{\nabla}_{a}\mathlarger{\nabla}_{b}f_{I}(x)\mathlarger{\nabla}_{a}\mathlarger{\nabla}_{b}f_{I}(x)d\mathcal{V}(x)\nonumber\\&
={\bm{{\bm{\mathbf{E}}}}}\left[\left\|\mathscr{I}_{a}\right\|^{1}_{H^{1}(\mathfrak{G})}\right]
+{A}^{2}\big\|\Phi_{a}\big\|^{2}_{\mathrm{E}(\mathfrak{G})}\left(\mathbf{W}\big(\mu\|\Phi_{a}\|_{\mathrm{E}(\mathfrak{G})}\right)^{2}\nonumber\\&\otimes
\sum_{I=1}^{\infty}\mathrm{Z}_{I}\int_{\mathfrak{G}}\mathlarger{\nabla}_{a}\mathlarger{\nabla}_{b}f_{I}(x)\mathlarger{\nabla}_{a}\mathlarger{\nabla}_{b}f_{I}(x)d\mathcal{V}(x)
\end{align}
\end{lem}
\begin{proof}
The $H^{1}$ norm is
\begin{align}
\left\|\mathscr{I}_{a}\right\|^{2}_{H^{1}(\mathfrak{G})}&=
\left\|\mathscr{I}_{a}\right\|^{2}_{L_{2}(\mathfrak{G})}+\left\|\mathlarger{\nabla}_{a}\mathscr{I}_{a}\right\|^{2}_{L_{2}(\mathfrak{G})}\nonumber\\&
\left\|\Phi_{a}+{A}\Phi_{a}\mathbf{W}\big(\mu\|\Phi_{a}\|_{{E}(\mathfrak{G})}\big)\sum_{I=1}^{\infty}\mathrm{Z}_{I}^{1/2}
f_{I}(x)\otimes\mathscr{Z}_{I}\right\|^{2}_{L_{2}(\mathfrak{G})}\nonumber\\&
\left\|\Phi_{a}+{A}\Phi_{a}\mathbf{W}\big(\mu\|\Phi_{a}\|_{E(\mathfrak{G})}\big)\sum_{I=1}^{\infty}\mathrm{Z}_{I}^{1/2}
\mathlarger{\nabla}_{a}f_{I}(x)\otimes\mathscr{Z}_{I}\right\|^{2}_{L_{2}(\mathfrak{G})}\nonumber\\&
\int_{\mathfrak{G}}\left|\Phi_{a}+{A}\Phi_{a}\mathbf{W}\big(\mu\|\Phi_{a}\|_{{E}(\mathfrak{G})}\big)\sum_{I=1}^{\infty}{\mathrm{Z}_{I}}^{1/2}
f_{I}(x)\otimes\mathscr{Z}_{I}\right|^{2}d\mathcal{V}(x)\nonumber\\&
\int_{\mathfrak{G}}\left|\Phi_{a}+{A}\Phi_{a}\mathbf{W}\big(\mu\|\Phi_{a}\|_{{E}(\mathfrak{G})}\big)\sum_{I=1}^{\infty}{\mathrm{Z}_{I}}^{1/2}
\mathlarger{\nabla}_{a}f_{I}(x)\otimes\mathscr{Z}_{I}\right|^{2}d\mathcal{V}(x)\nonumber\\&
=\int_{\mathfrak{G}}\bigg\lbrace\|\Phi_{a}\|^{2}_{{E}(\mathfrak{G})}+2{A}\|\Phi_{a}\|^{2}_{E(\mathfrak{G})}(\mathbf{W}
\big(\mu\|\Phi_{a}\|_{E(\mathfrak{G})}\big)\sum_{I=1}^{\infty}{\mathrm{Z}_{I}}^{1/2}f_{I}(x)\otimes\mathscr{Z}_{I}\nonumber\\&+
{A}^{2}\|\Phi_{a}\|^{2}_{E(\mathfrak{G})}(\mathbf{W}\big(\mu\|\Phi_{a}\|_{{E}(\mathfrak{G})}\big)^{2}\sum_{a=1}
\sum_{I=1}^{\infty}{\mathrm{Z}_{I}}^{1/2}{\mathrm{Z}_{I}}^{1/2}f_{I}(x)f_{I}(x)\bigg(\mathscr{Z}_{I}\otimes\mathscr{Z}_{J}\bigg)\bigg\rbrace d\mathcal{V}(x)\nonumber\\&=\int_{\mathfrak{G}}\|2{A}\|\Phi_{a}\|^{2}_{{E}(\mathfrak{G})}(\mathbf{W}\big(\mu\|\Phi_{a}\|_{{E}(\mathfrak{G})}\big)
\sum_{I=1}^{\infty}{\mathrm{Z}_{I}}^{1/2}\mathlarger{\nabla}_{a}f_{I}(x)\otimes\mathscr{Z}_{I}d\mathcal{V}(x)\nonumber\\&+
{A}^{2}\int_{\mathfrak{G}}\|\Phi_{a}\|^{2}_{{E}(\mathfrak{G})}(\mathbf{W}\big(\mu\|\Phi_{a}\|_{{E}(\mathfrak{G})}\big)^{2}\sum_{I=1}\sum_{I=1}^{\infty}
{\mathrm{Z}_{I}}^{1/2}{\mathrm{Z}_{I}}^{1/2}\mathlarger{\nabla}_{a}f_{I}(x)\mathlarger{\nabla}_{a}f_{I}(x)\bigg(\mathscr{Z}_{I}\otimes\mathscr{Z}_{I}\bigg)d\mathcal{V}(x)
\end{align}
Taking the expectation
\begin{align}
&\left\|\mathscr{I}_{a}\right\|^{2}_{H^{1}(\mathfrak{G})}=\int_{\mathfrak{G}}\bigg\lbrace\|\Phi_{a}\|^{2}_{E(\mathfrak{G})}+2{A}\|\Phi_{a}\|^{2}_{E(\mathfrak{G})}
(\mathbf{W}\big(\mu\|\Phi_{a}\|_{E(\mathfrak{G})}\big)\sum_{I=1}^{\infty}{\mathrm{Z}_{I}}^{1/2}f_{I}(x){\bm{{\bm{\mathbf{E}}}}}\bigg[\mathscr{Z}_{I}\bigg]\nonumber\\&+
{A}^{2}\|\Phi_{a}\|^{2}_{E(\mathfrak{G})}(\mathbf{W}\big(\mu\|\Phi_{a}\|_{E(v)}\big)^{2}\sum_{a=1}\sum_{I=1}^{\infty}{\mathrm{Z}_{I}}^{1/2}
{\mathrm{Z}_{I}}^{1/2}f_{I}(x)f_{I}(x){\bm{{\bm{\mathbf{E}}}}}\left[\bigg(\mathscr{Z}_{I}\otimes\mathscr{Z}_{I}\bigg)\right]\bigg\rbrace d\mathcal{V}(x)\nonumber\\&=\int_{\mathfrak{G}}\bigg\lbrace\|2 {A}\|\Phi_{a}\|^{2}_{E(\mathfrak{G})}(\mathbf{W}
\big(\mu\|\Phi_{a}\|_{E(\mathfrak{G})}\big)\sum_{I=1}^{\infty}{\mathrm{Z}_{I}}^{1/2}\mathlarger{\nabla}_{a}f_{I}(x)
{\bm{{\bm{\mathbf{E}}}}}\bigg[\mathscr{Z}_{I}\bigg]\nonumber\\&+{A}^{2}\|\Phi_{a}\|^{2}_{E(\mathfrak{G})}(\mathbf{W}
\big(\mu\|\Phi_{a}\|_{{E}(\mathfrak{G})}\big)^{2}\sum_{a=1}\sum_{I=1}^{\infty}\mathrm{Z}_{I}^{1/2}
{\mathrm{Z}_{I}}^{1/2}\mathlarger{\nabla}_{a}f_{I}(x)\mathlarger{\nabla}_{a}f_{I}(x){\bm{{\bm{\mathbf{E}}}}}\left[\bigg(\mathscr{Z}_{I}\otimes\mathscr{Z}_{I}\bigg)\right]\bigg\rbrace d\mathcal{V}(x)\nonumber\\&=\big\|\Phi_{a}\|^{2}_{E(\mathfrak{G})}{vol}(\mathfrak{G})+{A}^{2}\|\Phi_{a}\|^{2}_{E(\mathfrak{G})}(\mathbf{W}
\big(\mu\|\Phi_{a}\|_{{E}(\mathfrak{G})}\big)^{2}\sum_{a=1}\sum_{I=1}^{\infty}{\mathrm{Z}_{I}}^{1/2}
{\mathrm{Z}_{I}}^{1/2}\int_{\mathfrak{G}}f_{I}(x)d\mathcal{V}(x)\nonumber\\&+{A}^{2}\|\Phi_{a}\|^{2}_{E(\mathfrak{G})}(\mathbf{W}
\big(\mu\|\Phi_{a}\|_{{E}(\mathfrak{G})}\big)^{2}\sum_{a=1}\sum_{I=1}^{\infty}{\mathrm{Z}_{I}}^{1/2}
{\mathrm{Z}_{I}}^{1/2}\int_{\mathfrak{G}}\mathlarger{\nabla}_{a}f_{I}(x)\mathlarger{\nabla}_{a}f_{I}(x)d\mathcal{V}(x)
\nonumber\\&=\big\|\Phi_{a}\|^{2}_{{E}(\mathfrak{G})}{vol}(\mathfrak{G})+{A}^{2}\|\Phi_{a}\|^{2}_{E(\mathfrak{G})}(\mathbf{W}
\big(\mu\|\Phi_{a}\|_{{E}(\mathfrak{G})}\big)^{2}\sum_{I=1}^{\infty}{\mathrm{Z}_{I}}_{a}\nonumber\\& +{A}^{2}\|\Phi_{a}\|^{2}_{{E}(\mathfrak{G})}(\mathbf{W}\big(\mu\|\Phi_{a}\|_{E(\mathfrak{G})}\big)^{2}\sum_{I=1}^{\infty}{\mathrm{Z}_{I}}(x)
\int_{\mathfrak{G}}\mathlarger{\nabla}_{a}f_{I}(x)\mathlarger{\nabla}_{a}f_{I}(x)d\mathcal{V}(x) \nonumber\\&=\big\|\Phi_{a}\|^{2}_{E(\mathfrak{G})}{vol}(\mathfrak{G})+{A}^{2}\|\Phi_{a}\|^{2}_{E(\mathfrak{G})}(\mathbf{W}
\big(\mu\|\Phi_{a}\|_{{E}(\mathfrak{G})}\big)^{2}\int_{\mathfrak{G}}var(x)d\mathcal{V}(x)\nonumber\\& +{A}^{2}\|\Phi_{a}\|^{2}_{E(\mathfrak{G})}(\mathbf{W}\big(\mu\|\Phi_{a}\|_{{E}(\mathfrak{G})}\big)^{2}\sum_{I=1}^{\infty}\mathrm{Z}_{I}\int_{\mathfrak{G}}\mathlarger{\nabla}_{a}f_{I}(x)\mathlarger{\nabla}_{a}f_{I}(x)d\mathcal{V}(x)
\end{align}
since $\int_{\mathfrak{G}}f_{I}(x)f_{I}(x)d\mathcal{V}(x)=1$ and $\sum_{I=1}^{\infty}\mathrm{Z}_{I}=\int_{\mathfrak{G}}var(x)d\mathcal{V}(x)$ . The second-order norm is derived in the same way.
\end{proof}
\section{APPLICATION TO HYDRODYNAMIC TURBULENCE}
The formalism developed in the previous section for stochastic vector flows or random fields within a domain is now tentatively applied to the problem of turbulent flow of a viscous incompressible fluid within a domain. It may be that this random field can capture at least some salient features of a turbulent fluid flow. In particular, the goal is to establish if one can have anomalous dissipation in the inviscid limit as $\mu\rightarrow 0$.
\subsection{Basic results from fluid mechanics}
Some basic background results from smooth or deterministic 'laminar' fluid mechanics are briefly given \textbf{[33,34,50,51,54]}. In the absence of turbulence, we consider a set of smooth  and deterministic solutions $(u_{a}(x,t),\rho)$ of the steady state viscous Burgers equations, with the pressure gradient term set to zero in the Navier-Stokes equations. Here $u_{a}(x,t)$ is the deterministic fluid velocity at $x\in{\mathfrak{G}}\subset{\mathbb{R}}^{d}$, and $\rho$ is the (uniform) density. For the general Burgers equations, let ${\mathfrak{G}}\subset\bm{\mathbb{R}}^{3}$ be a compact bounded domain with ${x}\in{\mathfrak{G}}$ and filled with a fluid of density $\rho:[0,T]\otimes\mathbb{R}^{3}\rightarrow{\mathbb{R}^{3}}$, and velocity $u:\mathfrak{G}\otimes\mathbb{R}^{(+)}\rightarrow{\mathbb{R}}^{3}$ so that $\rho=\rho(x,t)$ and $u_{a}(x,t)=u_{a}(x,t)$. The Burger's equations is then
\begin{align}
&\frac{\partial}{\partial t}u_{a}(x,t)+\mathbb{I\!N}u_{a}(x,t)\nonumber\\&\equiv\frac{\partial}{\partial t}u_{a}(x,t)
-\nu \mathlarger{\Delta} u_{a}(x,t)+u^{b}(x,t)\mathlarger{\nabla}_{b}u_{a}(x,t)=0,(x,t)\in{\mathfrak{G}}\otimes\mathbb{R}^{+}
\end{align}
The viscosity of the fluid is $\nu$ is very small so that $\nu\sim 0$, with the incompressibility condition $\mathlarger{\nabla}_{a}{U}^{a}=0$.
\begin{defn}
The following will also apply:
\begin{enumerate}[(a)]
\item The smooth initial Cauchy data is $u(x,0)_{a}=U^{o}(x)$. One could also impose periodic boundary conditions if ${\mathfrak{G}}$ is a cube or box of with sides of length $\mathrm{Z}$ such that $u_{a}({x}+{L},t)=u_{a}(x,t)$, or no-slip BCs $u_{a}(x,t)=0, \forall~x\in\partial{\mathfrak{G}}$. For some $C,K>0$, the initial data will also satisfy a bound of the typical form $ |\mathlarger{\nabla}^{\alpha}U{o}(x)|\le C(1+|x|)^{-K}$. By a \textbf{\textit{smooth deterministic flow}}, we mean a $u_{a}(x,t)$ which is deterministic and non-random  and evolves \textit{predictably} by the NS equations from some initial Cauchy data $u_{a}(x,0)=U^{o}(x)$. For example, a simple trivial laminar flow solution is $u_{a}(x,t)=u_{a}=const$. A generic smooth flow will be differentiable to at least 2nd order so that $\mathlarger{\nabla}_{b}u_{a}(x,t)$ and $\mathlarger{\nabla}_{a}\mathlarger{\nabla}_{b}u_{a}(x,t)$ exist. The fluid velocity $u_{a}(x,t)$ is a divergence-free vector field that should be physically reasonable: that is, the solution should not grow too large or blow up as $t\rightarrow \infty$.
\item The Reynolds number or 'Reynolds function' within ${\mathfrak{G}}$ with $vol({\mathfrak{G}})\sim L^{d}$ is
\begin{align}
\mathsf{RE}(x,t)=\mathsf{RE}(x,t)=\frac{\|u_{a}(x,t)\| L}{\nu}
\end{align}
and for a constant fluid                                                                                                                                             velocity $u_{a}(x,t)=u_{a}$ one has $\mathsf{RE}=\frac{\|u_{a}\| L}{\nu}$ For $\nu>0$ but $\nu \sim 0$, then the Reynolds number will be very large but not infinite.
\item The vorticity is
\begin{align}
\omega^{a}(x,t)=\varepsilon^{abc}\mathlarger{\nabla}_{b}U_{c}(x,t)
\end{align}
\item The basic energy balance equation for a deterministic viscous fluid flow is obeyed such that
\begin{align}
&\|u_{a}(\bullet,t\|_{L_{2}(\mathfrak{G})}^{2}=\frac{d E(t)}{dt}=\frac{d}{dt}{\int}_{\mathfrak{G}}u_{a}(x,t)u^{a}(x,t)d\mathcal{V}(x)
\nonumber\\&=-\nu{\int}_{\mathfrak{G}}|\mathlarger{\nabla}_{b}u_{a}(x,t)\mathlarger{\nabla}^{b}u_{a}(x,t)|^{2}d\mathcal{V}(x)
\end{align}
or $\frac{d E(t)}{dt}=-\nu\mathcal{Z}(t)$, where $\mathcal{Z}$ is the enstrophy. This can be also be derived from the NS equations.
\item The energy dissipation rate for a constant viscosity $\nu$ is
\begin{align}
\mathlarger{\epsilon}(x,t)=\frac{1}{2}\nu\left(\mathlarger{\nabla}_{b}u_{a}(x,t)+\mathlarger{\nabla}^{a}U^{b}(x,t)\right)^{2}
\end{align}
If the fluid is isotropic then $\mathlarger{\nabla}_{b}u_{a}=\mathlarger{\nabla}_{a}u_{b}$ so that ${\epsilon}(x,t)=\nu\left(\mathlarger{\nabla}_{b}u_{a}(x,t)\mathlarger{\nabla}^{b}u_{a}(x,t)\right) $
\end{enumerate}
\end{defn}
\subsection{Turbulent fluid flow as a stochastic vector flow or current}
The ideas of Section 3 for a stochastic vector flow described by the random field $\mathscr{I}_{a}(x,t)$ in a domain ${\mathfrak{G}}$ with $(x,t)\in\mathfrak{G}\otimes[0,T]$ is now tentatively applied to the description of a turbulent fluid flow in a domain, where the flux $\Phi_{a}(x,t)$ is now identified with a stochastic or random velocity flow ${\mathscr{U}}_{a}(x,t)$. The underlying deterministic vector field $\Phi_{a}(x,t)$ is also now identified with the underlying unperturbed fluid flow $u_{a}(x,t)$ which evolves by the Navier-Stokes or Burgers equation, and which becomes turbulent via random perturbations. The key feature of this random field is that there is a nonlinear coupling or 'weighting' so that the amplitudes and random fluctuations  grow or change with changes in the underlying flow $u_{a}(x,t)$. This is due to the (dimensionless) weighting or 'modulating' functional $\mathbf{W}(\mu\|u_{a}(x,t)\|)$ so that the turbulent flow or current has the random field description
\begin{align}
{\mathscr{U}}_{a}(x,t)=u_{a}(x,t)+A u_{a}(x,t)\mathbf{W}(\mu\|u_{a}(x,t)\|)\sum_{I=1}^{\infty}{\mathrm{Z}^{1/2}_{I}}f_{I}(x)\otimes
\mathscr{Z}_{I}
\end{align}
A suitable nonlinear weighting functional $\mathbf{W}$ can be defined in terms of the Reynolds number as follows. \emph{The key feature is that the randomness will grow as the Reynolds number increases, keeping L and $\nu$ fixed.} To quote Von Neumann :"\textit{The transition from 'laminar' flow $\mathcal{L}$ to fully turbulent flow $\mathcal{T}$ is best defined by a critical value of the Reynolds number $\mathsf{RE}_{*}$ than by any other geometric quantity}". A mathematical description of the transition $\mathcal{L}\rightarrow\mathcal{T}$ and intermittence is fraught with mathematical difficulties and may even be intractable. Again, no truly rigorous theory or description exists but the transition to turbulence in the fluid typically occurs for $\mathsf{RE}_{*}\sim 2000$. In the turbulent flows it is assumed that $\mathsf{RE}(x,t)\gg\mathsf{RE}_{*}$.
\begin{prop}
Let $u_{a}(x,t)$ be a smooth laminar fluid flow within $\mathfrak{G}$ where $vol(\mathfrak{G})=\int_{\mathfrak{G}}d\mathcal{V}(x)\sim L^{3}$ and let $\mu=L/\nu$, where $\nu$ is the fluid viscosity. Then
\begin{align}
\mathbf{W}[\mu\|u_{a}(x,t)\|)=\mathbf{W}\left(\frac{L}{\nu}\bigg\|u_{a}(x,t)\bigg\|\right)=\mathbf{W}(\mathsf{RE}(x,t))
\end{align}
since $\mathsf{RE}(x,t)=\|u_{a}(x,t)\|L/\nu $. Let $\mathsf{RE}_{*}$ be the critical Reynolds number at which turbulence in the flow is initiated and assume
$\mathsf{RE}(x,t)\gg\mathsf{RE}_{*}$. Define the following set
\begin{align}
\mathcal{S}=\lbrace\mathsf{RE}(x,t)\ge\mathsf{RE}_{*}: (x,t)\in\mathfrak{G}\otimes\mathbb{R}^{+}\rbrace
\end{align}
The following nonlinear polynomial ansatz can be chosen such that for all $x\in\mathfrak{G}$ and $t>0$
\begin{align}
{\mathbf{W}}(\mathsf{RE}(x,t)-\mathsf{RE}_{*})=\mathsf{RE}(x,t)-\mathsf{RE}_{*})^{\beta}\mathlarger{\chi}_{\mathcal{S}}(\mathsf{RE}(x,t))
\end{align}
where $\beta\le \tfrac{1}{2}$ or $\beta>\tfrac{1}{2}$.  Here $\chi_{\mathcal{S}}(\mathsf{RE}_{*})$ is an 'indicator function' such that $\chi_{\mathcal{S}}(\mathsf{RE}_{*})=0$ if $\mathsf{RE}(x,t)\le \mathsf{RE}_{*}$ and $\chi_{\mathcal{S}}(\mathsf{RE}_{*})=1$ if $\mathsf{RE}(x,t)> \mathsf{RE}_{*}$.
\begin{align}
  \mathlarger{\chi}_{\mathcal{S}}(\mathsf{RE}(x,t)) =
  \begin{cases}
    1 & \text{if } \mathsf{RE}(x,t) \in \mathcal{S} \\
    0 & \text{if } \mathsf{RE}(x,t) \notin \mathcal{S}
  \end{cases}
\end{align}
Hence ${\mathbf{W}}(\mathsf{RE}(x,t))=0$ if $\mathsf{RE}(x,t)\le \mathsf{RE}_{*}$.
\end{prop}
A turbulent fluid flow can then be (tentatively) represented by constructing the following weighted random vector field or current/flow within the domain as in Section 3.
\begin{prop}
Let ${\mathscr{T}}(x)$ be a GRF and let $u_{a}:{\mathfrak{G}}\otimes\mathbb{R}^{+}\rightarrow\mathbb{R}^{3}$ be a smooth deterministic vector field describing the flow of a viscous incompressible fluid, existing for all $(x,t)\in{\mathfrak{G}}\otimes\mathbb{R}^{+}$ such that following hold:
\begin{enumerate}
\item By smooth, the first and second derivatives $\mathlarger{\nabla}_{b}u_{a}(x,t)$ and $\mathlarger{\nabla}_{a}\mathlarger{\nabla}_{b}u_{a}(x,t)$ exist, and deterministic is taken to mean that the field $u_{a}(x,t)$ evolves from some initial data $u_{a}(x,0)$ and boundary conditions on $\partial{\mathfrak{G}}$, via the Navier-Stokes or Burgers PDE of the generic form
\begin{align}
&\frac{\partial}{\partial t}u_{a}(x,t)+\mathbb{I\!N}[\mathlarger{\nabla}_{b},\mathlarger{\Delta},u_{a}(x,t)]u_{a}(x,t)\nonumber\\&=\frac{\partial}{\partial t}u_{a}(x,t)-\nu \mathlarger{\Delta} u_{a}(x,t)+u^{b}(x,t)\mathlarger{\nabla}_{b}u_{a}(x,t)=0\nonumber
\end{align}
\item The weighting functional is of the polynomial form
\begin{align}
\mathbf{W}(\mathsf{RE}(x,t))=\big(\mathsf{RE}(x,t)-\mathsf{RE}_{*}\big)^{\beta}\mathlarger{\mathlarger{\chi}}_{\mathcal{S}}(\mathsf{RE}(x,t))\nonumber
\end{align}
where $\beta\le \tfrac{1}{2}$
\item The GRF ${\mathscr{T}}(x)$ has a spectral representation via the Karhunen-Loeve theorem and has a well-behaved homogenous and isotropic kernel $K(x,y;\lambda)$.
\end{enumerate}
Then the following stochastic vector field $\mathscr{U}_{a}(x,t)$ representing a turbulent fluid flow can be defined or 'engineered' such that
\begin{align}
&{\mathscr{U}}_{a}(x,t)=u_{a}(x,t)+\alpha u_{a}(x,t)\big(\mathsf{RE}(x,t)-{\mathsf{RE}}_{c}\big)^{\beta}{\mathlarger{\chi}}_{\mathcal{S}}({\mathsf{RE}}(x,t))
{\mathscr{Z}}_{I}\nonumber\\&
=u_{a}(x,t)+Au_{a}(x,t)({\mathsf{RE}}(x,t)-\mathsf{RE}_{*})^{\beta}{\mathlarger{\chi}}_{\mathcal{S}}\big[{\mathsf{RE}}(x,t)\big]
\sum_{I=1}^{\infty}{\mathrm{Z}^{1/2}_{I}}f_{I}(x)\otimes{\mathscr{Z}}_{I}
\end{align}
where $A$ is an arbitrary fixed amplitude or constant. The expectation of the field is then
\begin{align}
\mathbf{E}[{\mathscr{U}}_{a}(x,t)]&=u_{a}(x,t)+\alpha u_{a}(x,t)\big(\mathbf{\mathsf{RE}}(x,t)-{\mathsf{RE}}_{c}\big)^{\beta}{\mathlarger{\chi}}_{\mathcal{S}}({\mathsf{RE}}(x,t))
\mathbf{E}[{\mathscr{T}}(x)]\nonumber\\&
=u_{a}(x,t)+{A}u_{a}(x,t)({\mathsf{RE}}(x,t)-{\mathsf{RE}}_{c})^{\beta}{\mathlarger{\chi}}_{\mathcal{S}}({\mathbf{\mathsf{RE}}}(x,t))\nonumber\\&
\otimes\sum_{I=1}^{\infty}{\mathrm{Z}^{1/2}_{I}}f_{I}(x)\mathbf{E}[{\mathscr{Z}}_{I}]=u_{a}(x,t)
\end{align}
since $\mathbf{E}[\mathscr{T}(x)]=0$ and $\mathbf{E}[\mathscr{Z}_{I}]=0$.
\end{prop}
Note that for velocities giving a Reynolds number below the critical Reynolds number
\begin{align}
\mathbf{E}[{\mathscr{U}}_{a}(x,t)]=u_{a}(x,t)
\end{align}
and the fluid flow is again smooth or laminar/deterministic. Since we are considering the case of high Reynolds number flows with ${\mathsf{RE}}(x,t)\gg{\mathsf{RE}}_{c}$ the random or stochastic current can be written as
\begin{align}
{\mathscr{U}}_{a}(x,t)&=u_{a}(x,t)+\alpha u_{a}(x,t)({\mathsf{RE}}(x,t)-{\mathsf{RE}}_{c})^{\beta}
\sum_{I=1}^{\infty}\mathrm{Z}^{1/2}_{I}f_{I}(x)\otimes{\mathscr{Z}}_{I}\nonumber\\&
=u_{a}(x,t)+\mathscr{V}_{a}(x,t)
\end{align}
\begin{rem}
This representation introduces essentially dynamics and feedback between the random and deterministic elements or contributions. The 'engineered' random field $u_{a}(x,t)$ is to be considered as a very basic but nonetheless original model of a turbulent flow. If $u_{a}(x,t)$ satisfies an underlying nonlinear PDE such as the Burgers or Navier-Stokes equations then the random term ${\mathbb{P}_{a}(x,t)}$ now grows or scales with the Reynolds number ${\mathsf{RE}}(x,t)=\|u_{a}(x,t)\|L/\nu$ for fixed $L$ and $\nu$. The random term is weighted by a nonlinear term $A({\mathsf{RE}}(x,t)-\mathsf{RE}_{*})^{\beta}{U}_{a}(x,t)$, which scales with the Reynolds number, so the turbulent flow grows more random with increasing ${\mathsf{RE}}(x,t)$. This provides a basic mechanism or framework to account for increasing turbulence with increasing Reynolds number, and to incorporate growing stochasticity. To summarise, this random field has the following desirable properties:
\begin{enumerate}
\item \textbf{Stochasticity and Reynolds number} The term $(\mathsf{RE}(x,t)-\mathsf{RE}_{*})^{\beta}$ leads to a stochasticity that scales nonlinearly and polynomially with Reynolds number. This is consistent with the intuitive notion that turbulence grows with increasing Reynolds number
\item \textbf{Non-linearity interaction}. The inclusion of a nonlinear weighting or 'modulating' term of the form $u_{a}(x,t){A}(\mathsf{RE}(x,t)-\mathsf{RE}_{*})^{\beta}$ in the random perturbations introduces a nonlinearity that is reflective in the complexities of turbulent flow.
\item \textbf{Simplicity} The description is relatively simple compared to the full complexity of turbulent fluids but may capture some of the salient features of turbulent flows.
\end{enumerate}
However, while the model may capture some essential salient features of turbulent flows, it must also be noted that turbulence is a still a very complex and multifaceted phenomena that involves a wide range of scales, interactions and complexities. Thus, this description probably cannot fully capture the richness and nuances of turbulent flows at high Reynolds number. The random field model can still be utilised to explore various problems such as anomalous dissipation.
\end{rem}
\subsection{Reynolds~number~and~the~transition~of laminar flow to~turbulent flow}
A mathematical description of the transition $\mathcal{L}\rightarrow\mathcal{T}$ and intermittence is fraught with mathematical difficulties and may even be intractable. No truly rigorous theory or description exists. In the 1940s and even into the 1970s, the accepted theory was that of Landau and Hopf \textbf{[12]}. They (independently) proposed a 'branching theory' whereby a smooth laminar flow essentially undergoes an 'infinity of transitions', during which an additional frequencies (or wavenumbers) arise due to flow instabilities, leading to complex turbulent motion. However, this (very) heuristic theory has been shown to be untenable in virtually all turbulence scenarios and was never experimentally observed or verified, and is based on linearised approximations. Nevertheless, it has an interesting feature in that \textbf{the amplitude grows and turbulence evolves as a power law of the difference of the Reynolds number with the critical Reynolds number.} (See the argument in \textbf{[12]}) The amplitude of oscillation is then proportional to the square root of the difference of the Reynolds number and critical Reynolds number
\begin{align}
\bm{\mathcal{A}}(t)=A\left(|{\mathsf{RE}}-{\mathsf{RE}}_{c}|\right)^{1/2}
\end{align}
Although the Landau-Hopf theory is more or less untenable, one can be inspired to try and define a random field representing a turbulent flow that grows in  amplitude and randomness with an increasing Reynolds number above some critical value.
\begin{prop}
For the turbulent flow ${\mathscr{U}}_{a}(x,t)$, one can define the ansatz
\begin{align}
\bm{\mathcal{A}}(t)=\sqrt{{\mathsf{RE}}(x,t)-{\mathsf{RE}}_{c})}\equiv ({\mathsf{RE}}(x,t)-{\mathsf{RE}}_{c})^{1/2}
\end{align}
\begin{align}
A(t)=Au_{a}(x,t)\left({\mathsf{RE}}(x,t)-{\mathsf{RE}}_{c})\right)^{1/2}
\end{align}
which is in the spirit of (4.15) with $\beta=1/2$. Then
\begin{align}
&{\mathscr{U}}_{a}(x,t)=u_{a}(x,t)+Au_{a}(x,t)({\mathsf{RE}}(x,t)-{\mathsf{RE}}_{c})^{1/2}\sum_{I=1}^{\infty}\mathrm{Z}_{I}^{1/2}
f_{I}\otimes \mathscr{Z}_{I}\nonumber\\&
=u_{a}(x,t)\left(1+{A}(t)\sum_{I=1}^{\infty}\mathrm{Z}_{I}^{1/2}f_{I}\otimes \mathscr{Z}_{I}\right)
\end{align}
\end{prop}
\begin{lem}
In the limit of large viscosity, the random term becomes suppressed and vanishes to zero so that
\begin{align}
\lim_{\nu\rightarrow\infty}{\mathscr{U}}_{a}(x,t)=u_{a}(x,t)
\end{align}
\end{lem}
\begin{proof}
\begin{align}
&\lim_{\nu\rightarrow\infty}\mathscr{U}_{a}(x,t)=u_{a}(x,t)+ {A}u_{a}(x,t)\lim_{\nu\rightarrow\infty}({\mathsf{RE}}(x,t)-\mathsf{RE}_{*})^{1/2}\sum_{I=1}^{\infty}\mathrm{Z}_{I}^{1/2}
f_{I}(x)\otimes\mathscr{Z}_{I}\nonumber\\&
=u_{a}(x,t)+ {A}u_{a}(x,t)\lim_{\nu\rightarrow\infty}\sqrt{\left(\frac{\|u_{a}(x,t)\|L}{\nu}-\mathsf{RE}_{*}\right)}\sum_{I=1}^{\infty}\mathrm{Z}_{I}^{1/2}f_{I}(x)
\mathscr{Z}_{I}=u_{a}(x,t)
\end{align}
\end{proof}
Intuitively, this what one would expect as viscosity suppresses turbulence.

All moments to all orders are finite and bounded.
\begin{lem}
Given the random field or turbulent flow ${\mathscr{U}}_{a}(x,t)$ then the moments $\mathfrak{G}(x,t;P)=\mathbf{E}[\mathscr{U}_{a}(x,t)|^{p}]$ are finite and bounded in that
\begin{align}
&\mathfrak{G}(x,t;p)={\mathbf{E}}[{\mathscr{U}_{a}(x,t)}|^{p}]\nonumber\\&\le 2^{p-1}|u_{a}(x,t)|^{p}+2^{p-1}|A^{p}|u_{a}(x,t)|^{p}({\mathsf{RE}(x.t)-\mathsf{RE}_{*}})^{p/2}\sum_{I=1}^{\infty}\mathrm{Z}^{1/2}_{I}f_{I}(x)^{p}\otimes \mathscr{Z}_{I}<\infty
\end{align}
if $_{a}(x,t)|^{p}<\infty$ for all $(x,t)\in\mathfrak{G}\otimes[0,T]$.
\end{lem}
\begin{proof}
The proof follows immediately from Lemma (3.7)
\end{proof}
\section{THE ZERO-VISCOSITY LIMIT AND ANOMALOUS DISSIPATION}
The main theorem concerning the zeroth law and anomalous dissipation in the zero-viscosity or inviscid limit is now given in this section. The limit $\nu\rightarrow 0$ or $\mathsf{RE}\rightarrow \infty$ is the inviscid limit and formally gives the Euler equations from the Navier-Stokes equations. There are subtle technical issues which have been explored in some detail. If $u^{E}_{a}(x,t)\equiv u^{E}(x,t)$ is a solution of the Euler equations then
\begin{align}
\frac{\partial}{\partial t}u^{E}_{a}(x,t)+u^{b}_{E}\mathlarger{\nabla}_{b}u^{E}_{a}(x,t)=-\mathlarger{\nabla}_{a}p(x,t)
\end{align}
The Euler equations are quasi-linear with an elliptic-hyperbolic nature. Kato \textbf{[99]} provided criteria for convergence to the Euler equations in this limit. If $u^{NS}_{a}(x,t)$ is a solution of the NS equations then the strong convergence condition is
\begin{align}
\lim_{\nu\rightarrow 0}\sup_{t\in[0,T]}\big\|u^{NS}_{a}(x,t)-u^{E}_{a}(x,t)\|_{L_{2}(\mathfrak{G})}=0
\end{align}
and the vanishing of the energy dissipation rate for a deterministic NS flow such that
\begin{align}
\lim_{\nu\rightarrow 0}\nu\int_{[0,T]}\big\|\mathlarger{\nabla}_{b}u^{NS}_{a}(x,t)\big\|^{2}_{L_{2}(\mathfrak{G})}dt=0
\end{align}
For the random vectorial field ${\mathscr{U}}_{a}(x,t)$ within $\mathfrak{G}$, we wish to try and establish an anomalous dissipation law or 'zeroth-type law' of the form
\begin{align}
\lim_{\nu\rightarrow 0}\nu\int_{[0,T]}{{\bm{\mathbf{E}}}}\left[\big\|\mathlarger{\nabla}_{b}{\mathscr{U}}_{a}(x,t)\big\|^{2}_{L_{2}(\mathfrak{G})}\right]dt=0
\end{align}
The following would still be expected to hold, which expresses the zeroth law or anomalous dissipation.
\begin{align}
\lim_{\nu\rightarrow 0}\epsilon(x,t)=\lim_{\nu\rightarrow 0}\nu{\bm{{\bm{\mathbf{E}}}}}\big[\big|\mathlarger{\nabla}_{b}\mathscr{U}(x,t)\big|^{2}\big]=\epsilon > 0
\end{align}
This can also be stated in terms of the finiteness and positivity of the following expression
\begin{align}
\lim_{\nu\rightarrow 0}\sup~\nu\int_{0}^{T}\mathbf{E}\bigg[\bigg\|\mathlarger{\nabla}_{a}{\bm{\mathscr{U}}(\bullet,t)}\bigg\|^{2}_{L_{2}(\mathfrak{G})}\bigg] ds
=\lim_{\nu\rightarrow 0}\sup~\nu \int_{0}^{T}\int_{\mathfrak{G}}
\bigg[\bigg|\mathlarger{\nabla}_{b}{\bm{\mathscr{U}}_{a}(x,t)}\bigg|^{2}\bigg] d\mathcal{V}(x) ds>0
\end{align}
Before proceeding to the main theorem, the following lemma is required for the weighting or modulating functional involving the Reynolds number.
\begin{lem}
The generic random field ${\bm{\mathscr{U}}_{a}(x,t)}$ representing the turbulent flow within $\mathfrak{G}$ for an arbitrary $\beta$ is
\begin{align}
&=\mathscr{U}_{a}(x,t)+\lim_{\nu\rightarrow 0}u_{a}(x,t)(\mathsf{RE}(x,t)-\mathsf{RE}_{*})^{\beta}\sum_{I=1}^{\infty}\mathrm{Z}_{I}^{1/2}
f_{I}(x)\otimes\bm{\mathscr{Z}}_{I}\nonumber\\&
=u_{a}(x,t)+u_{a}(x,t)\left(\frac{\|u_{a}(x,t)\|L}{\nu}-\mathsf{RE}_{*}\right)^{\beta}
\sum_{I=1}^{\infty}\mathrm{Z}_{I}^{1/2}f_{I}(x)\otimes
\mathscr{Z}_{I}
\end{align}
with $\beta\le \tfrac{1}{2}$ and setting $\mathlarger{\chi}_{\mathcal{S}}(\mathsf{RE}(x,t))=1$ since $\mathsf{RE}(x,t)\gg\mathsf{RE}_{*}$ for all $(x,t)\in\mathfrak{G}\otimes\mathbb{R}^{+}$.
Let
\begin{align}\mathbf{S}(\mathsf{RE}-\mathsf{RE}_{*})&=\nu |\mathbf{W}\left(\mu\|u_{a}(x,t)\|\right)^{2\beta}\nonumber\\&=\nu(\mu\left(\tfrac{L}{\nu}\|u_{a}(x,t)\|
-\mathsf{RE}_{*}\right))^{2}=(\mathsf{RE}(x,t)-\mathsf{RE}_{*})^{2\beta}
\end{align}
Then
\begin{enumerate}
\item For $\beta=\frac{1}{2}$, and keeping $\|u_{a}(x,t)\|$ and L fixed, the inviscid limit is
\begin{align}
&\lim_{\nu\rightarrow 0}|{W}(\mathsf{RE}(x,t),\mathsf{RE}_{*})|^{2}=\lim_{\nu\rightarrow 0}\nu(\mathsf{RE}(x,t)-\mathsf{RE}_{*})^{2\beta}=\lim_{\nu\rightarrow 0}\left(\nu(\frac{\|u_{a}(x,t)\|L}{\nu}-\mathsf{RE}_{*}\right)^{2\beta}\nonumber\\&=\|u_{a}(x,t)\|L>0
\end{align}
and the limit is strictly positive.
\item For $\beta<\frac{1}{2}$, and keeping $\|u_{a}(x,t)\|$ and L fixed, the inviscid limit is zero so that
\begin{align}
\lim_{\nu\rightarrow 0}{\nu}|\mathbf{W}(\mathsf{RE}(x,t),\mathsf{RE}_{*})|^{2}=\lim_{\nu\rightarrow 0}\nu(\mathsf{RE}(x,t)-\mathsf{RE}_{*})^{2\beta}=\lim_{\nu\rightarrow 0}\nu \left(\frac{\|u_{a}(x,t)\|L}{\nu}-\mathsf{RE}_{*}\right)^{2\beta}=0
\end{align}
and the limit is strictly zero.
\item For $\beta>\frac{1}{2}$, and keeping $\|u_{a}(x,t)\|$ and L fixed, the inviscid limit is divergent or undetermined
\end{enumerate}
\end{lem}
\begin{proof}
If $\beta=\tfrac{1}{2}$ then
\begin{align}
&\lim_{\nu\rightarrow 0}{\nu}(\mathsf{RE}(x,t),\mathsf{RE}_{*})=\lim_{\nu\rightarrow 0}\nu(\mathsf{RE}(x,t)-\mathsf{RE}_{*})=\lim_{\nu\rightarrow 0}\nu\left(\frac{\|u_{a}(x,t)\|L}{\nu}-\mathsf{RE}_{*}\right)\nonumber\\&
=\frac{\nu}{\nu}\left({\|u_{a}(x,t)\|L}-\nu\mathsf{RE}_{*}\right)
\end{align}
Now
\begin{align}
\lim_{\nu\rightarrow 0}\frac{\nu}{\nu}=\lim_{\nu\rightarrow 0}\frac{\nu^{\prime}}{\nu^{\prime}}\equiv\frac{d\nu/d\nu}{d\nu/d\nu}=1
\end{align}
by the L'Hopital rule so that
\begin{align}
&\lim_{\nu\rightarrow 0}{\mu}(\mathsf{RE}(x,t),\mathsf{RE}_{*})=\lim_{\nu\rightarrow 0}\nu(\mathsf{RE}(x,t)-\mathsf{RE}_{*})=\lim_{\nu\rightarrow 0}\left(\frac{\|u_{a}(x,t)\|L}{\nu}-\mathsf{RE}_{*}\right)\nonumber\\&=\lim_{\nu\rightarrow 0}\frac{\nu}{\nu}\left(\|u_{a}(x,t)\|L-\nu\mathsf{RE}_{*}\right)=\|u_{a}(x,t)\|L>0
\end{align}
so this is always positive. For $\beta<\tfrac{1}{2}$, we again consider
\begin{align}
\lim_{\nu\rightarrow 0}\nu\left(\frac{\|u_{a}(x,t)L}{\nu}-\mathsf{RE}_{*}\right)^{2\beta}
\end{align}
Since $\beta<\tfrac{1}{2}$, the term $\left(\frac{\|u_{a}(x,t)L}{\nu}-\mathsf{RE}_{*}\right)^{2\beta}$ becomes very large as $\nu\rightarrow 0$ but the $\nu$ in front approaches zero. Hence, the overall limit goes to zero.
\begin{align}
\lim_{\nu\rightarrow 0}\nu\left(\frac{\|u_{a}(x,t)L}{\nu}-\mathsf{RE}_{*}\right)^{2\beta}=0
\end{align}
As a consistency check, a more involved proof can utilise the binomial series. First write
\begin{align}
|\mathbf{W}(\mathsf{RE}(x,t),\mathsf{RE}_{*})|^{2}=\nu\left(\frac{\|\mathscr{U}_{a}(x,t)\|L}{\nu}-\mathsf{RE}_{*}\right)^{2\beta}=\frac{\nu(\|u_{a}\|L)^{2\beta}}{\nu^{2\beta}}
\left(1-\frac{\nu}{\|u_{a}(x,t)\|L}\mathsf{RE}_{*}\right)^{2\beta}
\end{align}
The binomial series is
\begin{align}
(1+X)^{\alpha}=\sum_{n=0}^{\infty}\binom{2\beta}{n}X^{n}
\end{align}
with $\alpha \in\mathbb{R}$. Setting $X=-\tfrac{\nu}{\|u_{a}(x,t)\|L}\mathsf{RE}_{*}$ and $\alpha=2\beta$ gives
\begin{align}
\lim_{\nu\rightarrow 0}\nu\left(\frac{\|u_{a}(x,t)L}{\nu}-\mathsf{RE}_{*}\right)^{2\beta}&=\lim_{\nu\rightarrow 0}\frac{\nu(\|u_{a}\|L)^{2\beta}}{\nu^{2\beta}}\sum_{n=0}^{\infty}\binom{2\beta}{n}\left(-\frac{\nu}{\|u_{a}(x,t)\|L}\mathsf{RE}_{*}\right)^{n}\nonumber\\&
=\lim_{\nu\rightarrow 0}\frac{\nu(\|u_{a}\|L)^{2\beta}}{\nu^{2\beta}}\left(1+\sum_{n=1}^{\infty}\binom{2\beta}{n}\left(-\frac{\nu}{\|u_{a}(x,t)\|L}\mathsf{RE}_{*}\right)^{n}\right)\nonumber\\&
\lim_{\nu\rightarrow 0}\frac{\nu(\|u_{a}\|L)^{2\beta}}{\nu^{2\beta}}\left(1+\sum_{n=1}^{\infty}\frac{2\beta!}{n!(2\beta-n)!}
\left(-\frac{\nu}{\|u_{a}(x,t)\|L}\mathsf{RE}_{*}\right)^{n}\right)\nonumber\\&
=\lim_{\nu\rightarrow 0}\frac{\nu(\|u_{a}\|L)^{2\beta}}{\nu^{2\beta}}\left(1+\sum_{n=1}^{\infty}\frac{2\beta-n}{n-1}
\left(-\frac{\nu}{\|u_{a}(x,t)\|L}\mathsf{RE}_{*}\right)^{n}\right)
\end{align}
Now for all terms after the summation with $n\ge 1$
\begin{align}
\lim_{\nu\rightarrow 0}\frac{2\beta-n}{n-1}
\left(-\frac{\nu}{\|u_{a}(x,t)\|L}\mathsf{RE}_{*}\right)^{n}=0
\end{align}
so the overall series reduces to zero. This leaves
\begin{align}
\begin{rcases}
&\lim_{\nu\rightarrow 0}\nu\left(\frac{\|{\mathscr{I}}_{a}(x,t)L}{\nu}-\mathsf{RE}_{*}\right)^{2\beta}\equiv\lim_{\nu\rightarrow 0}\frac{\nu(\|u_{a}\|L)^{2\beta}}{\nu^{2\beta}}=\|u_{a}(x,t)\|L >0 ~~\text{if}~~ \beta=\tfrac{1}{2} \nonumber \\&
\lim_{\nu\rightarrow 0}\nu\left(\frac{\|u_{a}(x,t)L}{\nu}-\mathsf{RE}_{*}\right)^{2\beta}\equiv \lim_{\nu\rightarrow 0}\frac{\nu(\|u_{a}\|L)^{2\beta}}{\nu^{2\beta}}=0 ~~\text{if}~~ \beta<\tfrac{1}{2} \nonumber
\end{rcases}
\end{align}
If $\beta>\tfrac{1}{2}$ then it will diverge. The criteria for the utilisation and convergence of the binomial series are also satisfied.
First, it is required that $\left|\tfrac{\nu}{\|u_{a}(x,t)\|L}\mathsf{RE}_{*}\right|<1$ which is $|\mathsf{RE}_{*}/\mathsf{RE}(x,t)|<1`$ or
$\mathsf{RE}(x,t)>\mathsf{RE}_{*}$ and this is always satisfied since we always have $\mathsf{RE}(x,t)\gg \mathsf{RE}_{*}$. The series will always converge by the ratio test, that is, if the ratio of successive or adjacent terms tends to a value which is less than 1, as $n\rightarrow \infty$ so that
\begin{align}
&\lim_{n\rightarrow\infty}\frac{\binom{2\beta}{n+1}\left(-\frac{\nu}{\|u_{a}(x,t)\|L}\mathsf{RE}_{*}\right)^{n+1}}{\binom{2\beta}{n}\left(-\frac{\nu}{\|
u_{a}(x,t)\|L}\mathsf{RE}_{*}\right)^{n}}\equiv
\lim_{n\rightarrow\infty}\frac{\binom{2\beta}{n+1}\left(-\frac{\nu}{\|u_{a}(x,t)\|L}\mathsf{RE}_{*}\right)^{n}\left(-\frac{\nu}{\|u_{a}(x,t)\|L}
{\mathsf{RE}}_{c}\right)}{\binom{2\beta}{n}\left(-\frac{\nu}{\|u_{a}(x,t)\|L}\mathsf{RE}_{*}\right)^{n}}\nonumber\\&
=\lim_{n\rightarrow\infty}\frac{\binom{2\beta}{n+1}}{\binom{2\beta}{n}}\left(-\frac{\nu}{\|u_{a}(x,t)\|L}\mathsf{RE}_{*}\right)
=\lim_{n\rightarrow\infty}\frac{2\beta-n}{n+1}\left(-\frac{\nu}{\|u_{a}(x,t)\|L}\mathsf{RE}_{*}\right)\nonumber\\&
=\lim_{n\rightarrow\infty}\left(\frac{2\beta}{n+1}-\frac{n}{n+1}\right)\left(-\frac{\nu}{\|u_{a}(x,t)\|L}\mathsf{RE}_{*}\right)\equiv
\lim_{n\rightarrow\infty}\left(\frac{2\beta}{n+1}-\frac{n}{n+1}\right)\left(\frac{-\mathsf{RE}_{*}}{\mathsf{RE}(x,t)}\right)\nonumber\\&
= \lim_{n\rightarrow\infty}\left(\frac{2\beta}{n+1}-\frac{1}{1+\frac{1}{n}}\right)\left(\frac{-{\mathsf{RE}}_{c}}{\mathsf{RE}(x,t)}\right)=\frac{\mathsf{RE}_{*}}{\mathsf{RE}(x,t)}<1
\end{align}
since $\lim_{n\rightarrow\infty}\left(\frac{2\beta}{n+1}-\frac{1}{1+\frac{1}{n}}\right)=-1$ so the ratio test criterion for convergence of the binomial series is always satisfied since $\mathsf{RE}(x,t)\gg \mathsf{RE}_{*}$
\end{proof}
In the introduction it was stated that anomalous dissipation requires that there is a blowup in the derivative or gradient of the velocity as $\nu\rightarrow 0$. We find that this holds for the random field ${\mathscr{U}_{a}(x,t)}$. Also, the random contribution to the field tends to zero as viscosity increases.
\begin{lem}
Given the random turbulence field $\mathscr{U}_{a}(x,t)$ then one has the blowups
\begin{align}
&\lim_{\nu\rightarrow 0}\mathlarger{\nabla}_{a}{\mathscr{U}_{a}(x,t)}=\infty\\&
\lim_{\nu\rightarrow 0}[\mathlarger{\nabla}_{a}\mathscr{U}_{a}(x,t)]=\infty\\&
\lim_{\nu\rightarrow 0}[\mathlarger{\Delta}\mathscr{U}_{a}(x,t)]=\infty
\end{align}
Since the fluid is incompressible with $\mathlarger{\nabla}^{a}u_{a}(x,t)=0$ and $\mathlarger{\nabla}^{a}\|u_{a}(x,t)\|=0$
\begin{align}
\mathlarger{\nabla}^{a}{\mathscr{U}}_{a}(x,t)&=\mathlarger{\nabla}^{a}u_{a}(x,t)+A\mathlarger{\nabla}^{a}u_{a}(x,t)(\mathsf{RE}(x,t)-\mathsf{RE}_{*})^{\beta}\sum_{I=1}^{\infty}
\mathrm{Z}^{1/2}_{I}f_{I}(x)\otimes\mathscr{Z}_{I}\nonumber\\&
+A u_{a}(x,t)\mathlarger{\nabla}^{a}(\mathsf{RE}(x,t)-\mathsf{RE}_{*})^{\beta}\sum_{I=1}^{\infty}
\mathrm{Z}^{1/2}_{I}f_{I}(x)\otimes\mathscr{Z}_{I}\nonumber\\&
+A u_{a}(x,t)(\mathsf{RE}(x,t)-\mathsf{RE}_{*})^{\beta}\sum_{I=1}^{\infty}
\mathrm{Z}^{1/2}_{I}\mathlarger{\nabla}^{a}f_{I}(x)\otimes\mathscr{Z}_{I}\nonumber\\&
=A u_{a}(x,t)\mathlarger{\nabla}^{a}(\mathsf{RE}(x,t)-\mathsf{RE}_{*})^{\beta}\sum_{I=1}^{\infty}
\mathrm{Z}^{1/2}_{I}f_{I}(x)\otimes\mathscr{Z}_{I}\nonumber\\&
+A u_{a}(x,t)(\mathsf{RE}(x,t)-\mathsf{RE}_{*})^{\beta}\sum_{I=1}^{\infty}
\mathrm{Z}^{1/2}_{I}\mathlarger{\nabla}^{a}f_{I}(x)\otimes\mathscr{Z}_{I}\nonumber\\&
=A u_{a}(x,t)\mathlarger{\nabla}^{a}\left(\frac{\|u_{a}(x,t)\|)L}{\nu}-\mathsf{RE}_{*}\right)^{\beta}\sum_{I=1}^{\infty}
\mathrm{Z}^{1/2}_{I}f_{I}(x)\otimes\mathscr{Z}_{I}\nonumber\\&
+A u_{a}(x,t)\left(\frac{\|u_{a}(x,t)\|L}{\nu}-\mathsf{RE}_{*}\right)^{\beta}\sum_{I=1}^{\infty}
\mathrm{Z}^{1/2}_{I}\mathlarger{\nabla}^{a}f_{I}(x)\otimes\mathscr{Z}_{I}\nonumber\\&
=A u_{a}(x,t)\left(\frac{\|u_{a}(x,t)\|)L}{\nu}-\mathsf{RE}_{*}\right)^{\beta-1}\mathlarger{\nabla}^{a}\|u_{a}(x,t)\|\frac{L}{\nu}\sum_{I=1}^{\infty}
\mathrm{Z}^{1/2}_{I}f_{I}(x)\otimes\mathscr{Z}_{I}\nonumber\\&
+A u_{a}(x,t)\left(\frac{\|u_{a}(x,t)\|L}{\nu}-\mathsf{RE}_{*}\right)^{\beta}\sum_{I=1}^{\infty}
\mathrm{Z}^{1/2}_{I}\mathlarger{\nabla}^{a}f_{I}(x)\otimes\mathscr{Z}_{I}
\end{align}
Hence
\begin{align}
&\lim_{\nu\rightarrow 0}\mathlarger{\nabla}^{a}{\mathscr{U}}_{a}(x,t)=A u_{a}(x,t)\left(\frac{\|u_{a}(x,t)\|)L}{\nu}-\mathsf{RE}_{*}\right)^{\beta-1}\mathlarger{\nabla}^{a}\|u_{a}(x,t)\|\frac{L}{\nu}\sum_{I=1}^{\infty}
\mathrm{Z}^{1/2}_{I}f_{I}(x)\otimes\mathscr{Z}_{I}\nonumber\\&
+\lim_{\nu\rightarrow 0}A u_{a}(x,t)\left(\frac{\|u_{a}(x,t)\|L}{\nu}-\mathsf{RE}_{*}\right)^{\beta}\sum_{I=1}^{\infty}
\mathrm{Z}^{1/2}_{I}\mathlarger{\nabla}^{a}f_{I}(x)\otimes\mathscr{Z}_{I}\nonumber\\&
=\lim_{\nu\rightarrow 0}\mathlarger{\nabla}^{a}{\mathscr{U}}_{a}(x,t)=A u_{a}(x,t)\left(\frac{\|u_{a}(x,t)\|)L}{\nu}\right)^{\beta-1}\mathlarger{\nabla}^{a}\|u_{a}(x,t)\|\frac{L}{\nu}\sum_{I=1}^{\infty}
\mathrm{Z}^{1/2}_{I}f_{I}(x)\otimes\mathscr{Z}_{I}\nonumber\\&
+\lim_{\nu\rightarrow 0}A u_{a}(x,t)\left(\frac{\|u_{a}(x,t)\|L}{\nu}\right)^{\beta}\sum_{I=1}^{\infty}
\mathrm{Z}^{1/2}_{I}\mathlarger{\nabla}^{a}f_{I}(x)\otimes\mathscr{Z}_{I}\nonumber\\&
=\lim_{\nu\rightarrow 0}A u_{a}(x,t)\left(\|u_{a}(x,t)\|)^{\beta-1}\right)\frac{L^{\beta}}{\nu^{\beta}}\mathlarger{\nabla}^{a}\|u_{a}(x,t)\|\sum_{I=1}^{\infty}
\mathrm{Z}^{1/2}_{I}f_{I}(x)\otimes\mathscr{Z}_{I}\nonumber\\& \nonumber\\&
+\lim_{\nu\rightarrow 0}A u_{a}(x,t)\left(\frac{\|u_{a}(x,t)\|L}{\nu}\right)^{\beta}\sum_{I=1}^{\infty}
\mathrm{Z}^{1/2}_{I}\mathlarger{\nabla}^{a}f_{I}(x)\otimes\mathscr{Z}_{I}=\infty
\end{align}
fort all $\beta>0$. Also
\begin{align}
\lim_{\nu\rightarrow 0}\mathlarger{\Delta}{\mathscr{U}}_{a}(x,t)=\lim_{\nu\rightarrow 0}A u_{a}(x,t)\left(\frac{\|u_{a}(x,t)\|L}{\nu}-\mathsf{RE}_{*}\right)^{\beta}\sum_{I=1}^{\infty}
\mathrm{Z}^{1/2}_{I}\mathlarger{\Delta} f_{I}(x)\otimes\mathscr{Z}_{I}=\infty
\end{align}
which gives (5.23).
\end{lem}
The main theorem  establishing the zeroth law and anomalous dissipation is now as follows.
\begin{thm}\textbf{(Anomalous Dissipation and a zeroth law)}\newline
Let ${\mathfrak{G}}\subset\mathbb{R}^{3}$ contain an isotropic and incompressible fluid of low viscosity $\nu$ and velocity $u_{a}(x,t)$
for all $X\equiv(x,t)\in{\mathfrak{G}}\otimes\mathbb{R}^{+}$. The volume of the domain is $vol({\mathfrak{G}})=\int_{\mathfrak{G}}d\mathcal{V}(x)=\mathrm{Z}^{3}$. The following hold:
\begin{enumerate}
\item The underlying deterministic flow $u_{a}(x,t)$ satisfies the Navier-Stokes equations
\begin{align}
&\frac{\partial}{\partial t}u_{a}(x,t)-\nu \mathlarger{\Delta} u_{a}(x,t)+u^{b}(x,t)\mathlarger{\nabla}_{b}u_{a}(x,t)+\mathlarger{\nabla}_{a}p(x,t)\nonumber\\&
\equiv\frac{\partial}{\partial t}u_{a}(x,t)-\nu \mathlarger{\Delta} u_{a}(x,t)+U^{b}(x,t)\mathlarger{\nabla}_{b}u_{a}(x,t)+\mathlarger{\nabla}_{a}p(x,t)
=0,~~(x,t)\in{\mathfrak{G}}\otimes\mathbb{R}^{+}
\end{align}
Deterministic means he flow evolves from some initial Cauchy data $v^{o}(x,0)=\mathbf{g}(x)$ with some suitable boundary conditions on $\partial{\mathfrak{G}}$. If the fluid is incompressible and isotropic then $ \mathlarger{\nabla}_{b}u^{b}(x,t)=0$ and $\mathlarger{\nabla}_{b}u_{a}(x,t)=\mathlarger{\nabla}_{a}U{b}(x,t)$. A laminar flow $u_{a}(x,t)=u_{a}=(0,0,U)$ is then a trivial and steady state solution of the NS equations.
\item The Reynolds number or function within ${\mathfrak{G}}$ for fixed $L$ is
\begin{align}
\mathsf{RE}(x,t)=\frac{\|u_{a}(x,t)\| L}{\nu}
\end{align}
and evolves as $u_{a}(x,t)$. For fixed $L$ and $\nu$ one has a fixed Reynolds number $\mathsf{RE}(x,t)=\frac{\|u_{a}\| L}{\nu}
=\frac{v L}{\nu} $. Since $\nu\gtrapprox 0$ then the Reynolds number is very high but not infinite which ensures that $\mathsf{RE}\gg \mathsf{RE}_{*}$ where
${\mathsf{RE}}_{*}$ is a critical Reynolds number. Then $\mathlarger{\nabla}_{b}\mathsf{RE}=0$
\item The energy dissipation rate is $\epsilon(x,t)=\nu|\mathlarger{\nabla}_{b}u_{a}(x,t)|^{2}$. When $u_{a}(x,t)=u_{a}=(0,0,v)$ then $\|u_{a}\|^{3}=v^{2}=\epsilon L$.
\item The random field ${\mathscr{T}}(x)$ is a zero-centred differentiable and regulated Gaussian random field with a regulated kernel the properties
$\mathbf{E}[{\mathscr{T}}(x)]=0, \mathbf{E}[{\mathscr{T}}(x)\otimes{\mathscr{T}}(y)]=K(x,y;\lambda )$ and
$\mathbf{E}[{\mathscr{T}}(x)\otimes{\mathscr{T}}(x)]={C}$ and $\mathbf{E}\big[\mathlarger{\nabla}_{b}\bm{\mathscr{T}}(x)\otimes\bm{\mathscr{T}}(x)]$
=$\mathbf{E}\big[\bm{\mathscr{T}}(x)\otimes\mathlarger{\nabla}_{b}{\mathscr{T}}(x)\big]=0,\mathbf{E}\big[\mathlarger{\nabla}_{b}\bm{\mathscr{T}}(x)\otimes \mathlarger{\nabla}^{b}
\bm{\mathscr{T}}(x)\big]=1$ where $\lambda\ll L$ is a correlation length for the kernel.
\item The random Gaussian field $\bm{\mathscr{T}}(x)$ has the Kahunen-Loeve spectral expansion\newline
$ {\mathscr{T}}(x)=\sum_{I=1}^{\infty}{\mathrm{Z}^{1/2}_{I}}f_{I}(x)\otimes\mathscr{Z}_{I}$.
\item The following integral identities hold for the eigenfunctions and their derivatives.
\begin{align}
&\bm{\mathbf{H}}(\mathfrak{G},I)=\int_{{\mathfrak{G}}}f_{I}(x)f_{I}(x)d\mathcal{V}(x)=1\nonumber\\&
\bm{\mathbf{H}}_{a}(\mathfrak{G},I)=\int_{{\mathfrak{G}}}f_{I}(x)\mathlarger{\nabla}^{a}f_{I}(x)d\mathcal{V}(x)=0\nonumber\\&
\bm{\mathbf{H}}_{a}^{a}(\mathfrak{G},I)=\int_{{\mathfrak{G}}}\mathlarger{\nabla}_{a}f_{I}(x)\mathlarger{\nabla}^{a}f_{I}(x)d\mathcal{V}(x)>0\nonumber
\end{align}
\item The turbulent velocity or flow is given by the weighted random vectorial field as before
\begin{align}
&{\mathscr{U}}_{a}(x,t)=u_{a}(x,t)+{A}u_{a}(x,t)\big(|\mathsf{RE}(x,t)-\mathsf{RE}_{*}|
\big)^{\beta}{\mathlarger{\chi}}[\mathsf{RE}(x,t)]\sum_{I=1}^{\infty}\mathrm{Z}^{1/2}_{I}f_{I}(x)\otimes\mathscr{Z}_{I}\nonumber\\&
\equiv u_{a}(x,t)+{\mathscr{V}}(x)
\end{align}
with $\beta\le \tfrac{1}{2}$ so that the amplitude and random fluctuations grow as the Reynolds number increases. For very turbulent flows $\mathsf{RE}(x,t)\gg \mathsf{RE}_{*}$ so that ${\mathlarger{\chi}}[\mathsf{RE}(x,t)]=1$ and
\begin{align}
&{\mathscr{U}}_{a}(x,t)=u_{a}(x,t)+A\big(|\mathsf{RE}(x,t)-\mathsf{RE}_{*}\big)^{\beta}u_{a}(x,t)
\otimes\mathscr{T}(x)
\end{align}
The mean flow is then $\mathbf{E}[{u_{a}(x,t)}]=u_{a}(x,t)$.
\end{enumerate}
Now taking the limit of zero viscosity, and constant underlying fluid velocity,a zeroth-type law of anomalous dissipation holds if
\begin{align}
\lim_{\nu\rightarrow 0}\lim_{u_{a}(x,t)\rightarrow u_{a}}\inf \mathbf{E}\left[\nu\int_{\mathfrak{G}}\int_{0}^{T}|
\mathlarger{\nabla}_{b}{\mathscr{U}}_{a}(x,t)|^{2}
d\mathcal{V}(x) ds\right] >0
\end{align}
or
\begin{align}
\lim_{\nu\rightarrow 0}\lim_{u_{a}(x,t)\rightarrow u_{a}}\inf~\nu\int_{\mathfrak{G}}\int_{0}^{T}|\mathlarger{\nabla}_{b}
\mathbf{E}\left[{\mathscr{U}}_{a}(x,t)|^{2}\right] d\mathcal{V}(x)ds >0
\end{align}
Specifically,
\begin{align}
&\lim_{\nu\rightarrow 0}\lim_{\mathscr{U}_{a}(x,t)\rightarrow u_{a}}\inf~\nu \int_{\mathfrak{G}}\int_{0}^{T}|{\bm{{{\bm{\mathbf{E}}}}}}\left[\mathlarger{\nabla}_{b}
{{\mathscr{U}}_{a}(x,s)}|^{2}\right]d\mathcal{V}(x) ds\nonumber\\&=\sup {A}^{2}\left\lbrace\lim_{\nu\rightarrow 0}\nu(\mathsf{RE}(x,t)-\mathsf{RE}_{*})^{2\beta}\right\rbrace
u_{a}u_{a}T\sum_{I=1}^{\infty}\mathrm{Z}_{I}\int_{\mathfrak{G}}
\mathlarger{\nabla}_{b}f_{I}(x)\mathlarger{\nabla}^{b}f_{I}(x)d\mathcal{V}(x)
\nonumber\\&=\sup {A}^{2}\lim_{\nu\rightarrow 0}\left\lbrace\nu\left(\frac{\|u_{a}\|L}{\nu}-\mathsf{RE}_{*}\right)^{2\beta}\right\rbrace
u_{a}u_{a}T\sum_{I=1}^{\infty}\mathrm{Z}_{I}\int_{\mathfrak{G}}
\mathlarger{\nabla}_{b}f_{I}(x)\mathlarger{\nabla}^{b}f_{I}(x)d\mathcal{V}(x)\nonumber\\&
\equiv \sup {A}^{2}\|u_{a}\|^{3}L\|T\sum_{I=1}^{\infty}\mathrm{Z}_{I}\int_{\mathfrak{G}}
\mathlarger{\nabla}_{b}f_{I}(x)\mathlarger{\nabla}^{b}f_{I}(x)d\mathcal{V}(x)>0
\end{align}
for all $T>0$ iff $\beta=\tfrac{1}{2}$ and $\int_{\mathfrak{G}}\mathlarger{\nabla}_{b}f_{I}(x)\mathlarger{\nabla}^{b}f_{I}(x)d\mathcal{V}(x)>0$.
\end{thm}
\begin{proof}\underline{(Brute force version.)}\newline The derivative of the random field ${\mathscr{U}_{a}(x,t)}$ is
\begin{align}
&\mathlarger{\nabla}_{b}{\mathscr{U}}_{a}(x,t)=\mathlarger{\nabla}_{b}u_{a}(x,t)+{A}\mathlarger{\nabla}_{b}u_{a}(x,t)(\mathsf{RE}(x,t)-\mathsf{RE}_{*})^{\beta}
\sum_{I=1}^{\infty}{\mathrm{Z}^{1/2}_{I}}f_{I}(x)\otimes{\mathscr{Z}}_{I}\nonumber\\&
+A u_{a}(x,t)\mathlarger{\nabla}_{b}(\mathsf{RE}(x,t)-\mathsf{RE}_{*})^{\beta}\sum_{I=1}^{\infty}{\mathrm{Z}^{1/2}_{I}}f_{I}(x)\otimes\mathscr{Z}_{I}
\nonumber\\&+{A}u_{a}(x,t)(\mathsf{RE}(x,t)-\mathsf{RE}_{*})^{\beta}\sum_{I=1}^{\infty}{\mathrm{Z}^{1/2}_{I}}\mathlarger{\nabla}_{b}
f_{I}(x)\otimes\mathscr{Z}_{I}
\end{align}
Then
\begin{align}
&\sup~\nu\int_{\mathfrak{G}}\int_{0}^{T}|\mathlarger{\nabla}_{b}{u_{a}(x,t)}|^{2}d{\mu}_{3}(x)ds=\sup~\nu \int_{\mathfrak{G}}\int_{0}^{T}\mathlarger{\nabla}_{\mu}u_{a}(x,t)\mathlarger{\nabla}^{b}u_{a}(x,t)\nonumber\\&
+\sup~\nu{A}\int_{\mathfrak{G}}\int_{0}^{T}\mathlarger{\nabla}_{b}u_{a}(x,t)\mathlarger{\nabla}^{b}
u^{\alpha}(x,t)(\mathsf{RE}(x,t)-\mathsf{RE}_{*})^{\beta}\sum_{I=1}^{\infty}{\mathrm{Z}^{1/2}_{I}}f_{I}(x)\otimes\mathscr{Z}_{I}d{\mu}_{3}(x)ds\nonumber\\&
+\sup~\nu{A}\int_{\mathfrak{G}}\int_{0}^{T}\mathlarger{\nabla}_{b}u_{a}(x^{\alpha})u_{a}(x,t)\mathlarger{\nabla}^{b}\big(\mathsf{RE}(x,t)-\mathsf{RE}_{*}\big)^{\beta}\sum_{I=1}^{\infty}
{\mathrm{Z}^{1/2}_{I}}f_{I}(x)\otimes\mathscr{Z}_{I}d{\mu}_{3}(x)ds\nonumber\\&
+\sup~\nu{A}\int_{\mathfrak{G}}\int_{0}^{T}\mathlarger{\nabla}_{b}u_{a}(x,t)u_{a}(x,t)(\mathsf{RE}(x,t)-\mathsf{RE}_{*})^{\beta}\sum_{I=1}^{\infty}{\mathrm{Z}^{1/2}_{I}}
f_{I}(x)\otimes\mathscr{Z}_{I}d{\mu}_{3}(x)ds
\nonumber\\&+\sup~\nu{A}\int_{\mathfrak{G}}\int_{0}^{T}\mathlarger{\nabla}_{a}u_{b}(x,t)\mathlarger{\nabla}^{b}u_{a}(x,t)(\mathsf{RE}(x,t)-\mathsf{RE}_{*})^{\beta}
\sum_{I=1}^{\infty}{\mathrm{Z}^{1/2}_{I}}f_{I}(x)\otimes{\mathscr{Z}}_{I}d{\mu}_{3}(x)ds\nonumber\\&
+\sup~\nu {A}^{2}\int_{\mathfrak{G}}\int_{0}^{T}\mathlarger{\nabla}_{b}u_{a}(x,t)\mathlarger{\nabla}^{b}u_{a}(x,t)(\mathsf{RE}(x,t)-\mathsf{RE}_{*})^{\beta}
\sum_{I=1}^{\infty}{\mathrm{Z}_{I}}f_{I}(x)f_{I}(x)\bigg(\mathscr{Z}_{I}\otimes\mathscr{Z}_{I}\bigg)d{\mu}_{3}(x)ds
\nonumber\\&+\sup~\nu {A}^{2}\int_{\mathfrak{G}}\int_{0}^{T}\mathlarger{\nabla}_{b}u_{a}(x,t)u_{a}(x,t)(\mathsf{RE}(x,t)-\mathsf{RE}_{*})^{2\beta}\sum_{I=1}^{\infty}
\mathrm{Z}_{I}f_{I}(x)f_{I}(x)\bigg(\mathscr{Z}_{I}\otimes\mathscr{Z}_{I}\bigg)d{\mu}_{3}(x)ds\nonumber\\&
+\sup~\nu {A}^{2}\int_{\mathfrak{G}}\int_{0}^{T}\mathlarger{\nabla}_{b}u_{a}(x,t)u_{a}(x,t)(\mathsf{RE}(x,t)-\mathsf{RE}_{*})^{2\beta}
\sum_{I=1}^{\infty}{\mathrm{Z}_{I}}f_{I}(x)\mathlarger{\nabla}^{b}f_{I}(x)\bigg(\mathscr{Z}_{I}\otimes\mathscr{Z}_{I}\bigg) d{\mu}_{3}(x)ds\nonumber\\&
\sup~\nu {A}\int_{\mathfrak{G}}\int_{0}^{T}u_{a}(x,t)\mathlarger{\nabla}^{b}u^{a}(x,t)\mathlarger{\nabla}_{b}(\mathsf{RE}(x,t)-\mathsf{RE}_{*})^{\beta}
\sum_{I=1}^{\infty}{\mathrm{Z}^{1/2}_{I}}f_{I}(x)\otimes\mathscr{Z}_{I}d{\mu}_{3}(x)ds\nonumber\\&
+\sup~\nu {A}^{2}\int_{\mathfrak{G}}\int_{0}^{T}u_{a}(x,t)\mathlarger{\nabla}_{b}(\mathsf{RE}(x,t)-\mathsf{RE}_{*})^{2\beta}\mathlarger{\nabla}^{b}u_{a}(x,t)
\sum_{I=1}^{\infty}\mathrm{Z}_{I}f_{I}(x)f_{I}(x)\bigg(\mathscr{Z}_{I}\otimes\mathscr{Z}_{I}\bigg)d\mathcal{V}(x)ds\nonumber\\&
+\sup~\nu {A}^{2}\int_{\mathfrak{G}}\int_{0}^{T}u_{a}(x,t)u_{a}(x,t)\mathlarger{\nabla}_{b}
(\mathsf{RE}(x,t)-\mathsf{RE}_{*})^{\beta}\mathlarger{\nabla}^{b}(\mathsf{RE}(x,t)-\mathsf{RE}_{*})^{\beta}\nonumber\\&\otimes \sum_{I=1}^{\infty}f_{I}(x)f_{I}(x)
\bigg(\mathscr{Z}_{I}\otimes\mathscr{Z}_{I}\bigg) d{\mu}_{3}(x)ds
\nonumber\\&+\sup~\nu {A}^{2}\int_{\mathfrak{G}}\int_{0}^{T}u_{a}(x,t)u_{a}(x,t)(\mathsf{RE}(x,t)-\mathsf{RE}_{*})^{\beta}
\mathlarger{\nabla}_{b}(\mathsf{RE}(x,t)-\mathsf{RE}_{*})^{\beta}\nonumber\\&\otimes\sum_{I=1}^{\infty}\mathrm{Z}_{I}f_{I}(x)\mathlarger{\nabla}^{b}f_{I}(x)
\bigg(\mathscr{Z}_{I}\otimes\mathscr{Z}_{I}\bigg)d\mathcal{V}(x)ds\nonumber\\&
+\sup~\nu {A}\int_{\mathfrak{G}}\int_{0}^{T}u_{a}(x,t)\mathlarger{\nabla}^{b}u_{a}(x,t)(\mathsf{RE}(x,t)-\mathsf{RE}_{*})
\sum_{I=1}^{\infty}{\mathrm{Z}^{1/2}_{I}}\mathlarger{\nabla}_{b}f_{I}(x)\otimes\mathscr{Z}_{I}\nonumber\\&
+\sup~\nu {A}^{2}\int_{\mathfrak{G}}\int_{0}^{T}u_{a}(x,t)\mathlarger{\nabla}^{b}u_{a}(x,t)
\mathsf{RE}(x,t)-\mathsf{RE}_{*})^{\beta}\sum_{I=1}^{\infty}\mathrm{Z}_{I}\mathlarger{\nabla}_{b}f_{I}(x)f_{I}(x)\bigg({\mathscr{Z}}_{I}
\otimes\mathscr{Z}_{I}\bigg)d\mathcal{V}(x)ds\nonumber\\&\sup~\nu {A}^{2}\int_{\mathfrak{G}}\int_{0}^{T}u_{a}(x,t)u_{a}(x,t)
(\mathsf{RE}(x,t)-\mathsf{RE}_{*})^{\beta}\sum_{I=1}^{\infty}\mathrm{Z}_{I}\mathlarger{\nabla}_{b}f_{I}(x)
f_{I}(x)\bigg({\mathscr{Z}}_{I}\otimes{\mathscr{Z}}_{I}\bigg)d\mathcal{V}(x)ds\nonumber\\&
+\sup~\nu {A}^{2}\int_{\mathfrak{G}}\int_{0}^{T}u_{a}(x,t)(\mathsf{RE}(x,t)-\mathsf{RE}_{*})u_{a}(x,t)\sum_{I=1}^{\infty}f_{I}
\mathlarger{\nabla}_{b}f_{I}(x)\mathlarger{\nabla}^{b}f_{I}(x)\bigg(\mathscr{Z}_{I}\otimes\mathscr{Z}_{I}\bigg)d\mathcal{V}(x)ds
\end{align}
Now take the stochastic expectation $\mathbf{E}[...]$. Since $\mathbf{E}[\mathscr{Z}_{I}]=0$ and $\mathbf{E}[\mathscr{Z}_{I}
\otimes\mathscr{Z}_{I}]=1$ only the underbraced terms remain
\begin{align}
&\sup~\nu{\bm{{{\bm{\mathbf{E}}}}}}\left[\int_{\mathfrak{G}}\int_{0}^{T}|\mathlarger{\nabla}_{b}{\mathscr{U}}_{a}(x,t)|^{2}\right]d\mathcal{V}(x)ds
=\sup~\nu\int_{\mathfrak{G}}\int_{0}^{T}{\bm{{{\bm{\mathbf{E}}}}}}[\mathlarger{\nabla}_{b}{\mathscr{U}}_{a}(x,t)\otimes \mathlarger{\nabla}^{b}{\mathscr{U}}^{i}(x,t)]
d\mathcal{V}(x) ds\nonumber\\&=\sup~\nu \int_{\mathfrak{G}}\int_{0}^{T}\mathlarger{\nabla}_{b}u_{a}(x,t)\mathlarger{\nabla}^{b}u_{a}(x,t)\nonumber\\&
+\sup~\nu {A}\int_{\mathfrak{G}}\int_{0}^{T}\mathlarger{\nabla}_{b}u_{a}(x,t)\mathlarger{\nabla}^{b}u_{a}(x,t)\left(\mathsf{RE}(x,t)-\mathsf{RE}_{*}\right)^{\beta}\sum_{I=1}^{\infty}
\mathrm{Z}_{I}^{1/2}f_{I}(x){\bm{{{\bm{\mathbf{E}}}}}}[\mathscr{Z}_{I}]d\mathcal{V}(x)ds\nonumber\\&+\sup~\nu {A}\int_{0}^{T}\!\!\!\!\int_{\mathfrak{G}}
\mathlarger{\nabla}_{b}u_{a}(x,t)u_{a}(x,t)\mathlarger{\nabla}^{b}(\mathsf{RE}(x,t)-\mathsf{RE}_{*})^{\beta}\sum_{I=1}^{\infty}{\mathrm{Z}^{1/2}_{I}}f_{I}(x){\bm{{{\bm{\mathbf{E}}}}}}\big[\mathscr{Z}_{I}\big]
d\mathcal{V}(x)ds\nonumber\\&+\sup~\nu {A}\int_{\mathfrak{G}}\int_{0}^{T}\mathlarger{\nabla}_{b}u_{a}(x,t)u_{a}(x,t)(\mathsf{RE}(x,t)-\mathsf{RE}_{*})^{\beta}\sum_{I=1}^{\infty}
{\mathrm{Z}^{1/2}_{I}}f_{I}(x){{\bm{{{\bm{\mathbf{E}}}}}}}\big[\mathscr{Z}_{I}\big]d\mathcal{V}(x)ds\nonumber\\&+\sup~\nu {A}\int_{\mathfrak{G}}\int_{0}^{T}
\mathlarger{\nabla}_{a}U{b}(x,t)\mathlarger{\nabla}^{b}u_{a}(x,t)(\mathsf{RE}(x,t)-\mathsf{RE}_{*})^{\beta}\sum_{I=1}^{\infty}
{\mathrm{Z}^{1/2}_{I}}f_{I}(x){\bm{{{\bm{\mathbf{E}}}}}}\big[\mathscr{Z}_{I}\big]d\mathcal{V}(x)ds\nonumber\\&
+\underbrace{\sup~\nu {A}^{2}\int_{\mathfrak{G}}\int_{0}^{T}\mathlarger{\nabla}_{b}u_{a}(x,t)\mathlarger{\nabla}^{b}u_{a}(x,t)({\mathsf{RE}(x,t)-\mathsf{RE}_{*}})^{2\beta}\sum_{I=1}^{\infty}
{\mathrm{Z}_{I}}f_{I}(x)f_{I}(x){\bm{{{\bm{\mathbf{E}}}}}}\big[\mathscr{Z}_{I}\otimes\mathscr{Z}_{I}\big]d\mathcal{V}(x)ds}\nonumber\\&
+\underbrace{\sup~\nu {A}^{2}\int_{\mathfrak{G}}\int_{0}^{T}\mathlarger{\nabla}_{b}u_{a}(x,t)u_{a}(x,t)({\mathsf{RE}(x,t)-\mathsf{RE}_{*}})^{2\beta}
\sum_{I=1}^{\infty}\mathrm{Z}_{I}^{1/2}f_{I}(x)f_{I}(x){\bm{{{\bm{\mathbf{E}}}}}}\big[\mathscr{Z}_{I}\otimes\mathscr{Z}_{I}\big]
d\mathcal{V}(x)ds}\nonumber\\&+\underbrace{\sup~\nu A^{2}\int_{0}^{T}\!\!\!\!\int_{\mathfrak{G}}\mathlarger{\nabla}_{b}u_{a}(x,t)u_{a}(x,t)
(\mathsf{RE}(x,t)-\mathsf{RE}_{*})^{2\beta}\sum_{I=1}^{\infty}\mathrm{Z}_{I}f_{I}(x)\mathlarger{\nabla}^{b}f_{I}(x){\bm{{{\bm{\mathbf{E}}}}}}
\big[\mathscr{Z}_{I}\otimes\mathscr{Z}_{I}\big] d\mathcal{V}(x)ds}\nonumber\\&\sup~\nu {A}\int_{\mathfrak{G}}\int_{0}^{T}u_{a}(x,t)\mathlarger{\nabla}^{b}u_{a}(x,t)\mathlarger{\nabla}_{b}(\mathsf{RE}(x,t)-\mathsf{RE}_{*})^{\beta}
\sum_{I=1}^{\infty}\mathrm{Z}_{I}^{1/2}f_{I}(x){\bm{{{\bm{\mathbf{E}}}}}}\big[{\mathscr{Z}}_{I}\big]d\mathcal{V}(x)ds\nonumber\\&
+\underbrace{\sup~\nu {A}^{2}\int_{\mathfrak{G}}\int_{0}^{T}u_{a}(x,t)\mathlarger{\nabla}_{b}(\mathsf{RE}(x,t)-\mathsf{RE}_{*})^{2\beta}\mathlarger{\nabla}^{b}u_{a}(x,t)
\sum_{I=1}^{\infty}\mathrm{Z}_{I}f_{I}(x)f_{I}(x){\bm{{{\bm{\mathbf{E}}}}}}\big[\mathscr{Z}_{I}\otimes\mathscr{Z}_{I}\big]d{\mu}_{3}(x)ds}
\nonumber\\&+\underbrace{\sup~\nu {A}^{2}\int_{\mathfrak{G}}\int_{0}^{T}u_{a}(x,t)u_{a}(x,t)\mathlarger{\nabla}_{b}
(\mathsf{RE}(x,t)-\mathsf{RE}_{*})^{\beta}\mathlarger{\nabla}^{b}({\mathsf{RE}(x,t)}-\mathsf{RE}_{*}})^{\beta}\nonumber\\&\otimes\underbrace{\sum_{I=1}^{\infty}\mathrm{Z}_{I}f_{I}(x)
f_{I}(x)\mathbf{E}\big[\mathscr{Z}_{I}\otimes\mathscr{Z}_{I}\big]d\mathcal{V}(x)ds}
\nonumber\\&+\underbrace{\sup~\nu {A}^{2}\int_{\mathfrak{G}}\int_{0}^{T}u_{a}(x,t)u_{a}(x,t)(\mathsf{RE}(x,t)-\mathsf{RE}_{*})^{\beta}(x,t)\mathlarger{\nabla}_{b}
(\mathsf{RE}(x,t)-\mathsf{RE}_{*})^{\beta}}\nonumber\\&  ~~~~~~~~~~~\otimes\underbrace{\sum_{I=1}^{\infty}\mathrm{Z}_{I}f_{I}(x)\mathlarger{\nabla}^{b}f_{I}(x)
\mathbf{E}\big[\mathscr{Z}_{I}\otimes\mathscr{Z}_{I}\big]d\mathcal{V}(x)ds}\nonumber\\& +\sup~\nu {A}\int_{\mathfrak{G}}\int_{0}^{T}u_{a}(x,t)\mathlarger{\nabla}^{b}u_{a}(x,t)({\mathsf{RE}}(x,t)-\mathsf{RE}_{*})^{\beta}
\sum_{I=1}^{\infty}{\mathrm{Z}^{1/2}_{I}}\mathlarger{\nabla}_{b}f_{I}(x){\bm{{{\bm{\mathbf{E}}}}}}\big[\mathscr{Z}_{I}\big]
\end{align}
The remaining terms are then
\begin{align}
&\sup~\nu{\bm{{{\bm{\mathbf{E}}}}}}\left[\int_{\mathfrak{G}}\int_{0}^{T}|\mathlarger{\nabla}_{b}{\bm{\mathscr{U}}_{a}(x,t)}|^{2}\right]d\mathcal{V}(x)ds\nonumber\\&=\sup~\nu\int_{0}^{T}
\int_{\mathfrak{G}}\mathbf{E}[\mathlarger{\nabla}_{b}{\bm{\mathscr{U}}_{a}(x,t)}\otimes \mathlarger{\nabla}^{b}{\bm{\mathscr{U}}^{a}(x,t)}]d\mathcal{V}(x) ds\nonumber\\&=\sup~\nu \int_{\mathfrak{G}}\int_{0}^{T}\mathlarger{\nabla}_{b}u_{a}(x,t)\mathlarger{\nabla}^{b}u_{a}(x,t)\nonumber\\&
+{\sup~\nu {A}^{2}\int_{\mathfrak{G}}\int_{0}^{T}\mathlarger{\nabla}_{b}u_{a}(x,t)\mathlarger{\nabla}^{b}u_{a}(x,t)(\mathsf{RE}(x,t)-\mathsf{RE}_{*})^{2\beta}\sum_{I=1}^{\infty}\mathrm{Z}_{I}
f_{I}(x)f_{I}(x)d\mathcal{V}(x)ds}\nonumber\\&
+{\sup~\nu{A}^{2}\int_{\mathfrak{G}}\int_{0}^{T}\mathlarger{\nabla}_{b}u_{a}(x,t)u_{a}(x,t)(\mathsf{RE}(x,t)-\mathsf{RE}_{*})^{2\beta}\sum_{I=1}^{\infty}\mathrm{Z}_{I}
f_{I}(x)f_{I}(x)d\mathcal{V}(x)ds}\nonumber\\&
+{\sup~\nu{A}^{2}\int_{\mathfrak{G}}\int_{0}^{T}\mathlarger{\nabla}_{b}u_{a}(x,t)u_{a}(x,t)({\mathsf{RE}(x,t)-\mathsf{RE}})^{2\beta}\sum_{I=1}^{\infty}\mathrm{Z}_{I}
f_{I}(x)\mathlarger{\nabla}^{b}f_{I}(x)d\mathcal{V}(x)ds}
\nonumber\\&+{\sup~\nu {A}^{2}\int_{\mathfrak{G}}\int_{0}^{T}u_{a}(x,t)\mathlarger{\nabla}_{b}(\mathsf{RE}(x,t)-\mathsf{RE}_{*})^{2\beta}\mathlarger{\nabla}^{b}u_{a}(x,t)\sum_{I=1}^{\infty}
\mathrm{Z}_{I}f_{I}(x)f_{I}(x)d\mathcal{V}(x)ds}\nonumber\\&
+{\sup~\nu
{A}^{2}\int_{\mathfrak{G}}\int_{0}^{T}u_{a}(x,t)u_{a}(x,t)\mathlarger{\nabla}_{b}(\mathsf{RE}(x,t)-\mathsf{RE}_{*})^{\beta}\mathlarger{\nabla}^{b}(\mathsf{RE}(x,t)-\mathsf{RE}_{*})^{\beta}\sum_{I=1}^{\infty}
\mathrm{Z}_{I}f_{I}(x)f_{I}(x)d\mathcal{V}(x)ds}\nonumber\\&
+{\sup~\nu {A}^{2}\int_{\mathfrak{G}}\int_{0}^{T}u_{a}(x,t)u_{a}(x,t)(\mathsf{RE}(x,t)-\mathsf{RE}_{*})^{2\beta}\mathlarger{\nabla}_{b}(\mathsf{RE}(x,t)-\mathsf{RE}_{*})^{2\beta}\sum_{I=1}^{\infty}
\mathrm{Z}_{I}f_{I}(x)\mathlarger{\nabla}^{b}f_{I}(x)d\mathcal{V}(x)ds}\nonumber\\&
+{\sup~\nu {A}^{2}\int_{0}^{T}\!\!\!\!\int_{\mathfrak{G}}u_{a}(x,t)\mathlarger{\nabla}^{b}u_{a}(x,t)(\mathsf{RE}(x,t)-\mathsf{RE}_{*})^{\beta}\sum_{I=1}^{\infty}
\mathrm{Z}_{I}\mathlarger{\nabla}_{b}f_{I}(x)f_{I}(x)
d\mathcal{V}(x)ds}\nonumber\\&
{\sup~\nu {A}^{2}\int_{\mathfrak{G}}\int_{0}^{T}u_{a}(x,t)u_{a}(x,t)(\mathsf{RE}(x,t)-\mathsf{RE}_{*})^{2\beta}\sum_{I=1}^{\infty}\mathrm{Z}_{I}
\mathlarger{\nabla}_{b}f_{I}(x)f_{I}(x)d\mathcal{V}(x)ds}\nonumber\\&
+{\sup~\nu {A}^{2}\int_{\mathfrak{G}}\int_{0}^{T}u_{a}(x,t)(\mathsf{RE}(x,t)-\mathsf{RE}_{*})^{2\beta}u_{a}(x,t)\sum_{I=1}^{\infty}\mathrm{Z}_{I}\mathlarger{\nabla}_{b}
f_{I}(x)\mathlarger{\nabla}^{b}f_{I}(x)d\mathcal{V}(x)ds}
\end{align}
One now takes the limit that $u_{a}(x,t)\rightarrow u_{a}$ and $\mathsf{RE}(x,t)\rightarrow \mathsf{RE}$. Then $\mathlarger{\nabla}_{b}u_{a}=\mathlarger{\nabla}^{b}u_{a}=0,\mathlarger{\nabla}_{b}
\mathsf{RE}=0$. Then only the last two underbraced terms survive
\begin{align}
&\lim_{u_{a}(x,t)\rightarrow u_{a}}\sup~\nu{\bm{{{\bm{\mathbf{E}}}}}}\left[\int_{\mathfrak{G}}\int_{0}^{T}|\mathlarger{\nabla}_{b}{\mathscr{U}_{a}(x,t)}|^{2}\right]d\mathcal{V}(x)ds
\nonumber\\&=\sup~\nu\int_{\mathfrak{G}}\int_{0}^{T}{\bm{{{\bm{\mathbf{E}}}}}}[\mathlarger{\nabla}_{b}\bm{\mathscr{U}}_{a}(x,t)\otimes \mathlarger{\nabla}^{b}\bm{\mathscr{U}}^{i}(x,t)]d\mathcal{V}(x) ds\nonumber\\&=\lim_{{U}_{a}(x,t)\rightarrow u_{a}}\sup~\nu \int_{\mathfrak{G}}\int_{0}^{T}\mathlarger{\nabla}_{b}u_{a}(x,t)\mathlarger{\nabla}^{b}u_{a}(x,t)\nonumber\\&
+\lim_{u_{a}(x,t)\rightarrow u_{a}}{\sup~\nu {A}^{2}\int_{\mathfrak{G}}\int_{0}^{T}\mathlarger{\nabla}_{b}u_{a}(x,t)\mathlarger{\nabla}^{b}
u_{a}(x,t)(\mathsf{RE}(x,t)-\mathsf{RE}_{*})^{2\beta}\sum_{I=1}^{\infty}\mathrm{Z}_{I}f_{I}(x)f_{I}(x)d\mathcal{V}(x)ds}\nonumber\\&
+\lim_{u_{a}(x,t)\rightarrow u_{a}}{\sup~\nu A^{2}\int_{\mathfrak{G}}\int_{0}^{T}\mathlarger{\nabla}_{b}u_{a}(x,t){u}^{a}(x,t)(\mathsf{RE}(x,t)-\mathsf{RE}_{*})^{2\beta}
\sum_{I=1}^{\infty}\mathrm{Z}_{I}f_{I}(x)f_{I}(x)d\mathcal{V}(x)ds}\nonumber\\&
+\lim_{u_{a}(x,t)\rightarrow u_{a}}\sup~\nu {A}^{2}\int_{\mathfrak{G}}\int_{0}^{T}\mathlarger{\nabla}_{b}u_{a}(x,t)u_{a}(x,t)(\mathsf{RE}(x,t)-\mathsf{RE}_{*})^{2\beta}
\sum_{I=1}^{\infty}\mathrm{Z}_{I}f_{I}(x)f_{I}(x)\mathlarger{\nabla}^{b}f_{I}(x)d\mathcal{V}(x)ds
\nonumber\\&+\lim_{u_{a}(x,t)\rightarrow u_{a}}{\sup~\nu {A}^{2}\int_{\mathfrak{G}}\int_{0}^{T}u_{a}(x,t)\mathlarger{\nabla}_{b}(\mathsf{RE}(x,t)-\mathsf{RE})^{2\beta}\mathlarger{\nabla}^{b}u_{a}(x,t)\sum_{I=1}^{\infty}\mathrm{Z}_{I}
f_{I}(x)f_{I}(x)d\mathcal{V}(x)ds}\nonumber\\&
+\lim_{u_{a}(x,t)\rightarrow u_{a}}\sup~\nu {A}^{2}\int_{\mathfrak{G}}\int_{0}^{T}u_{a}(x,t)u_{a}(x,t)\nonumber\\&\otimes\mathlarger{\nabla}_{b}(\mathsf{RE}(x,t)-\mathsf{RE}_{*})^{\beta}\mathlarger{\nabla}^{b}
(\mathsf{RE}(x,t)-\mathsf{RE}_{*})^{\beta}\sum_{I=1}^{\infty}\mathrm{Z}_{I}f_{I}(x)f_{I}(x) d\mathcal{V}(x)ds\nonumber\\&
+\lim_{u_{a}(x,t)\rightarrow u_{a}}\sup~\nu {A}^{2}\int_{\mathfrak{G}}\int_{0}^{T}u_{a}(x,t)u_{a}(x,t)(\mathsf{RE}(x,t)-\mathsf{RE}_{*})^{\beta}\nonumber\\&\otimes
\mathlarger{\nabla}_{b}(\mathsf{RE}(x,t)-\mathsf{RE}_{*})^{2\beta}\sum_{I=1}^{\infty}\mathrm{Z}_{I}f_{I}(x)\mathlarger{\nabla}^{b}f_{I}(x)
d\mathcal{V}(x)ds\nonumber\\&+\lim_{u_{a}(x,t)\rightarrow u_{a}}{\sup~\nu {A}^{2}\int_{\mathfrak{G}}\int_{0}^{T}u_{a}(x,t)\mathlarger{\nabla}^{b}u_{a}(x,t)
(\mathsf{RE}(x,t)-\mathsf{RE}_{*})^{2\beta}\sum_{I=1}^{\infty}\mathrm{Z}_{I}\mathlarger{\nabla}_{b}f_{I}(x)f_{I}(x)
d\mathcal{V}(x)ds}\nonumber\\&\underbrace{\lim_{u_{a}(x,t)\rightarrow u_{a}}{\sup~\nu {A}^{2}\int_{\mathfrak{G}}\int_{0}^{T}u_{a}(x,t)u_{a}(x,t)(\mathsf{RE}(x,t)-\mathsf{RE}_{*})^{2\beta}\sum_{I=1}^{\infty}\mathrm{Z}_{I}\mathlarger{\nabla}_{b}
f_{I}(x)f_{I}(x)d\mathcal{V}(x)ds}}\nonumber\\&+\underbrace{\lim_{u_{a}(x,t)\rightarrow u_{a}}{\sup~\nu {A}^{2}\int_{\mathfrak{G}}\int_{0}^{T}u_{a}(x,t)(\mathsf{RE}(x,t)-\mathsf{RE}_{*})^{2\beta}u_{a}(x,t)\sum_{I=1}^{\infty}\mathrm{Z}_{I}
\mathlarger{\nabla}_{b}f_{I}(x)\mathlarger{\nabla}^{b}f_{I}(x)d\mathcal{V}(x)ds}}
\end{align}
Then
\begin{align}
&\lim_{u_{a}(x,t)\rightarrow u_{a}}\sup~\nu{\bm{{{\bm{\mathbf{E}}}}}}\left[\int_{\mathfrak{G}}\int_{0}^{T}|\mathlarger{\nabla}_{b}{\mathscr{U}}_{a}(x,t)|^{2}\right]d\mathcal{V}(x)ds\nonumber\\&
=\sup~\nu\int_{\mathfrak{G}}\int_{0}^{T}
{\bm{{{\bm{\mathbf{E}}}}}}[\mathlarger{\nabla}_{b}{\mathscr{U}}_{a}(x,t)\otimes \mathlarger{\nabla}^{b}{\mathscr{U}}^{i}(x,t)]d\mathcal{V}(x)ds\nonumber\\&
=\sup~\nu {A}^{2}\int_{\mathfrak{G}}\int_{0}^{T}u_{a}u_{a}(\mathsf{RE}-\mathsf{RE}_{*})^{2\beta}\sum_{I=1}^{\infty}\mathrm{Z}_{I}\mathlarger{\nabla}_{b}
f_{I}(x)f_{I}(x)d\mathcal{V}(x)ds\nonumber\\&+\sup~\nu {A}^{2}\int_{\mathfrak{G}}\int_{0}^{T}u_{a}(\mathsf{RE}-\mathsf{RE}_{*})^{2\beta}({u}^{a}
\sum_{I=1}^{\infty}\mathrm{Z}_{I}\mathlarger{\nabla}_{b}f_{I}(x)\mathlarger{\nabla}^{b}f_{I}(x)d\mathcal{V}(x)ds
\end{align}
or equivalently
\begin{align}
&\lim_{u_{a}(x,t)\rightarrow u_{a}}\sup~\nu{\bm{{{\bm{\mathbf{E}}}}}}\left[\int_{\mathfrak{G}}\int_{0}^{T}|\mathlarger{\nabla}_{b}{\mathscr{U}}_{a}(x,t)|^{2}\right]d\mathcal{V}(x)ds\nonumber\\&
=\sup~\nu\int_{\mathfrak{G}}\int_{0}^{T}
{\bm{{{\bm{\mathbf{E}}}}}}[\mathlarger{\nabla}_{b}{\mathscr{U}}_{a}(x,t)\otimes \mathlarger{\nabla}^{b}u_{a}(x,t)]d\mathcal{V}(x)ds\nonumber\\&
=\sup~\nu {A}^{2}u_{a}u_{a}(\mathsf{RE}-\mathsf{RE}_{*})^{2\beta}\sum_{I=1}^{\infty}\mathrm{Z}_{I}\int_{\mathfrak{G}}\int_{0}^{T}\mathlarger{\nabla}_{b}
f_{I}(x)f_{I}(x)d\mathcal{V}(x)ds\nonumber\\&+\sup~\nu {A}^{2}(\mathsf{RE}-\mathsf{RE}_{*})^{2\beta}u_{a}u_{a}\sum_{I=1}^{\infty}
\mathrm{Z}_{I}\int_{\mathfrak{G}}\int_{0}^{T}
\mathlarger{\nabla}_{b}f_{I}(x)\mathlarger{\nabla}^{b}f_{I}(x)d\mathcal{V}(x)ds
\end{align}
However
\begin{align}
\int_{\mathfrak{G}}f_{I}(x)\mathlarger{\nabla}_{b}f_{I}(x)d\mathcal{V}(x)=~0
\end{align}
so that only one term remains
\begin{align}
&\lim_{u_{a}(x,t)\rightarrow u_{a}}\sup~\nu{\bm{{{\bm{\mathbf{E}}}}}}\left[\int_{\mathfrak{G}}\int_{0}^{T}|\mathlarger{\nabla}_{b}{\mathscr{U}}_{a}(x,t)|^{2}\right]d\mathcal{V}(x)ds\nonumber\\&
=\sup~\nu\int_{\mathfrak{G}}\int_{0}^{T}
{\bm{{{\bm{\mathbf{E}}}}}}\left[\mathlarger{\nabla}_{b}{\mathscr{U}}_{a}(x,t)\otimes \mathlarger{\nabla}^{b}{\mathscr{U}}^{i}(x,t)\right]d\mathcal{V}(x)ds\nonumber\\&=\sup~\nu {A}^{2}
\mathsf{RE}-\mathsf{RE}_{*})^{2\beta}u_{a}
\sum_{I=1}^{\infty}\mathrm{Z}_{I}\int_{\mathfrak{G}}\int_{0}^{T}
\mathlarger{\nabla}_{b}f_{I}(x)\mathlarger{\nabla}^{b}f_{I}(x)d\mathcal{V}(x)ds\nonumber\\&
=\sup~\nu {A}^{2}(\mathsf{RE}-\mathsf{RE}_{*})^{2\beta}u_{a}u_{a}T\sum_{I=1}^{\infty}\mathrm{Z}_{I}\int_{\mathfrak{G}}
\mathlarger{\nabla}_{b}f_{I}(x)\mathlarger{\nabla}^{b}f_{I}(x)d\mathcal{V}(x)\nonumber\\&
=\sup~\nu {A}^{2}\left(\frac{\|u_{a}\|L}{\nu}-\mathsf{RE}_{*}\right)^{2\beta} u_{a}u_{a}T\sum_{I=1}^{\infty}\mathrm{Z}_{I}\int_{\mathfrak{G}}
\mathlarger{\nabla}_{b}f_{I}(x)\mathlarger{\nabla}^{b}f_{I}(x)d\mathcal{V}(x)
\end{align}
since the Reynolds number ${\mathsf{RE}}\gg{\mathsf{RE}}$ is now constant for fixed viscosity. However, in the limit that the viscosity is reduced to zero
\begin{align}
&\lim_{\nu\rightarrow0}\sup~\nu{\bm{{{\bm{\mathbf{E}}}}}}\left[\int_{\mathfrak{G}}\int_{0}^{T}|\mathlarger{\nabla}_{b}\mathscr{U}_{a}(x,t)|^{2}\right]d\mathcal{V}(x)ds\nonumber\\&
=\sup~\nu\int_{\mathfrak{G}}\int_{0}^{T}{\bm{{{\bm{\mathbf{E}}}}}}\left[\mathlarger{\nabla}_{b}{\mathscr{U}}_{a}(x,t)\otimes \mathlarger{\nabla}^{b}u_{a}(x,t)\right]d\mathcal{V}(x)ds\nonumber\\&
=\sup {A}^{2}\left\lbrace\lim_{\nu\rightarrow 0}\nu\left(\frac{\|u_{a}\|L}{\nu}-\mathsf{RE}_{*}\right)^{2\beta}\right\rbrace
u_{a}u_{a}T\sum_{I=1}^{\infty}\mathrm{Z}_{I}\int_{\mathfrak{G}}\mathlarger{\nabla}_{b}f_{I}(x)\mathlarger{\nabla}^{b}f_{I}(x)d\mathcal{V}(x)
\end{align}
From Lemma (5.1) if $\beta\in(0,\tfrac{1}{2})$ then
\begin{align}
\sup {A}^{2}\left\lbrace\lim_{\nu\rightarrow 0}\nu\left(\frac{\|u_{a}\|L}{\nu}-\mathsf{RE}_{*}\right)^{2\beta}\right\rbrace
u_{a}u_{a}T\sum_{I=1}^{\infty}\mathrm{Z}_{I}\int_{\mathfrak{G}}\mathlarger{\nabla}_{b}f_{I}(x)\mathlarger{\nabla}^{b}f_{I}(x)d\mathcal{V}(x)=0
\end{align}
However, if $\beta=1/2$ then
\begin{align}
&\sup {A}^{2}\left\lbrace\lim_{\nu\rightarrow 0}\nu\left(\frac{\|u_{a}\|L}{\nu}-\mathsf{RE}_{*}\right)^{2\beta}\right\rbrace
u_{a}u_{a}T\sum_{I=1}^{\infty}\mathrm{Z}_{I}\int_{\mathfrak{G}}
\mathlarger{\nabla}_{b}f_{I}(x)\mathlarger{\nabla}^{b}f_{I}(x)d\mathcal{V}(x)\nonumber\\&
=\sup {A}^{2}\|u_{a}\|L u_{a}u_{a}T\sum_{I=1}^{\infty}\mathrm{Z}_{I}\int_{\mathfrak{G}}
\mathlarger{\nabla}_{b}f_{I}(x)\mathlarger{\nabla}^{b}f_{I}(x)d\mathcal{V}(x)\nonumber\\&
\equiv \sup {A}^{2}\|u_{a}\|L\|u_{a}\|^{2}T\sum_{I=1}^{\infty}\mathrm{Z}_{I}\int_{\mathfrak{G}}
\mathlarger{\nabla}_{b}f_{I}(x)\mathlarger{\nabla}^{b}f_{I}(x)d\mathcal{V}(x)\nonumber\\&
\equiv \sup {A}^{2}\|u_{a}\|^{3}L\|T\sum_{I=1}^{\infty}\mathrm{Z}_{I}\int_{\mathfrak{G}}
\mathlarger{\nabla}_{b}f_{I}(x)\mathlarger{\nabla}^{b}f_{I}(x)d\mathcal{V}(x)>0
\end{align}
since $\int_{\mathfrak{G}}\mathlarger{\nabla}_{b}f_{I}(x)\mathlarger{\nabla}^{b}f_{I}(x)d\mathcal{V}(x)>0$ and all $\mathrm{Z}_{I}>0$. Hence, the positivity of the integral is proved.
\end{proof}
\begin{proof}\textbf{Simple version.}\newline
A much simpler proof exploits the incompressibility of the fluid from the onset but gives the same result. The square of the derivative of the random field is
using Lemma (5.2) with $\mathlarger{\nabla}_{a}u_{a}(x,t)=0$
\begin{align}
&|\mathlarger{\nabla}^{a}\mathscr{U}_{a}(x,t)|^{2}=A^{2}|u_{a}(x,t)|^{2}\left(\frac{\|u_{a}(x,t)\|L}{\nu}-\mathsf{RE}_{*}\right)^{2\beta}\sum_{I=1}^{\infty}\mathrm{Z}_{I}
\mathlarger{\nabla}^{a}f_{I}(x)\mathlarger{\nabla}_{a}f_{I}(x)\left[{{\mathscr{Z}}}(x)\otimes{{\mathscr{Z}}}(x)\right]\nonumber\\&
\end{align}
with expectation
\begin{align}
{{\bm{\mathbf{E}}}}[|\mathlarger{\nabla}^{a}\mathscr{U}_{a}(x,t)|^{2}]&=A^{2}|u_{a}(x,t)|^{2}\left(\frac{\|u_{a}(x,t)\|L}{\nu}-\mathsf{RE}_{*}\right)^{2\beta}\sum_{I=1}^{\infty}\mathrm{Z}_{I}
\mathlarger{\nabla}^{a}f_{I}(x)\mathlarger{\nabla}_{a}f_{I}(x){{\bm{\mathbf{E}}}}\left[{{\mathscr{Z}}}(x)\otimes{{\mathscr{Z}}}(x)\right]\nonumber\\&
=A^{2}|u_{a}(x,t)|^{2}\left(\frac{\|u_{a}(x,t)\|L}{\nu}-\mathsf{RE}_{*}\right)^{\beta}\sum_{I=1}^{\infty}\mathrm{Z}_{I}
\mathlarger{\nabla}^{a}f_{I}(x)\mathlarger{\nabla}_{a}f_{I}(x)
\end{align}
Then
\begin{align}
&\nu\int_{\mathfrak{G}}\int_{0}^{T}{{\bm{\mathbf{E}}}}[|\mathlarger{\nabla}^{a}\mathscr{U}_{a}(x,t)|^{2}]d\mathcal{V}(x)ds\nonumber\\&
=\nu\int_{\mathfrak{G}}\int_{0}^{T}A^{2}|u_{a}(x,t)|^{2}\left(\frac{\|u_{a}(x,t)\|L}{\nu}-\mathsf{RE}_{*}\right)^{2\beta}\sum_{I=1}^{\infty}\mathrm{Z}_{I}
\mathlarger{\nabla}^{a}f_{I}(x)\mathlarger{\nabla}_{a}f_{I}(x)d\mathcal{V}(x)ds\nonumber\\&
=\frac{\nu}{\nu^{2\beta}}\int_{\mathfrak{G}}\int_{0}^{T}A^{2}|u_{a}(x,t)|^{2}\left({\|u_{a}(x,t)\|L}-\nu^{2\beta}\mathsf{RE}_{*}\right)^{2\beta}\sum_{I=1}^{\infty}\mathrm{Z}_{I}
\mathlarger{\nabla}^{a}f_{I}(x)\mathlarger{\nabla}_{a}f_{I}(x)d\mathcal{V}(x)ds
\end{align}
In the inviscid limit
\begin{align}
&\lim_{\nu\rightarrow 0}\nu\int_{\mathfrak{G}}\int_{0}^{T}{{\bm{\mathbf{E}}}}[|\mathlarger{\nabla}^{a}\mathscr{U}_{a}(x,t)|^{2}]d\mathcal{V}(x)ds\nonumber\\&
=\lim_{\nu\rightarrow 0}\frac{\nu}{\nu^{2\beta}}\int_{\mathfrak{G}}\int_{0}^{T}A^{2}|u_{a}(x,t)|^{2}\left({\|u_{a}(x,t)\|L}-\nu^{2\beta}\mathsf{RE}_{*}\right)^{2\beta}\sum_{I=1}^{\infty}\mathrm{Z}_{I}
\mathlarger{\nabla}^{a}f_{I}(x)\mathlarger{\nabla}_{a}f_{I}(x)d\mathcal{V}(x)ds
\end{align}
If $\beta=\tfrac{1}{2}$ then
\begin{align}
&\lim_{\nu\rightarrow \infty}\nu\int_{\mathfrak{G}}\int_{0}^{T}{{\bm{\mathbf{E}}}}[|\mathlarger{\nabla}^{a}\mathscr{U}_{a}(x,t)|^{2}]d\mathcal{V}(x)ds\nonumber\\&
=\lim_{\nu\rightarrow \infty}\frac{\nu}{\nu}\int_{\mathfrak{G}}\int_{0}^{T}A^{2}|u_{a}(x,t)|^{2}\left(\|u_{a}(x,t)\|L-\nu\mathsf{RE}_{*}\right)\sum_{I=1}^{\infty}\mathrm{Z}_{I}
\mathlarger{\nabla}^{a}f_{I}(x)\mathlarger{\nabla}_{a}f_{I}(x)d\mathcal{V}(x)ds\nonumber\\&
=\int_{\mathfrak{G}}\int_{0}^{T}A^{2}|u_{a}(x,t)|^{2}\left(\|u_{a}(x,t)\|L\right)\sum_{I=1}^{\infty}\mathrm{Z}_{I}
\mathlarger{\nabla}^{a}f_{I}(x)\mathlarger{\nabla}_{a}f_{I}(x)d\mathcal{V}(x)ds\nonumber\\&
=\int_{\mathfrak{G}}\int_{0}^{T}A^{2}\left(\|u_{a}(x,t)\|^{3}L\right)\sum_{I=1}^{\infty}\mathrm{Z}_{I}
\mathlarger{\nabla}^{a}f_{I}(x)\mathlarger{\nabla}_{a}f_{I}(x)d\mathcal{V}(x)ds
\end{align}
Now if $u_{a}(x,t)\rightarrow u_{a}$ then
\begin{align}
&\lim_{\nu\rightarrow \infty}\nu\int_{\mathfrak{G}}\int_{0}^{T}{{\bm{\mathbf{E}}}}[|\mathlarger{\nabla}^{a}\mathscr{U}_{a}(x,t)|^{2}]d\mathcal{V}(x)ds\nonumber\\&
=A^{2}\left(\|u_{a}\|^{3}L\right)\int_{O}^{T}\sum_{I=1}^{\infty}\mathrm{Z}_{I}
\int_{\mathfrak{G}}\mathlarger{\nabla}^{a}f_{I}(x)\mathlarger{\nabla}_{a}f_{I}(x)d\mathcal{V}(x)ds\nonumber\\&
=A^{2}\left(\|u_{a}\|^{3}L\right)T\sum_{I=1}^{\infty}\mathrm{Z}_{I}
\int_{\mathfrak{G}}\mathlarger{\nabla}^{a}f_{I}(x)\mathlarger{\nabla}_{a}f_{I}(x)d\mathcal{V}(x)
\end{align}
which agrees exactly with (5.45). Hence
\begin{align}
&\lim_{\nu\rightarrow \infty}\nu\int_{\mathfrak{G}}\int_{0}^{T}{{\bm{\mathbf{E}}}}[|\mathlarger{\nabla}^{a}\mathscr{U}_{a}(x,t)|^{2}]d\mathcal{V}(x)ds >0
\end{align}
since
\begin{align}
\sum_{I=1}^{\infty}\int_{\mathfrak{G}}\mathlarger{\nabla}^{a}f_{I}(x)\mathlarger{\nabla}_{a}f_{I}(x)d\mathcal{V}(x)>0
\end{align}
\end{proof}
\subsection{Anomalous dissipation for constant underlying fluid flow that is randomly perturbed}
One can prove anomalous dissipation by assuming from the beginning the stochastic or turbulent fluid flow
has a constant underlying velocity $u_{a}(x)=u_{a}(x)=u_{a}$ for all $x\in\mathfrak{G}$. As another consistency check this should also produce the same answer and give (5.39). The underlying constant smooth flow $u_{a}$ is randomly perturbed to produce the turbulent flow.
\begin{thm}
Let ${\mathscr{U}}_{a}(x,t)$ be a turbulent flow within a domain $\mathfrak{G}$ with constant underlying velocity $u_{a}$ and therefore constant Reynolds number such that
\begin{align}
{\mathscr{U}}_{a}(x,t)=u_{a}+ {A} u_{a}\bigg(\mathsf{RE}-\mathsf{RE}_{*}\bigg)^{\beta}\sum_{I=1}^{\infty}
{\mathrm{Z}_{I}}^{1/2}f_{I}(x)\otimes{\mathscr{Z}}_{I}
\end{align}
with $\beta\le \tfrac{1}{2}$. The volume of the domain is $vol(\mathfrak{G})\sim L^{3}$ and the fluid has viscosity $\nu$ so $\mathsf{RE}=\|u_{a}\|L/\nu$, with $\mathrm{u}_{a}$ and $L$ fixed and $\mathsf{RE}\gg \mathsf{RE}_{*}$. The energy dissipation rate for the constant unperturbed flow is $\epsilon$. Then the random flow $\mathscr{U}_{a}(x,t)$ satisfies an anomalous dissipation law in that
\begin{align}
\lim_{\nu\rightarrow 0}\nu\mathbf{E}\left[\left |\mathlarger{\nabla}_{b}{\mathscr{U}}_{a}(x)\right|^{2}\right]>0
\end{align}
or
\begin{align}
&\lim_{\nu\rightarrow 0}\nu\int_{\mathfrak{G}}\mathbf{E}\left[\left|\mathlarger{\nabla}_{b}{\mathscr{U}}_{a}(x)\right|^{2}\right]d\mathcal{V}(x)
= {A}^{2}\|u_{a}\|^{3}L\sum_{I=1}^{\infty}
\mathrm{Z}_{I}\int_{\mathfrak{G}}\mathlarger{\nabla}_{b}f_{I}(x)\mathlarger{\nabla}_{b}f_{I}(x)d\mathcal{V}(x)>0
\end{align}
or
\begin{align}
&\lim_{\nu\rightarrow 0}\nu\int_{0}^{T}\int_{\mathfrak{G}}\bm{{{\bm{\mathbf{E}}}}}\left[\left|\mathlarger{\nabla}_{b}{\mathscr{U}}_{a}(x)\right|^{2}\right]d\mathcal{V}(x)ds
= {A}^{2}\|u_{a}\|^{3}L\sum_{I=1}^{\infty}
\mathrm{Z}_{I}\int_{\mathfrak{G}}\mathlarger{\nabla}_{b}f_{I}(x)\mathlarger{\nabla}_{b}f_{I}(x)d\mathcal{V}(x)ds>0
\end{align}
\end{thm}
iff $\beta=\tfrac{1}{2}$ and $\sum_{I=1}^{\infty}\mathrm{Z}_{I}\int_{\mathfrak{G}}\mathlarger{\nabla}_{b}f_{I}(x)\mathlarger{\nabla}_{b}f_{I}(x)d\mathcal{V}(x)>0$
and $\sum_{I=1}^{\infty}\mathrm{Z}_{I}\int_{\mathfrak{G}}\mathlarger{\nabla}_{b}f_{I}(x)\mathlarger{\nabla}_{b}f_{I}(x)d\mathcal{V}(x)<\infty$.
 \begin{proof}
The derivative of the random field is
\begin{align}
&\mathlarger{\nabla}_{b}u_{a}(x)={A}u_{a}\bigg(\mathsf{RE}-\mathsf{RE}_{*}\bigg)^{\beta}\sum_{I=1}^{\infty}
\mathrm{Z}^{1/2}_{I}\mathlarger{\nabla}_{b}u_{a}(x)\otimes{\mathscr{Z}}_{I}
\end{align}
Then
\begin{align}
&\left|\mathlarger{\nabla}_{b}{\mathscr{U}}_{a}(x)\right|^{2}=\left|{A}u_{a}\bigg(\mathsf{RE}-\mathsf{RE}_{*}\bigg)^{\beta}\sum_{I=1}^{\infty}
\mathrm{Z}^{1/2}_{I}\mathlarger{\nabla}_{b}f_{I}(x)\otimes{\mathscr{Z}}_{I}\right|^{2}\nonumber\\&
={A}^{2}\|u_{a}\|^{2}\bigg(\mathsf{RE}-\mathsf{RE}_{*}\bigg)^{2\beta}\sum_{I=1}^{\infty}\sum_{I=1}^{\infty}
\mathrm{Z}^{1/2}_{I}\mathrm{Z}^{1/2}_{I}\mathlarger{\nabla}_{b}f_{I}(x)\mathlarger{\nabla}_{b}f_{I}(x)\bigg({\mathscr{Z}}_{I}\otimes{\mathscr{Z}}_{I}\bigg)\nonumber\\&
={A}^{2}\|u_{a}\|^{2}\bigg(\mathsf{RE}-\mathsf{RE}_{*}\bigg)^{2\beta}\sum_{I=1}^{\infty}
\mathrm{Z}_{I}\mathlarger{\nabla}_{b}f_{I}(x)\mathlarger{\nabla}_{b}f_{I}(x)\bigg({\mathscr{Z}}_{I}\otimes{\mathscr{Z}}_{I}\bigg)
\end{align}
so that
\begin{align}
&\lim_{\nu\rightarrow 0}\nu\int_{\mathfrak{G}}{\bm{{{\bm{\mathbf{E}}}}}}\left[\left|\mathlarger{\nabla}_{b}{\mathscr{U}}_{a}(x)\right|^{2}\right]d\mathcal{V}(x)\nonumber\\&
= {A}^{2}\|u_{a}\|^{2}\lim_{\nu\rightarrow 0}\left\lbrace\nu\bigg(\mathsf{RE}-\mathsf{RE}_{*}\bigg)^{2\beta}\right\rbrace\sum_{I=1}^{\infty}
\mathrm{Z}_{I}\int_{\mathfrak{G}}\mathlarger{\nabla}_{b}f_{I}(x)\mathlarger{\nabla}_{b}f_{I}(x)
\mathbf{E}\left[{\mathscr{Z}}_{I}\otimes{\mathscr{Z}}_{I}\right]d\mathcal{V}(x)\nonumber\\&
= {A}^{2}\|u_{a}\|^{2}\lim_{\nu\rightarrow 0}\left\lbrace\nu\bigg(\frac{\|u_{a}\|L}{\nu}-\mathsf{RE}_{*}\bigg)^{2\beta}\right\rbrace\sum_{I=1}^{\infty}
\mathrm{Z}_{I}\int_{\mathfrak{G}}\mathlarger{\nabla}_{b}f_{I}(x)\mathlarger{\nabla}_{b}f_{I}(x)d\mathcal{V}(x)\nonumber\\&
= {A}^{2}\|u_{a}\|^{2}\|u_{a}\|L\sum_{I=1}^{\infty}\mathrm{Z}_{I}\int_{\mathfrak{G}}
\mathlarger{\nabla}_{b}f_{I}(x)\mathlarger{\nabla}_{b}f_{I}(x)d\mathcal{V}(x)\nonumber\\&
= {A}^{2}\|u_{a}\|^{3}\|L\sum_{I=1}^{\infty}\mathrm{Z}_{I}\int_{\mathfrak{G}}
\mathlarger{\nabla}_{b}f_{I}(x)\mathlarger{\nabla}_{b}f_{I}(x)d\mathcal{V}(x)
\end{align}
which is again 5.51.
\end{proof}
\section{Stochastically averaged Navier-Stokes equation}
Finally, in this section, we considers the stochastically averaged N-S equations which the random or turbulent flow ${\mathscr{U}_{a}(x,t)}$ satisfies.
First we give the following preliminary lemmas which are required in the proof.
\begin{lem}
For the Gaussian random field $\mathscr{T}$
\begin{align}
&{\bm{{{\bm{\mathbf{E}}}}}}[\mathscr{T}(x)\otimes\mathlarger{\nabla}_{b}\mathscr{T}(x)]=2\mathlarger{\nabla}_{b}{\bm{{{\bm{\mathbf{E}}}}}}[\mathscr{T}(x)\otimes\mathscr{T}(x)]=2\lim_{y\rightarrow x}\mathlarger{\nabla}_{b}^{(x)}K(x,y;\lambda)
\end{align}
and for a Gaussian kernel this vanishes so that $\lim_{y\rightarrow x}\mathlarger{\nabla}_{b}^{(x)}K(x,y;\xi)=0$.
\end{lem}
\begin{proof}
For the Gaussian kernel, the 1st derivative is
\begin{align}
\lim_{y\rightarrow x}\mathlarger{\nabla}_{b}^{(x)}K(x-y;\lambda)=\lim_{y\rightarrow x} \frac{6(x_{b}-y^{b})}{\lambda^{2}}\exp\left(-\frac{(x_{a}-y^{b})(x_{a}-y^{b})}{\lambda^{2}}\right)=0
\end{align}
where $\mathlarger{\nabla}_{b}^{(x)}(x_{b}-y^{b})=3$ on $\mathbb{R}^{3}$.
\end{proof}
\begin{lem}
For the Gaussian random field $\mathscr{T}$
\begin{align}
&\mathbf{E}\left[\mathscr{T}(x)\otimes\mathlarger{\nabla}_{b}\mathscr{T}(x)\right]=
\sum_{I=1}^{\infty}\mathrm{Z}_{I}f_{I}(x)\mathlarger{\nabla}_{b}f_{I}(x){\bm{{{\bm{\mathbf{E}}}}}}\left[{\mathscr{Z}}_{I}\otimes{\mathscr{Z}}_{I}\right]
=\sum_{I=1}^{\infty}\mathrm{Z}_{I}f_{I}(x)
\mathlarger{\nabla}_{b}f_{I}(x)
\end{align}
and this is zero for a Gaussian random field with a Gaussian kernel.
\end{lem}
\begin{thm}
Let $\mathfrak{G}\in\mathbb{R}^{3}$ contain an incompressible fluid of viscosity $\nu$ which has a smooth laminar flow of constant velocity $u_{a}=u_{a}(x,0)=v^{(x)}$ for all $x\in\mathfrak{G}$. Then $u_{a}$ trivially satisfies the Navier-Stokes equations with vanishing pressure gradient $\mathlarger{\nabla}_{b}p=0$ so that
$\frac{\partial}{\partial t}u_{a}+U^{b}\mathlarger{\nabla}_{b}u_{a}-\nu \mathlarger{\Delta} u_{a}=0$. The Reynolds number throughout $\mathfrak{G}$ is constant in that $\mathsf{RE}=\|u_{a}\|L/\nu$, and the Reynolds number is taken to be high in that $\mathsf{RE}\gg\mathsf{RE}_{*}$. Let ${\mathscr{U}_{a}(x,)}$ be the randomly perturbed or turbulent flow, obtained by 'mixing' the smooth flow with the Gaussian random field. As before
\begin{align}
&{\mathscr{U}}_{a}(x,t)=u_{a}+{A}u_{a}\bigg(\mathsf{RE}-\mathsf{RE}_{*}\bigg)^{\beta}\otimes{\mathscr{T}}_{I}(x)\nonumber\\&
=u_{a}+{A}u_{a}\bigg(\mathsf{RE}-\mathsf{RE}_{*}\bigg)^{\beta}
\sum_{I=1}^{\infty}\mathrm{Z}^{1/2}_{I}f_{I}(x)\otimes{\mathscr{Z}}_{I}
\end{align}
with $\beta=1/2$. Then this turbulent flow satisfies the stochastically averaged PDE
\begin{align}
&\mathbf{E}\left[\frac{\partial}{\partial t}\mathscr{U}_{a}(x)+\mathscr{U}^{b}(x)\mathlarger{\nabla}_{b}\mathscr{U}_{a}(x)-\nu \mathlarger{\Delta} \mathscr{U}_{a}(x)\right]\nonumber\\&=A^{2}u^{b}u_{a}\bigg(\mathsf{RE}-\mathsf{RE}_{*}\bigg)^{\beta}\sum_{I=1}^{\infty}\mathrm{Z}^{1/2}_{I}
f_{I}(x)\mathlarger{\nabla}_{b}f_{I}(x)\nonumber\\&
=2{A}^{2}u^{b}u_{a}\bigg(\mathsf{RE}-\mathsf{RE}_{*}\bigg)^{\beta}\lim_{y\rightarrow x}\mathlarger{\nabla}_{b}^{(x)}K(x,y;\lambda)
\end{align}
which is zero for a Gaussian kernel.
\end{thm}
\begin{proof}
The first and second derivatives are
\begin{align}
&\frac{\partial}{\partial t}{\bm{\mathscr{U}}}_{a}(x,t)=0\\&
\mathlarger{\nabla}_{b}{\bm{\mathscr{U}}}_{a}(x,t)={A}u_{a}\bigg(\mathsf{RE}-\mathsf{RE}_{*}\bigg)^{\beta}\sum_{I=1}^{\infty}\mathrm{Z}^{1/2}_{I}\mathlarger{\nabla}_{b}f_{I}(x)\otimes{\mathscr{Z}}_{I}\\&
\mathlarger{\Delta}{{\mathscr{U}}}_{I}(x,t)=u_{a}\bigg(\mathsf{RE}-\mathsf{RE}_{*}\bigg)^{\beta}\sum_{I=1}^{\infty}
{\mathrm{Z}^{1/2}_{I}}\mathlarger{\Delta} f_{I}(x)\otimes{\mathscr{Z}}_{I}
\end{align}
and the nonlinear convective term is
\begin{align}
&{\mathscr{U}^{b}(x,t)}\mathlarger{\nabla}_{b}{\mathscr{U}}_{a}(x,t)\nonumber\\&
=A^{2}u^{b}u_{a}\bigg(\mathsf{RE}-\mathsf{RE}_{*}\bigg)^{2\beta}\sum_{I=1}^{\infty}
\mathrm{Z}^{1/2}_{I}f_{I}(x)\otimes{\mathscr{Z}}_{I}\nonumber\\&+{A}^{2}u^{b}u_{a}\bigg(\mathsf{RE}-\mathsf{RE}_{*}\bigg)^{2\beta}
\sum_{I=1}^{\infty}\mathrm{Z}_{I}f_{I}(x)\mathlarger{\nabla}_{b}
f_{I}(x)\otimes{\mathscr{Z}}_{I}\otimes{\mathscr{Z}}_{I}
\end{align}
Then
\begin{align}
&\frac{\partial}{\partial t}{\mathscr{U}_{a}(x,t)}+{\mathscr{U}^{b}(x,t)}\mathlarger{\nabla}_{b}{\mathscr{U}_{a}(x,t)}-\nu \mathlarger{\Delta}
{\mathscr{U}_{a}(x,t)}\nonumber\\&={A}^{2}u^{b}u_{a}\bigg(\mathsf{RE}-\mathsf{RE}_{*}\bigg)^{\beta}\sum_{I=1}^{\infty}
\mathrm{Z}^{1/2}_{I} f_{I}(x)\otimes{\mathscr{Z}}_{I}\nonumber\\&+{A}^{2}u^{b}u_{a}\bigg(\mathsf{RE}-\mathsf{RE}_{*}\bigg)^{2\beta}\sum_{I=1}^{\infty}\mathrm{Z}_{I}
f_{I}(x)\mathlarger{\nabla}_{b}f_{I}(x)\otimes{\mathscr{Z}}_{I}\otimes{\mathscr{Z}}_{I}\nonumber\\&
+\nu {A}u_{a}\bigg(\mathsf{RE}-\mathsf{RE}_{*}\bigg)^{\beta}\sum_{I=1}^{\infty}{\mathrm{Z}^{1/2}_{I}}\mathlarger{\Delta} f_{I}(x)\otimes{\mathscr{Z}}_{I}
\end{align}
Taking the stochastic expectation
\begin{align}
&\mathbf{E}\bigg[\frac{\partial}{\partial t}{\mathscr{U}_{a}(x)}+{\mathscr{U}^{b}(x)}\mathlarger{\nabla}_{b}{\mathscr{U}_{a}(x)}-\nu \mathlarger{\Delta} {\mathscr{U}_{a}(x)}\bigg]\nonumber\\&
={A}^{2}u^{b}u_{a}\bigg(\mathsf{RE}-\mathsf{RE}_{*}\bigg)^{\beta}\sum_{I=1}^{\infty}
\mathrm{Z}^{1/2}_{I}f_{I}(x)\mathbf{E}\bigg[{\mathscr{Z}}_{I}\bigg]\nonumber\\&+{A}^{2}u^{b}u_{a}\bigg(\mathsf{RE}-\mathsf{RE}_{*}
\bigg)^{2\beta}\sum_{I=1}^{\infty}\mathrm{Z}_{I}f_{I}(x)\mathlarger{\nabla}_{b}f_{I}(x){\bm{{{\bm{\mathbf{E}}}}}}\bigg[{\mathscr{Z}}_{I}\otimes{\mathscr{Z}}_{I}\bigg]\nonumber\\&
+\nu {A}u_{a}\bigg(\mathsf{RE}-\mathsf{RE}_{*}\bigg)^{\beta}\sum_{I=1}^{\infty}{\mathrm{Z}^{1/2}_{I}}\mathlarger{\Delta} f_{I}(x)
\mathbf{E}\bigg[{\mathscr{Z}}_{I}\bigg]\nonumber\\&
={A}^{2}u^{b}u_{a}\bigg(\mathsf{RE}-\mathsf{RE}_{*}\bigg)^{\beta}\sum_{I=1}^{\infty}\mathrm{Z}_{I}
f_{I}(x)\mathlarger{\nabla}_{b}f_{I}(x)\nonumber\\&={A}^{2}u^{b}u_{a}\bigg(\mathsf{RE}-\mathsf{RE}_{*}\bigg)^{\beta}\lim_{y\rightarrow x}\mathlarger{\nabla}_{a}K(x,y;\lambda)=0
\end{align}
since $\mathbf{E}[{\mathscr{Z}}_{I}]=0$ and $\mathbf{E}[{\mathscr{Z}}_{I}\otimes{\mathscr{Z}}_{I}]=1$. This is zero for the Gaussian kernel from (6.3).
\end{proof}
\clearpage
\appendix
\renewcommand{\theequation}{\Alph{section}.\arabic{equation}}
\section{}
\begin{proof}
\begin{align}
&\lim_{N\rightarrow\infty}{{{{\bm{\mathbf{E}}}}}}\left[\left\|\mathscr{T}(x)-\sum_{a=1}^{N}{\mathrm{Z}^{1/2}_{I}}f_{I}(x)\mathscr{Q}_{I}
\right\|^{2}_{L_{2}(\mathfrak{G})}\right]
=\lim_{N\rightarrow\infty}{{{{\bm{\mathbf{E}}}}}}\left[\int_{\mathfrak{G}}\left|\mathscr{T}(x)-\sum_{I=1}^{N}{\mathrm{Z}^{1/2}_{I}}f_{I}(x)\otimes{\mathscr{Z}}_{I}
\right|^{2}d\mathcal{V}(x)\right]\nonumber\\&
=\lim_{N\rightarrow\infty}\int_{\mathfrak{G}}\left[{{{\bm{\mathbf{E}}}}}\left|\mathscr{T}(x)-\sum_{I=1}^{N}{\mathrm{Z}^{1/2}_{I}}f_{I}(x)\otimes{\mathscr{Z}}_{I}
\right|^{2}d\mathcal{V}(x)\right]\nonumber\\&
=\lim_{N\rightarrow\infty}\int_{\mathfrak{G}}\mathbf{E}\left[\underbrace{\big|\mathscr{T}(x)\otimes\mathscr{T}(x)\big|-2\sum_{I=1}^{N}
\mathrm{Z}^{1/2}_{I}\mathscr{T}(x)f_{I}(x)\otimes{\mathscr{Z}}_{I}+\sum_{I=1}^{\infty}\mathrm{Z}_{I}|f_{I}(x)|^{2}
\bigg({\mathscr{Z}}_{I}\otimes{\mathscr{Z}}_{I}\bigg)}_{expanded~square}\right]d\mathcal{V}(x)\nonumber\\&
=\lim_{N\rightarrow\infty}\int_{\mathfrak{G}}{{{\bm{\mathbf{E}}}}}\left[\underbrace{\big|\mathscr{T}(x)\otimes\mathscr{T}(x)\big|
-2\sum_{I=1}^{N}{\mathrm{Z}^{1/2}_{I}}{\mathrm{Z}^{1/2}_{I}}|f_{I}(x)|^{2}\bigg({\mathscr{Z}}_{I}\otimes{\mathscr{Z}}_{I}\bigg)+
\sum_{I=1}^{\infty}\mathrm{Z}_{I}|f_{I}(x)|^{2}\bigg({\mathscr{Z}}_{I}\otimes{\mathscr{Z}}_{I}\bigg)}
_{substituting~KL~series~for~\mathscr{T}(x)}\right]d\mathcal{V}(x)\nonumber
\end{align}
\begin{align}
&=\int_{\mathfrak{G}}\bm{{{\bm{\mathbf{E}}}}}\left[\big|\mathscr{T}(x)\otimes\mathscr{T}(x)\big|-\sum_{I=1}^{N}{\mathrm{Z}^{1/2}_{I}}{\mathrm{Z}^{1/2}_{I}}|
f_{I}(x)|^{2}\bigg(\mathscr{Z}_{I}\otimes\mathscr{Z}_{I}\bigg)\right]d\mathcal{V}(x)\nonumber\\&
=\int_{\mathfrak{G}}\bm{{{\bm{\mathbf{E}}}}}\big[\big|\mathscr{Z}(x)\otimes\mathscr{Z}(x)\big|\big]d\mathcal{V}(x)
-\int_{\mathfrak{G}}\sum_{I=1}^{N}{\mathrm{Z}_{I}}|f_{I}(x)|^{2}\bm{{{\bm{\mathbf{E}}}}}\big[{\mathscr{Z}}_{I}\otimes {\mathscr{Z}}_{I}\big]d\mathcal{V}(x)\nonumber\\&
=\int_{\mathfrak{G}}\bm{{{\bm{\mathbf{E}}}}}\big[\big|\mathscr{Z}(x)\otimes\mathscr{Z}(x)\big|\big]d\mathcal{V}(x)
-\lim_{N\rightarrow\infty}\int_{\mathfrak{G}}\sum_{I=1}^{N}{\mathrm{Z}_{I}}|f_{I}(x)|^{2}d\mathcal{V}(x)\nonumber\\&
=\int_{\mathfrak{G}}{{{\bm{\mathbf{E}}}}}\big[\big|\mathscr{T}(x)\otimes\mathscr{T}(x)\big|\big]d\mathcal{V}(x)
-\lim_{N\rightarrow\infty}\int_{\mathfrak{G}}\sum_{a=1}^{N}{\mathrm{Z}_{I}}|f_{I}(x)|^{2}d\mathcal{V}(x)\nonumber\\&
=\int_{\mathfrak{G}}{{{\bm{\mathbf{E}}}}}\big[\big|\mathscr{T}(x)\otimes\mathscr{T}(x)\big|\big]d\mathcal{V}(x)
-\lim_{N\rightarrow\infty}\sum_{I=1}^{N}{\mathrm{Z}_{I}}\int_{\mathfrak{G}}|f_{I}(x)|^{2}d\mathcal{V}(x)\nonumber\\&
=\int_{\mathfrak{G}}{{{\bm{\mathbf{E}}}}}\big[\big|\mathscr{T}(x)\otimes\mathscr{T}(x)\big|\big]d\mathcal{V}(x)-\lim_{N\rightarrow\infty}\sum_{I=1}^{N}{\mathrm{Z}_{I}}\nonumber\\&
=\int_{\mathfrak{G}}{{{\bm{\mathbf{E}}}}}\big[\big|\mathscr{T}(x)\otimes\mathscr{T}(x)\big|\big]d\mathcal{V}(x)
-\sum_{I=1}^{\infty}{\mathrm{Z}_{I}}
\end{align}
If $K(x,x;\lambda)<\infty$ and $K(x,x|\lambda)=C$ then
\begin{align}
&\int_{\mathfrak{G}}{{{\bm{\mathbf{E}}}}}\big[\big|\mathscr{T}(x)\otimes\mathscr{T}(x)\big|\big]d\mathcal{V}(x)
-\sum_{I=1}^{\infty}{\mathrm{Z}_{I}}\nonumber\\&
=\int_{\mathfrak{G}}{K}(x,x;\lambda)d\mathcal{V}(x)
-\sum_{I=1}^{\infty}{\mathrm{Z}_{I}}=K(x,x;\xi)vol(\mathfrak{G})-\sum_{I=1}^{\infty}\mathrm{Z}_{I}=0
\end{align}
so that
\begin{align}
\bm{\mathbf{E}}\big[\mathscr{T}(x)\otimes\mathscr{T}(x)\big]=K(x,x;\lambda){vol}(\mathfrak{G})=\sum_{I=1}^{\infty}\mathrm{Z}_{I}
\end{align}
which is again Mercer's Theorem
\end{proof}
\section{}
\begin{proof}
The proof is given for the Gaussian kernel but the proof for the rational quadratic is very similar.
\begin{align}
H(x,y;\lambda)=C\mathlarger{\nabla}_{a}^{(x)}\mathlarger{\nabla}^{a}_{(y)}K_{G}(x,y;\lambda)=\mathlarger{\nabla}_{a}^{(x)}\mathlarger{\nabla}^{a}_{(y)}\exp\left(-\frac{|x-y|^{2}}{\lambda^{2}}\right)
\end{align}
where $\mathlarger{\nabla}^{a}_{(y)}=\tfrac{\partial}{\partial y^{a}}$ and $\mathlarger{\nabla}_{a}^{(x)}=\tfrac{\partial}{\partial x_{a}}$. Evaluating (B1) is not difficult but still somewhat tricky and care is required. First write $|x-y|^{2}=(x_{a}-y^{a}(x_{a}-y^{a})=(x_{a}-y^{a})^{2}\equiv \sum_{a=1}^{3}(x_{a}-y^{a})^{2}$. Next $\sum_{a=1}^{3}\mathlarger{\nabla}_{a}^{(x)}(x_{b}-y^{b})=\sum_{a,b}^{3}\delta_{ab}$ and $\sum_{a=1}^{3}\mathlarger{\nabla}^{a}_{(y)}(x_{b}-y^{b})=-\sum_{a,b}^{3}\delta_{ij}$ so that $\sum_{a=1}^{3}\mathlarger{\nabla}_{a}^{(x)}(x_{a}-y^{a})=\sum_{a}^{n}\delta_{aa}=3$ and $\sum_{a=1}^{3}\mathlarger{\nabla}^{a}_{(y)}(x_{a}-y^{a})=-\sum_{a}^{n}\delta_{aa}=-3$. Or without the summation $\mathlarger{\nabla}_{a}^{(x)}(x_{a}-y^{a})=\delta_{aa}=3$ and $\mathlarger{\nabla}^{a}_{(y)}(x_{a}-y^{a})=-\delta_{aa}=-3$. Then $\mathlarger{\nabla}_{a}^{(x)}|x-y|^{2}=\mathlarger{\nabla}_{a}^{(x)}((x_{a}-y^{a})(x_{a}-y^{a}))$ and $ 2(x_{a}-y^{a})\mathlarger{\nabla}_{a}^{(x)}(x_{a}-y^{a})=2(x_{a}-y_{a})*3=6(x_{a}-y^{a})$ and $\mathlarger{\nabla}^{a}_{(y)}|x-y|^{2}=-6(x_{a}-y^{a})$
The Gaussian kernel is
\begin{align}
K_{G}(x,y;\lambda)= \exp\left(-\frac{(x_{a}-y^{a})(x_{a}-y^{a})}{\lambda^{2}}\right)
\end{align}
Then
\begin{align}
K_{G}(x,y;\lambda)=\mathlarger{\nabla}_{a}^{(x)}\mathlarger{\nabla}^{a}_{(y)}{K}_{G}(x,y;\lambda)=\mathlarger{\nabla}_{a}^{(x)}\mathlarger{\nabla}^{a}_{(y)}
\exp\left(-\frac{(x_{a}-y^{a})(x_{a}-y^{a})}{\lambda^{2}}\right)
\end{align}
The first derivative with respect to y is
\begin{align}
&\mathlarger{\nabla}^{a}_{(y)}K(x,y;\lambda)=\mathlarger{\nabla}^{a}_{(y)}\exp\left(-\frac{(x_{a}-y^{a})(x_{a}-y^{a})}{\lambda^{2}}\right)\nonumber\\&
=-\frac{2(x_{a}-y^{a})\mathlarger{\nabla}^{a}_{(y)}(x_{a}-y^{a})}{\lambda^{2}}\exp\left(-\frac{(x_{a}-y^{a})(x_{a}-y^{a})}{\lambda^{2}}\right)\nonumber\\&
=\frac{6(x_{a}-y^{a})}{\lambda^{2}}\exp\left(-\frac{(x_{a}-y^{a})(x_{a}-y^{a})}{\lambda^{2}}\right)
\end{align}
Then
\begin{align}
&{{\mathscr{D}}}(x,y;\lambda)\nonumber\\&
=\mathlarger{\nabla}_{a}^{(x)}\mathlarger{\nabla}^{a}_{(y)}\exp\left(-\frac{(x_{a}-y^{a})(x_{a}-y^{a})}{\lambda^{2}}\right)=
\mathlarger{\nabla}_{a}^{(x)}\left\lbrace \frac{6(x_{a}-y^{a})}{\lambda^{2}}\Phi\exp\left(-\frac{(x_{a}-y^{a})(x_{a}-y^{a})}{\lambda^{2}}\right)
\right\rbrace\nonumber\\&
=\frac{6(x_{a}-y^{a})}{\lambda^{2}}\left\lbrace\mathlarger{\nabla}_{a}^{(x)}\exp\left(-\frac{(x_{a}-y^{a})(x_{a}-y^{a})}{\lambda^{2}}\right)\right\rbrace\nonumber\\&
+\left\lbrace \mathlarger{\nabla}^{a}_{(x)}\frac{6(x_{a}-y^{a})}{\lambda^{2}}\right\rbrace\exp\left(-\frac{(x_{a}-y^{a})(x_{a}-y^{a})}{\lambda^{2}}\right)\nonumber\\&
=\frac{6\mathcal{N}(x_{a}-y^{a})}{\lambda^{2}}\left(\frac{-2(x_{a}-y^{a})\mathlarger{\nabla}^{(x)}_{a}(x_{a}-y_{a})}{\lambda^{2}}\exp\left(-\frac{(x_{a}-y^{a})(x_{a}-y^{a})}{\lambda^{2}}\right)\right)\nonumber\\&
+\frac{18\mathcal{N}}{\lambda^{2}}\exp\left(-\frac{(x_{a}-y^{a})(x_{a}-y^{a})}{\lambda^{2}}\right)\nonumber\\&
=\frac{-12\mathcal{N}(x_{a}-y_{a})^{2}}{\lambda^{2}}\exp\left(-\frac{(x_{a}-y^{a})(x_{a}-y^{a})}{\lambda^{2}}\right)
+\frac{18 \mathcal{N}}{\lambda^{2}}\exp\left(-\frac{(x_{a}-y^{a})(x_{a}-y^{a})}{\lambda^{2}}\right)\nonumber\\&
=-\frac{2}{3}\mathrm{H}(x_{a}-y^{a})^{2}\exp\left(-\frac{(x_{a}-y^{a})(x_{a}-y^{a})}{\lambda^{2}}\right)+\mathrm{H}\exp\left(-\frac{(x_{a}-y^{a})(x_{a}-y^{a})}{\lambda^{2}}\right)\nonumber\\&
=\mathrm{H}\left(1-\frac{2}{3}(x_{a}-y^{a})^{2}\right)\exp\left(-\frac{(x_{a}-y^{a})(x_{a}-y^{a})}{\lambda^{2}}\right)
\end{align}
where $\mathcal{N}=18 C/\lambda^{2}$. Now
\begin{align}
{\mathcal{h}}(|x-y|;\lambda)=\mathcal{N}\left(1-\frac{2}{3}(x_{a}-y^{a})^{2}\right)\exp\left(-\frac{(x_{a}-y^{a})(x_{a}-y^{a})}{\lambda^{2}}\right)\le \mathcal{N}
\end{align}
so the conditions of the dominated convergence theorem are satisfied. Taking the limit
\begin{align}
\lim_{y\rightarrow x}{\mathfrak{G}}(|x-y|;\lambda)=\lim_{y\rightarrow x}\mathcal{N}\left(1-\frac{2}{3}(x_{a}-y^{a})^{2}\right)\exp\left(-\frac{(x_{a}-y^{a})(x_{a}-y^{a})}{\lambda^{2}}\right)= \mathcal{N}
\end{align}
From the dominated convergence theorem and theorem (2.26)
\begin{align}
\lim_{y\rightarrow x}\int_{\mathfrak{G}}K(|x-y|;\lambda)d\mathcal{V}(x)\equiv\int_{\mathfrak{G}}\lim_{y\rightarrow x}{J
}(|x-y|;\lambda)d\mathcal{V}(x)(x)
\end{align}
\begin{align}
&\lim_{y\rightarrow x}\int_{\mathfrak{G}}\mathlarger{\nabla}_{a}^{(x)}\mathlarger{\nabla}^{a}_{(y)}\mathcal{N}\exp\left(-\frac{(x_{a}-y^{a})(x_{a}-y^{a})}{\lambda^{2}}\right)d\mathcal{V}(x)(x)\nonumber\\&
\equiv\int_{\mathfrak{G}}\lim_{y\rightarrow x}\mathlarger{\nabla}_{a}^{(x)}\mathlarger{\nabla}^{a}_{(y)}\mathcal{N}\exp\left(-\frac{(x_{a}-y^{a})(x_{a}-y^{a})}{\lambda^{2}}\right)d\mathcal{V}(x)(x)
=\int_{\mathfrak{G}}\mathcal{N} d\mathcal{V}(x)(x)=\mathcal{N}vol({\mathfrak{G}})>0
\end{align}
Hence, (2.86) is true for this Gaussian kernel. The computation can be repeated for the rational quadratic kernel.
\end{proof}

}
\end{document}